\documentclass[11pt]{amsart}

\usepackage[pdftex,pdfpagelabels,bookmarks,hyperindex,hyperfigures]{hyperref}
  \hypersetup{
    colorlinks,
    linkcolor={red!55!black},
    urlcolor={blue!55!black},
    citecolor={green!55!black}
  }

\usepackage{amsmath}
\usepackage{amsthm}
\usepackage{amsfonts}
\usepackage{amssymb}
\usepackage{amsthm}
\usepackage{bbm}
\usepackage{enumitem}
\usepackage[mathscr]{euscript}
\usepackage{float}
\usepackage{ifthen}
\usepackage{listings}
\usepackage{mathdots}
\usepackage{mleftright}
\usepackage[new]{old-arrows}
\usepackage{rotating}
\usepackage{sansmath}
\usepackage{stmaryrd}
\usepackage{tabularx}
\usepackage{tikz}
  \usetikzlibrary{positioning}
\usepackage{tikz-cd}
  \tikzcdset{
    cells={font=\everymath\expandafter{\the\everymath\displaystyle}},
  }
  \usetikzlibrary{patterns}
\usepackage{todonotes}
\usepackage{xcolor}
\usepackage{xspace}

\usepackage[capitalize]{cleveref}

\definecolor{codegreen}{rgb}{0,0.6,0}
\definecolor{codegray}{rgb}{0.5,0.5,0.5}
\definecolor{codepurple}{rgb}{0.58,0,0.82}
\definecolor{backcolour}{rgb}{0.95,0.95,0.92}
\lstdefinestyle{mystyle}{
    xleftmargin={2.5em},
    backgroundcolor=\color{backcolour},
    commentstyle=\color{codegreen},
    keywordstyle=\color{magenta},
    numberstyle=\tiny\color{codegray},
    stringstyle=\color{codepurple},
    basicstyle=\ttfamily\footnotesize,
    breakatwhitespace=false,
    breaklines=true,
    captionpos=b,
    keepspaces=true,
    numbers=left,
    numbersep=5pt,
    showspaces=false,
    showstringspaces=false,
    showtabs=false,
    tabsize=2
}
\usepackage{subfiles}

\newtheorem{thm}{Theorem}
\numberwithin{thm}{subsection}

\newtheorem{prop}[thm]{Proposition}

\newtheorem{cor}[thm]{Corollary}

\theoremstyle{definition}

\newtheorem{conv}[thm]{Convention}
\newtheorem{exm}[thm]{Example}
\newtheorem{defn}[thm]{Definition}
\newtheorem{notn}[thm]{Notation}
\newtheorem{rmk}[thm]{Remark}
\newtheorem{warn}[thm]{Warning}
\newtheorem{qsn}[thm]{Question}

\newcommand{\bb}{\mathbb}

\newcommand{\f}{\mathfrak}
\newcommand{\s}{\mathscr}

\newcommand{\mc}{\mathcal}
\newcommand{\mrm}{\mathrm}

\newcommand{\mcA}{\mc{A}}

\newcommand{\mcM}{\mc{M}}

\newcommand{\mcP}{\mc{P}}
\newcommand{\mcQ}{\mc{Q}}

\newcommand{\mcX}{\mc{X}}
\newcommand{\mcY}{\mc{Y}}

\newcommand{\fD}{\f{D}}

\newcommand{\rmR}{\mrm{R}}

\newcommand{\sA}{\s{A}}
\newcommand{\sB}{\s{B}}
\newcommand{\sC}{\s{C}}
\newcommand{\sD}{\s{D}}

\newcommand{\sH}{\s{H}}

\newcommand{\sM}{\s{M}}

\newcommand{\bC}{\bb{C}}

\newcommand{\bF}{\bb{F}}

\newcommand{\bN}{\bb{N}}

\newcommand{\bR}{\bb{R}}

\newcommand{\bZ}{\bb{Z}}

\newcommand{\br}[1]{\mleft( #1 \mright)}
\newcommand{\sbr}[1]{\mleft[ #1 \mright]}
\newcommand{\dbr}[1]{\llbracket #1 \rrbracket}

\newcommand{\set}[2][]{
  \ifthenelse{\equal{#1}{}}{
    \mleft\{ #2 \mright\}
  }{
    \mleft\{ #1\ :\ #2 \mright\}
  }
}

\renewcommand{\to}[1][]{
  \ifthenelse{\equal{#1}{}}{
    \longrightarrow
  }{
    \stackrel{#1}{\longrightarrow}
  }
}

\newcommand{\To}[1][]{
  \ifthenelse{\equal{#1}{}}{
    \Longrightarrow
  }{
    \stackrel{#1}{\Longrightarrow}
  }
}

\renewcommand{\mapsto}[1][]{
    \ifthenelse{\equal{#1}{}}{
      \longmapsto
    }{
      \stackrel{#1}{\longmapsto}
    }
}

\newcommand{\ot}[1][]{
  \ifthenelse{\equal{#1}{}}{
  \longleftarrow
  }{
    \stackrel{#1}{\longleftarrow}
  }
}

\newcommand{\ol}{\overline}
\newcommand{\ul}{\underline}

\renewcommand{\ker}[1]{\text{ker}\br{#1}}

\newcommand{\Gl}{\mathrm{Gl}}

\newcommand{\tr}{\text{tr}\xspace}
\newcommand{\End}{\mathrm{End}}

\newcommand{\hto}[1][]{\stackrel{#1}{\longhookrightarrow}}

\newcommand{\id}{\mathrm{id}}

\newcommand{\prt}{\partial}
\newcommand{\oprt}{\ol{\prt}}

\renewcommand{\Ref}[2][]{\ifthenelse{\equal{#1}{}}{\ref{#2}}
                      {\hyperref[#2]{\ref*{#1}(\ref*{#2})}}}
\newcommand{\Aref}[2][]{\ifthenelse{\equal{#1}{}}{\autoref{#2}}
                      {\hyperref[#2]{\autoref*{#1}\ref*{#2}}}}
\newcommand{\Sref}[1]{\hyperref[#1]{\S \ref*{#1}}}

\newcommand{\Set}{\mathrm{Set}}

\newcommand{\Grpd}{\mathrm{Grpd}}
\newcommand{\op}{\mathrm{op}}

\newcommand{\Sh}{\mathrm{Sh}}

\newcommand{\Hom}{\mathrm{Hom}}
\newcommand{\Ext}{\mathrm{Ext}}

\newcommand{\HHom}{\mathcal{H}\mathrm{om}}
\newcommand{\CHom}[2]{[\nabla_{#1}, \nabla_{#2}]}
\newcommand{\CCHom}[2]{\llbracket \nabla_{#1}, \nabla_{#2} \rrbracket}
\newcommand{\Map}{\mathrm{Map}}

\newcommand{\Sc}[1][]{\ifthenelse{\equal{#1}{}}{\mathrm{Sch}}{\mathrm{Sch}_{/#1}}}

\newcommand{\Cat}{\s{C}\mathrm{at}}

\newcommand{\Vect}{\mathrm{Vect}}

\newcommand{\Conn}{\mathrm{Conn}}
\newcommand{\RConn}{\mathrm{RConn}}

\newcommand{\Dol}{\mathrm{Dol}}
\newcommand{\Ch}{\mathrm{Ch}}
\newcommand{\dg}{\mathrm{dg}}
\newcommand{\diff}{\mathrm{diff}}
\newcommand{\Mod}{\mathrm{Mod}}

\newcommand{\Ob}[1]{\mathrm{Ob}\br{#1}}

\newcommand{\pt}{\mathrm{pt}}
\newcommand{\RGl}{\mathrm{R}\Gamma}

\newcommand{\Cinf}{C^\infty}
\newcommand{\Diff}{\mathrm{Diff}}
\newcommand{\Mfd}{\mathrm{Mfd}}

\setenumerate{itemsep=1pt, label=(\roman*)}
\setitemize{itemsep=1pt, leftmargin=1.5em}

\usepackage[style=alphabetic]{biblatex}

\AtNextBibliography{\small}

\usepackage[margin=1.05in]{geometry}

\title[A diffeological perspective on non-Abelian Hodge theory]
      {A diffeological perspective on non-Abelian Hodge theory}

\author{Mahmud Azam}
\email{mahmud.azam@usask.ca}

\author{Steven Rayan}
\email{rayan@math.usask.ca}

\address{Centre for Quantum Topology and Its Applications (quanTA) and
Department of Mathematics and Statistics, University of Saskatchewan, SK,
Canada~ S7N 5E6}

\begin{document}

\maketitle

\begin{abstract}
We construct diffeological moduli stacks $\sM_{Dol}(X)$ and $\sM_{dR}(X)$ parametrizing smooth
families of Higgs bundles and those of flat bundles, respectively, on a compact K\"ahler manifold
$X$. We then establish an equivalence of stacks between diffeological substacks
$\sM^\sH_{Dol} \subset \sM_{Dol}(X)$ and $\sM^\sH_{dR}(X) \subset \sM_{dR}(X)$, whose fibres over
the point are the categories of semistable Higgs bundles, with the usual condition on Chern classes,
and of flat bundles on $X$, respectively. $\sM^\sH_{Dol}(X)$ contains families of semistable Higgs
bundles to points of which the classical correspondence of coarse moduli spaces does not extend
continuously. This shows that the equivalence we provide is, in a sense, a common extension, in the
context of diffeological moduli stacks, of both the homeomorphism of coarse moduli spaces, and of the
equivalence of categories between semistable Higgs bundles and arbitrary flat bundles.
\end{abstract}

\tableofcontents


\section{Introduction}
\label{sec:introduction}

\subsection{Motivation}

The non-Abelian Hodge correspondence over a compact K\"ahler manifold has two main forms:
\begin{itemize}
\item a homeomorphism of the coarse moduli spaces
of polystable Higgs bundles, with a certain condition on Chern classes, and of semisimple flat
bundles, and
\item a quasiequivalence between the $\dg$--categories of semistable Higgs bundles, with similar
conditions on Chern classes, and of flat bundles, which truncates to an equivalence of the
underlying categories.
\end{itemize}
There are also coarse moduli spaces of semistable Higgs bundles and of flat bundles, but the
homeomorphism does not extend from the polystable locus to the semistable locus
\cite[Counterexample on p. 39]{ModRepFunGrpII}. However, it is not known that this same
counterexample obstructs an equivalence of some suitable moduli stacks of the respective objects.
There are well known Artin stacks of Higgs bundles \cite{CW17} and of flat
connections --- see, for example, \cite[Theorem 5.22]{AR25} --- and whether there is an equivalence
between some loci of these stacks is an open problem. One can also consider topological moduli
stacks of these objects, and there is an equivalence of their Borel-Moore
homologies when the base is a smooth projective curve of genus zero or one \cite{Dav23}.
Whether there is an equivalence of these topological stacks is also an open problem.
There is, however, a stacky non-Abelian Hodge correspondence in the $p$--adic setting \cite{HX26}.

In addition to these investigations, it is also natural to ask:
is there some correspondence between smooth families of Higgs bundles and
those of flat bundles on a compact K\"ahler manifold? More precisely, can we define stacks over
the site of smooth manifolds whose $U$--points, for a smooth manifold $U$, are smooth families
of Higgs bundles or of flat bundles parametrized by $U$?
Then, is there any equivalence between any loci within these
stacks that extends the original equivalences in some sense?
Are these loci geometric stacks in any sense? In this paper, we provide positive answers to these
questions. Our methods are an adaptation of Carlos Simpson's argument for the quasi-equivalence
of the $\dg$--categories of semistable Higgs bundles, with the usual condition on Chern classes, and of
flat bundles given in his celebrated paper \cite{HiggsLocSys}.

\subsection{Overview and Main Results}

Let $X$ be a compact K\"ahler manifold and $U$, a smooth manifold. Let $\pi_X : U \times X \to X$ be
the projection. One sensible definition of a family of Higgs bundles
on $X$ parametrized by a smooth manifold $U$ is a complex vector bundle $E \to U \times X$, that is
relatively holomorphic --- in the sense that its transition functions are holomorphic in the $X$
coordinates --- and equipped with vector bundle maps
$\theta_u : E|_{\set{u} \times X} \to E|_{\set{u} \times X} \otimes T^{0, 1, *}X$,
whose matrices, in a common trivialization of $E$ and $T^{0, 1, *}X$,
are again holomorphic in the $X$ coordinates and smooth in the
$U$ coordinates, and which satisfy the integrability condition $\theta_u \wedge \theta_u = 0$
for all $u \in U$.
We say that such a family is stable, polystable or semistable if each member
$(E|_{\set{u} \times X}, \theta_u)$ is.
One can define a morphism of families
$(E, \set{\theta_u}_{u \in U}) \to (F, \set{\phi_u}_{u \in U})$ to be a vector bundle map $E \to F$
commuting with $\theta_u$ and $\phi_u$ in the usual way ``slicewise''.
One can then informally define a prestack on the site of smooth manifolds by the formula:
\[
\sM_{Dol}(X)(U) := \text{groupoid of families of Higgs bundles on $X$ parametrized by $U$}
\]
One can also define a family of flat bundles in a similar manner: $E|_{\set{u} \times X}$ is now
asked to be equipped with a flat connection $\nabla_{E, u}$ such that its connection matrices
are smooth in $u$. We define simple and semisimple families of flat bundles analogously to stable,
polystable and semistable families of Higgs bundles.
We also have a similar definition of morphism, and can, thus, define another prestack:
\[
\sM_{dR}(X)(U) := \text{groupoid of families of flat bundles on $X$ parametrized by $U$}
\]
In \cref{sec:mod-st-conn}, we give precise, slightly cleaner, definitions of the prestacks described
above --- see \cref{defn:Dol-prest} and \cref{defn:dR-prest} --- as Hom-wise truncations
of the more general differential graded prestacks of these objects --- see
\cref{defn:mod-prest-Higgs} and \cref{defn:mod-prest-conn}.
We prove that these prestacks are, in fact, diffeological stacks --- stacks over the site of smooth
manifolds, presented by diffeological groupoids \cite[Definition 2.9]{RV18} --- see
\cref{thm:mod-st-Dol-dR} and \cref{thm:mod-st-Dol-dR-diff}.
An important point to note is that these stacks are likely only weakly presented by Fr\'echet Lie
groupoids --- see \cref{rmk:Frechet-Lie-groupoid-weak-presentation} --- and hence it is most
convenient to treat them as diffeological stacks.

This section and the rest of the paper makes heavy use of
the language of abstract differentials and connections with respect to them as described in
\cite{GR15}. In \cref{sec:diff-conn}, we recall and extend this language to be able to describe
pullbacks and differential graded categories of abstract connections.
At the same time, this language allows us to give uniform definitions
of smooth families of Higgs bundles and of flat bundles as, simply, connections with respect to
what we call partial Dolbeault operators and partial exterior differentials, respectively ---
see \cref{subsec:exm-diff} and \cref{subsec:exm-conn} --- that are similar to
the relative Dolbeault operators of \cite[\S II.4]{KS60}. \cref{sec:harmonic} develops the notion
of smooth families of harmonic bundles and the theory needed to use these to mediate between
families of Higgs bundles and of flat bundles. In particular, we define a differential graded
prestack of harmonic bundles.

In \cref{sec:NAH}, we show that this differential graded prestack admits forgetful maps
to the differential graded prestacks of Higgs bundles and of flat bundles respectively, just like
the differential graded category of harmonic bundles admits such maps to those of Higgs bundles and
of flat bundles in the original theory.
We further show that morphisms of differential graded prestacks
behave well under taking extension completions of differential graded categories in the sense of
\cite[\S 3]{HiggsLocSys}, which, in turn, yields an equivalence of certain diffeological substacks
of $\sM_{Dol}(X)$ and $\sM_{dR}(X)$ containing, as their fibres over the point, the categories of
semistable Higgs bundles, with the condition on Chern classes, and of flat bundles respectively
--- see \cref{thm:NAH}.
A caveat is that while these substacks contain all
individual semistable Higgs bundles satisfying the Chern class condition on one side and all
individual flat bundles on another side, it is not clear that every family of such objects is
accounted for.
However, at least one family of semistable bundles is contained in the relevant substack of
$\sM_{Dol}(X)$ --- \cref{exm:sst-family} --- to a point of which the homeomorphism of coarse moduli
spaces does not extend continuously.
Therefore, \cref{thm:NAH} is, in fact, a common stack theoretic extension,
in the setting of diffeological moduli stacks,
of the two original forms of the non-Abelian Hodge correspondence. On the one hand, it is a geometrization of the categorical equivalence, and on the other,
it is an extension of the homeomorphism of coarse moduli spaces to the semistable locus. However,
a bit of caution is necessary for the last claim: maps into the coarse moduli spaces do not correspond
to families, and the exact relation between these maps and maps into our moduli stacks, which do correspond to
certain families, is still unclear.\\

\noindent\textbf{Acknowledgements.}
We thank Kuntal Banerjee, Eric Boulter, Robert Cornea, Dat Minh Ha, Matthew Koban, Evan Sundbo,
Rapha\"el Belliard, and Qixiang Wang for helpful conversations. The first-named author was funded
by a Natural Sciences and Engineering Research Council of Canada (NSERC) Canada Graduate Scholarship (Doctoral) and a Canadian Department of National Defence (DND) / NSERC Supplemental Funding Award during this work.  The second-named author was partially supported by an NSERC Discovery Grant. The second-named author is grateful to Tony Pantev for very useful discussions during a visit to the University of Pennsylvania in December 2025.

\section{Differentials and Connections on Sheaves}
\label{sec:diff-conn}

It is convenient to have a uniform language for dealing with Higgs bundles, connections and harmonic
bundles. This was developed in \cite{GR15} to a large extent using the concept of abstract
differentials on sheaves and connections with respect to such differentials. In this section, we
extend the theory of differentials and connections to be able to handle families of Higgs bundles,
connections and harmonic bundles, and their pullbacks. This language will be instrumental in giving
clean definitions of the moduli stacks of interest and reasoning about them.
The theory of differentials and connections can be developed for arbitrary ringed spaces and locally
free sheaves on them, and this is what we will do, and only specialize to manifolds when necessary.
In \cref{subsec:frames}, we will recall some elementary facts about frames for locally free sheaves
that will be useful in several computations.
In \cref{subsec:diff-F-sp}, we recall the theory of spaces equipped with abstract
differentials --- differential spaces, as we call them ---
and extend them to include a notion of morphism of differential spaces, which will be necessary to
define pullbacks of connections with respect to these abstract differentials.
Importantly, this captures the notion of relative or partial
differentials: given smooth manifolds $U, X$, it makes sense to speak of the
exterior differential of $U \times X$ in the direction of $X$. If $X$ is a complex manifold, then it
makes sense to also speak of relative or partial Dolbeault operators on $U \times X$, that detect
functions on $U \times X$ that are holomorphic on each slice $\set{u} \times X \subset U \times X$.
We discuss such examples of differentials in some detail in \cref{subsec:exm-diff}.
We recall the theory of connections with respect to abstract differentials in
\cref{subsec:lambda-d-conn}, including gluing of connections, extension of a given connection to
higher forms, curvature, and local formulas for these.
In \cref{subsec:cons-conn}, we provide several standard constructions for
connections with respect to abstract differentials, along with local formulas. These include
duals, tensor products, sums and scalar products. These will all be
necessary to make precise the notions of differential graded categories of these abstract
connections, differential graded pullback functors between them and differential graded natural
transformations, which will, in turn, be needed for our main results.
In \cref{subsec:pullback-conn}, we discuss in some detail the pullback of abstract connections,
including preservation of curvature, tensor products and duals under pullback.
In \cref{subsec:exm-conn}, we discuss several examples of these abstract connections, including
how families of holomorphic bundles, Higgs bundles and flat connections are all captured as abstract
connections. As a preview, the reader might consider connections with respect to the partial
exterior differentials discussed earlier: these are families of connections on $X$ parametrized by
$U$. Connections with respect to the partial Dolbeault operators are families of holomorphic
bundles. Similarly, families of Higgs bundles on $X$ parametrized by $U$ are connections with
respect to a certain extension of a partial Dolbeault opertor.

\subsection{Frames for Locally Free Sheaves of Modules}
\label{subsec:frames}

\begin{defn}[$\bF$--space]
For any ring $\bF$, an $\bF$--space is a pair $(Y, R)$ where $Y$ is a topological space
and $R$ is a sheaf on $Y$ of $\bF$--algebras. In particular, if $\bF = \bZ$, an $\bF$--space
is just a ringed space. Given two $\bF$--spaces $(Y, R), (Z, S)$, a morphism of
$\bF$--spaces is a pair $(f, f^\sharp)$, where $f : Y \to Z$ is a continuous function
and $f^\sharp : S \to f_*R$ is a morphism of sheaves of $\bF$--algebras.
\end{defn}

\begin{exm}
Our main example of interest will be a smooth manifold $Y$ equipped with its sheaf $\sA^0_Y$
of complex valued functions.
\end{exm}

\begin{defn}[Local Trivializations and Frames]
Let $(Y, R)$ be an $\bF$--space with an $R$--module $E$. A local trivialization of $E$ is a pair
$(V, \phi^V)$ where $V$ is an open subset of $Y$ and $\phi^V : R|_V^{\oplus n} \to E|_V$ is an
isomorphism of $R|_V$--modules. For any open $W \subset V$, let $b^W_i$ denote the element of
$R(W)^{\oplus n}$ with $1$ in the $i$--th entry and $0$ everywhere else --- that is, the $i$--th
standard basis vector of $R(W)^{\oplus n}$ --- and let $e^W_i := \phi^V(b^W_i)$. Then,
$\set{e^W_i}_{i = 1}^n$ is a basis for $E(W)$. We will abuse notation slightly and conflate
$e^W_i$ with $e^V_i$. We will also use $\set{e^V_i}_{i = 1}^n$ to denote the collection of bases
$\set[\set{e^W_i}_{i = 1}^n]{W \subset V \text{ is open}}$, and called it a frame of $E$ over
$V$.
\end{defn}

\begin{rmk}
Any local trivialization $(V, \phi^V)$ yields a local trivialization $(W, \phi^V|_W)$ of $E|_W$ for
any open $W \subset V$.
\end{rmk}

\begin{rmk}[Change of Frames]\label{rmk:frame-change}
Given two local trivializations $(V, \phi^V)$, $(W, \phi^W)$, we will denote denote the transition
functions as $g^{VW} := \phi^V|_{V \cap W}\phi^W|_{V \cap W}^{-1}$ so that $g^{VW}(e^W_i) = e^V_i$
for all $i = 1, \dots, n$. $g^{WV}$ is defined similarly.
There exist sections $g^{VW}_{ij} \in R(V), i, j \in \set{1, \dots, n}$ such that
for all $j \in \set{1, \dots, n}$, we have:
\[
e^V_j = \sum_{i = 1}^n g^{VW}_{ij} e^W_i
\]
That is, $[g^{VW}_{ij}]$ is the matrix of $g^{VW}$ in the basis $\set{e^W_i}_{i = 1}^n$ (of course,
for restrictions of the basis as well). For any open $U \subset V \cap W$ and any sections
$s^V_i, s^W_i \in R(U), i \in \set{1, \dots, n}$ with
$s = \sum_{i = 1}^n s^V_i e^V_i =  \sum_{i = 1}^n s^W_i e^W_i \in E(U)$, we have:
\[
g^{VW}(s)
= g^{VW}\br{\sum_{i = 1}^n s^W_i e^W_i}
= \sum_{i = 1}^n \br{\sum_{j = 1}^n g^{VW}_{ij} s^W_j} e^W_i
\]
At the same time, we have:
\[
s = \sum_{j = 1}^n s^V_j e^V_j
= \sum_{j = 1}^n s^V_j \br{\sum_{i = 1}^n g^{VW}_{ij} e^W_i}
= \sum_{i = 1}^n \br{\sum_{j = 1}^n g^{VW}_{ij} s^V_j} e^W_i
\]
so that:
\[
s^W_i = \sum_{j = 1}^n g^{VW}_{ij} s^V_j
\]
\end{rmk}

\begin{rmk}[Pullback of Frames]\label{rmk:pullback-frame}
Let $(f, f^\sharp) : (Y, R) \to (Z, S)$ be a morphism of $\bF$--spaces. Consider a locally free
$S$--module $E$ of finite rank $n$ on $Z$ with a local trivialization $(V, \phi_V)$.
Let $\set{e^V_i}_{i = 1}^r$ be the corresponding frame of $E|_V$.
Now, consider the open $f^{-1}V \subset Y$. This gives a frame of $f^*F|_{f^{-1}V}$ as follows.
Let $b_i \in S(V)^{\oplus n}$ denote the tuple of sections of $S(V)$ with $1$ in the $i$--th
entry and $0$ everywhere else. Then, letting $f^\sharp b_i := f^{\sharp, \oplus n}(b_i)$,
for every open $W \subset f^{-1}V$, $\set{f^\sharp b_i|_W}_{i = 1}^n$ is an $R(W)$--basis of
$R(W)^{\oplus n}$. Then, letting $f^\sharp e^V_i := f^*(\phi_V)(f^\sharp b_i)$,
$\set{f^\sharp e^V_i|_W}_{i = 1}^n$ is an $R(W)$--basis for $f^*E(W)$ so that
$\set{f^\sharp e^V_i}$ is frame for $f^*E|_{f^{-1}V}$. We will call this the pulled back frame
along $(f, f^\sharp)$.
\end{rmk}


\begin{cor}\label{cor:pullback-frame-change}
Let $(f, f^\sharp) : (Y, R) \to (Z, S)$ be a morphism of $\bF$--spaces. Consider a locally free
$S$--module $E$ of finite rank $n$ on $Z$ with local trivializations $(V, \phi_V), (W, \phi_W)$.
Let $\set{e^V_i}_{i = 1}^n, \set{e^W_i}_{i = 1}^n$ be the corresponding frames
$\set{f^\sharp e^V_i}_{i = 1}^n, \set{f^\sharp e^W_i}_{i = 1}^n$ be the pulled back frames along
$(f, f^\sharp)$. Then, the matrix of the transition function
$g^{f^{-1}V, f^{-1}W} : f^*E|_{f^{-1}V \cap f^{-1}W} \to f^*E|_{f^{-1}V \cap f^{-1}W}$ is
$[f^\sharp g^{VW}_{ij}]$.
\end{cor}

\subsection{Differential $\bF$--Spaces}
\label{subsec:diff-F-sp}

\begin{defn}[Differential on a Sheaf]\label{defn:diff-on-sh}
Let $(Y, R)$ be a $\bF$--space for some ring $\bF$ with an $R$--module $K$. A differential on $K$ is
a collection $d = \set[d^k : K^{\wedge k} \to K^{\wedge (k + 1)}]{k \in \bN}$
of maps of sheaves of $\bF$--modules satisfying:
\begin{enumerate}
\item Flatness: for all $k \in \bN$
\[
d^{k + 1} \circ d^k = 0
\]
\item Graded Leibniz rule: for all $\omega \in K^{\wedge k}, \omega' \in K^{\wedge k'}$,
\[
d^{k + k'}(\omega \wedge \omega')
    = d^k(\omega) \wedge \omega' + (-1)^{k} \omega \wedge d^{k'}(\omega')
\]
\end{enumerate}
The tuple $(Y, R, K, d)$, in this case, is called a differential ($\bF$--)space.
Given this context, a section of $K$ will be called an abstract differential form and the sheaf $K$
will be called a sheaf of abstract differential forms.
A morphism of differential ($\bF$--)spaces of the form
$\br{Y, R, K, d} \to \br{Z, S, L, d'}$  is a pair $(f, f^\sharp)$ where:
\begin{enumerate}
\item $f : Y \to Z$ is a continuous function,
\item $f^\sharp = \set{f^{\sharp, k} : L^{\wedge k} \to f_*(K^{\wedge (k + 1)})}_{k = 0}^\infty$
is a collections of morphisms of sheaves of $R$--modules such that:
    \begin{enumerate}[label=(\alph*)]
    \item $(f, f^{\sharp, 0})$ is a morphism of $\bF$--spaces,
    \item for all $k = 0, 1, 2, \dots$, the following diagram commutes:
    \[\begin{tikzcd}
    L^{\wedge k} \ar[r, "(d')^k"] \ar[d, "f^{\sharp, k}" left] &
    L^{\wedge (k + 1)} \ar[d, "f^{\sharp, k + 1}"] \\
    f_*(K^{\wedge k}) \ar[r, "f_*d^k" below] & f_*(K^{\wedge (k + 1)})
    \end{tikzcd}\]
    \item for all $k = 0, 1, 2, \dots$, the following diagram commutes:
    \[\begin{tikzcd}
    L^{\wedge p} \oplus L^{\wedge q}
        \ar[r, "- \wedge -"] \ar[d, "f^{\sharp, p} \oplus f^{\sharp, q}" left] &
    L^{\wedge (p + q)} \ar[d, "f^{\sharp, p + q}"] \\
    f_*(K^{\wedge p} \oplus K^{\wedge q}) \cong f_*(K^{\wedge p}) \oplus f_*(K^{\wedge q})
        \ar[r, "f_*(- \wedge -)" below] & f_*(K^{\wedge (p + q)})
    \end{tikzcd}\]
    \end{enumerate}
\end{enumerate}
\end{defn}

\begin{notn}
We will not mention the base ring $\bF$ unless necessary and omit the superscripts from $d^k$ and
$f^{\sharp, k}$, when convenient.
\end{notn}

\begin{cor}\label{cor:diff-rest}
Let $(Y, R, K, d)$ be a differential space and $Y'$ be an open subset of $Y$. Then,
$(Y', R|_{Y'}, K|_{Y'}, d|_{Y'})$ is a differential space.
\end{cor}

\begin{defn}[Restricted Differentials]\label{defn:rest-diff}
Let $(Y, R, K, d)$ be a differential space with $R$--modules $P, Q \subset K$ such that
$K = P \oplus Q$. Observe that
$K^{\wedge k} = \bigoplus_{p + q = k}^k P^{\wedge p} \otimes Q^{\wedge q}$
Write $K^{p, q} := P^{\wedge p} \otimes Q^{\wedge q}$ for convenience.
We then get projections and inclusions
$\iota^{p, q} : K^{p, q} \to K^{\wedge k}$ and $pr^{p, q} : K^{\wedge k} \to K^{p, q}$ respectively.
Define $d^{p, q}_{r, s} := pr^{r, s} \circ d \circ \iota^{p, q}$.
Write $d^{p, q} := d^{p, q}_{p + 1, q}$ and $\ol{d}^{p, q} := d^{p, q}_{p, q + 1}$ for
convenience. Then, we call $d^{\bullet, 0}$ and $\ol{d}^{0, \bullet}$ the
restriction of $d$ to $P$ and $Q$ respectively.
\end{defn}

\begin{rmk}
Restrictions of differentials do not lead to differential spaces in general. For example, we can
recall that the Dolbeault operators $\prt, \oprt$ of an almost complex manifold
are differentials if and only if the almost complex structure is integrable or, equivalently,
$d = \prt + \oprt$.
\end{rmk}

\begin{prop}\label{prop:rest-diff}
In the context of \cref{defn:rest-diff}, if $d^{p + 1, 0} \circ d^{p, 0} = 0$ for all
$p = 0, 1, 2, \dots$, then $(Y, R, P, d^{\bullet, 0})$ is a differential space.
\end{prop}
\begin{proof}
Let $\omega \in P^{\wedge p}(U), \tau \in P^{\wedge q}(U)$ for an open
$U \subset Y$ and observe that:
\[
d^{p, 0}(\omega \wedge \tau)
= pr^{p + 1, 0}\br{d(\omega) \wedge \tau + (-1)^p \omega \wedge d(\tau)}
= pr^{p + 1, 0}\br{d(\omega) \wedge \tau} + (-1)^{p} pr^{p + 1, 0}(\omega \wedge d(\tau))
\]
The conclusion now follows from the definition of $d^{\bullet, 0}$ and the fact that the exterior
product of $P^{\wedge \bullet}$ is the restriction of the exterior product of $K^{\wedge \bullet}$.
\end{proof}

\begin{prop}\label{prop:rest-diff-ext}
In the context of \cref{defn:rest-diff}, we have inclusions
$\iota^{k + 1} : \bigoplus_{p + q = k} K^{p + 1, q} \to K^{\wedge (k + 1)}$
and $\ol{\iota}^{k + 1} : \bigoplus_{p + q = k} K^{p, q + 1} \to K^{\wedge (k + 1)}$.
Denote $D^k := \iota^k \circ \bigoplus_{p + q = k} d^{p, q} : K^{\wedge k}
\to K^{\wedge (k + 1)}$
and $\ol{D}^k := \ol{\iota}^k \circ \bigoplus_{p + q = k} \ol{d}^{p, q} :
K^{\wedge k} \to K^{\wedge (k + 1)}$.
If $d^k = D^k + \ol{D}^k$ for all $k = 0, 1, 2, \dots$, then the following are differential
spaces:
\begin{enumerate}
\item $(Y, R, K, D^\bullet)$
\item $(Y, R, K, \ol{D}^{\bullet})$
\item $(Y, R, P, d^{\bullet, 0})$
\item $(Y, R, Q, \ol{d}^{0, \bullet})$
\end{enumerate}
\end{prop}
\begin{proof}
An element $\omega \in K^{\wedge k}(U)$ is of the form
$\sum_{p + q = k} \omega^{p, q}$ where
$\omega^{p, q} \in K^{p, q} = P^{\wedge p} \otimes Q^{\wedge q}$.
Then, by assumption:
\[
0 = d^{k + 1}(d^k(\omega))
= \sum_{p + q = k} D^{k + 1}(D^k(\omega^{p, q}))
   + \sum_{p + q = k} \ol{D}^{k + 1}(\ol{D}^k(\omega^{p, q}))
\]
By the definition of $D$ and $\ol{D}$,
the left and right summands are in different direct summands of $K^{\wedge (k + 1)}$
and hence, they are individually zero. That is,
$D^{k + 1} \circ D^k = 0 = \ol{D}^{k + 1} \circ \ol{D}^k$.
Now, suppose $\tau \in K^{\wedge l}$. Then,
\begin{align*}
 & D(\omega \wedge \tau) + \ol{D}(\omega \wedge \tau) \\
=& d(\omega \wedge \tau) \\
=& (D + \ol{D})(\omega) \wedge \tau + (-1)^k \omega \wedge (D + \ol{D})(\tau) \\
=& (D(\omega) \wedge \tau + (-1)^k \omega \wedge D(\tau))
   + (\ol{D}(\omega) \wedge \tau + (-1)^k \omega \wedge \ol{D}(\tau))
\end{align*}
By the definition of $D$, $\ol{D}$ and the wedge product, the left summands
on the first line and the last line are in the same direct summand. The same holds for the
right summands. The left summands are in different direct summands as the right summands.
This shows that $D$ and $\ol{D}$ satisfy the Leibniz rule individually. That is,
(i) and (ii) are differential spaces.

For (iii), observe that:
\[
0 = D^{k + 1}(D^k(\omega))
= \sum_{p + q = 0} d^{p + 1, q}(d^{p, q}(\omega^{p, q}))
\]
where the summands are all in different direct summands of $K^{\wedge k}$, so that each
of the summands are individually zero. In particular, $d^{p + 1, 0} \circ d^{p, 0} = 0$
so that we may apply \cref{prop:rest-diff}. The argument for (iv) is identical.
\end{proof}

\begin{prop}\label{prop:rest-diff-morphism}
Consider a morphism $(f, f^\sharp) : (Y, R, K, d) \to (Z, S, L, d')$ of differential spaces and
direct summands $P \subset K, Q \subset L$. Consider the inclusions
$\iota^k_A : A^{\wedge k} \to B^{\wedge k}$ and projections
$pr^k_{A} : B^{\wedge k} \to A^{\wedge k}$ for $(A, B) = (P, K), (Q, L)$ respectively.
Suppose the condition of \cref{prop:rest-diff} are satisfied and let
$d_P$ and $d'_Q$ be the differentials restricted to $P$ and $Q$ respectively.
Suppose there exists a collection $f^\sharp_{P, Q}$ of maps $f^{\sharp, k}_{P, Q}$ of $S$--modules
which make the following diagrams commute:
\[\begin{tikzcd}
Q^{\wedge k} \ar[r, "\iota^k_Q"] \ar[d, "f^{\sharp, k}_{P, Q}" left] &
L^{\wedge k} \ar[d, "f^\sharp"] \\
f_*(P^{\wedge k}) \ar[r, "f_*\iota^k_Q" below] &
f_*{K^{\wedge k}}
\end{tikzcd}\hspace{1.5em}\begin{tikzcd}
Q^{\wedge k} \ar[from=r, "pr^k_Q" above] \ar[d, "f^{\sharp, k}_{P, Q}" left] &
L^{\wedge k} \ar[d, "f^\sharp"] \\
f_*(P^{\wedge k}) \ar[from=r, "f_*pr^k_Q" below] &
f_*{K^{\wedge k}}
\end{tikzcd}\]
Then, $(f, f^\sharp_{P, Q}) : (Y, R, P, d_P) \to (Z, S, Q, d_Q)$ is a morphism of differential
spaces.
\end{prop}
\begin{proof}
We observe that $f^{\sharp, 0}_{P, Q} = f^{\sharp, 0}$ so that $(f, f^{\sharp, 0}_{P, Q})$ is a
morphism of $\bF$--spaces. Consider the following diagram:
\[\begin{tikzcd}
Q^{\wedge k}
    \arrow[dd, "{f^{\sharp, k}_{P, Q}}" description]
    \arrow[rr, "d'_Q" description]
    \arrow[rd, "\iota^k_Q" description] & &
Q^{\wedge (k + 1)}
    \arrow[dd, "{f^{\sharp, k + 1}_{P, Q}}" description, near start] & \\ &
L^{\wedge k}
    \arrow[rr, crossing over, "d'" description, near start] & &
L^{\wedge (k + 1)}
    \arrow[lu, "pr^{k + 1}_Q" description]
    \arrow[dd, "{f^{\sharp}}" description] \\
f_*(P^{\wedge k})
    \arrow[rr, "f_*d_P" description, near end] & &
f_*(P^{\wedge (k + 1)}) & \\ &
f_*(K^{\wedge k})
    \arrow[rr, "f_*d" description]
    \arrow[from=uu, crossing over, "{f^{\sharp}}" description, near start]
    \arrow[from=lu, "f_*\iota^{k}_P" description] & &
f_*(K^{\wedge (k + 1)}) \arrow[lu, "f_*pr^{k + 1}_P" description]
\end{tikzcd}\]
The top and bottom faces of this diagram commute by definition of restricted differentials.
The front face commutes because $(f, f^\sharp)$ is a morphism of differential spaces.
The left and right faces commute by hypothesis.
Therefore, the back face commutes. By a similar argument,
the back face of the following diagram commutes:
\[\begin{tikzcd}
Q^{\wedge p} \oplus Q^{\wedge q}
    \arrow[dd, "f^{\sharp, p}_{P, Q} \oplus f^{\sharp, q}_{P, Q}" description]
    \arrow[rr, "- \wedge -"]
    \arrow[rd, "\iota^p_Q \oplus \iota^q_Q" description] & &
Q^{\wedge (p + q)}
    \arrow[dd, "f^{\sharp, p + q}_{P, Q}" description, near start] & \\ &
L^{\wedge p} \oplus L^{\wedge q}
    \arrow[rr, crossing over, "- \wedge -" near start] & &
L^{\wedge (p + q)}
    \arrow[lu, "pr^{p + q}_Q" description]
    \arrow[dd, "f^{\sharp}" description] \\
f_*(P^{\wedge p}) \oplus f_*(P^{\wedge q})
    \arrow[rr, "f_*(- \wedge -)" description, near end] & &
f_*(P^{\wedge (k + 1)}) & \\ &
f_*(K^{\wedge p}) \oplus f_*(K^{\wedge q})
    \arrow[rr, "f_*(- \wedge -)" description]
    \arrow[from=uu, crossing over, "f^{\sharp} \oplus f^\sharp" description, near start]
    \arrow[from=lu, "f_*\iota^{p}_P \oplus f_*\iota^q_P" description] & &
f_*(K^{\wedge (k + 1)}) \arrow[lu, "f_*pr^{p + q}_P" description]
\end{tikzcd}\]
\end{proof}

\begin{cor}[Sums of Differentials]
Let $(Y, R)$ be a $\bF$--space with a finite rank $R$--module $K$ equipped with two differentials
$d_1, d_2$. Then, $d_1 + d_2$ is easily verified to be a differential again.
\end{cor}

\begin{cor}[Scalar Products of Differentials]
Let $(Y, R)$ be an $\bF$--space with a finite rank $R$--module $K$ equipped with a differential
$d$. Let $\lambda \in R(Y)$ be a global section of $R$. Then, $\lambda \cdot d$ is easily verified
to be a differential again.
\end{cor}

\begin{defn}[Module of Differentials]
Let $(Y, R)$ be a $\bF$--space with a finite rank $R$--module $K$. Then, the set
\[
D(Y, R, K) = \set[d^\bullet : K^{\wedge \bullet} \to K^{\wedge (\bullet + 1)}]
                 {d \text{ is a differential on } K}
\]
is an $R(Y)$--module.
\end{defn}

\subsection{Examples of Differential Spaces and their Morphisms}
\label{subsec:exm-diff}

\begin{notn}
For a smooth manifold $X$ and $p : E \to X$, an arbitrary smooth vector bundle, we will use the
following notation:
\begin{figure}[H]
\begin{tabularx}{\textwidth}{l c l}
$E^{sec}$ & : & sheaf of smooth sections of $E$ \\
$E_\bC$ & : & complexification of $E$
\end{tabularx}
\end{figure}
We will also use the following notation for the specific bundles and sheaves on $X$ that will be
used throughout the paper:
\begin{figure}[H]
\begin{tabularx}{\textwidth}{l c l}
$C^\infty_X$ & : & $(X \times \bR)^{sec}$, sheaf of smooth real valued functions on $X$ \\
$TX$ & : & tangent bundle of $X$ \\
$T_X$ & : & $(TX)^{sec}$, sheaf of real vector fields on $X$ \\
$\Theta_{X}$ & : & $((TX)_\bC)^{sec}$, sheaf of complex vector fields on $X$ \\
$T^*X$ & : & cotangent bundle of $X$ \\
$T^{*, k}_X$ & : & $((T^*X)^{sec})^{\wedge k}$, sheaf of real $k$--forms on $X$ \\
$\sA^k_X$ & : & $((T^*X_\bC)^{sec})^{\wedge k}$, sheaf of complex $k$--forms on $X$ \\
$\sA^{p, q}_X$ & : & sheaf of complex $(p, q)$--forms on $X$, under an almost complex structure \\
$\Omega^k_X$ & : & sheaf of holomorphic $k$--forms on $X$, under a complex structure \\
$\Omega^{p, q}_X$ & : & sheaf of holomorphic $(p, q)$--forms on $X$, under a complex structure
\end{tabularx}
\end{figure}
\end{notn}

We now provide the main examples of differentials that we will use.

\begin{exm}[Exterior Differentials]\label{exm:ext-diff}
The real and complex exterior differentials:
\[\begin{array}{ccccc}
d_\bR^k &:& T^{*, k}_X &\to& T^{*, k + 1}_X \\
d_\bC^k &:& \sA^{k}_X &\to& \sA^{k + 1}_X
\end{array}\]
are differentials on the sheaves $T^*_X$ and $\sA^1_X$ respectively. We will denote them both
by $d$, unless necessary to distinguish between them.
\end{exm}

\begin{exm}[Dolbeault Operators]\label{exm:Dol-op}
Given an almost complex structure $J : TX_\bC \to TX_\bC$, we have the splitting
$TX_\bC = TX^{0, 1} \oplus TX^{1, 0}$, inducing a splitting $\sA^1 = \sA^{0, 1} \oplus \sA^{0, 1}$
giving the Dolbeault operators $\prt^{p, q} := pr^{p + 1, q} d_\bC \iota^{p, q}$,
$\oprt^{p, q} := pr^{p, q + 1} d_\bC \iota^{p, q}$.
When the complex structure is integrable, we have $d_\bC = \prt + \oprt$ so that
by \cref{prop:rest-diff-ext}, taking $d^{p, q} = \prt^{p, q}, \ol{d}^{p, q} = \oprt^{p, q}$
and denoting the corresponding $D^k$ and $\ol{D}^k$ by simply $\prt^k$ and $\oprt^k$ respectively,
we get differential spaces:
\begin{enumerate}
\item $(X, \sA_X^0, \sA_X^\bullet, \prt^\bullet)$
\item $(X, \sA_X^0, \sA_X^\bullet, \oprt^\bullet)$
\item $(X, \sA_X^0, \sA_X^{\bullet, 0}, \prt^{\bullet, 0})$
\item $(X, \sA_X^0, \sA_X^{0, \bullet}, \oprt^{0, \bullet})$
\end{enumerate}
When there is no confusion, we will simply write $\prt$ and $\oprt$ for all of these operators.
\end{exm}

\begin{exm}[Partial Exterior Differential]\label{exm:prt-ext-diff}
Consider the projections of a product of smooth manifolds $\pi_W : U \times X \to W, W = U, X$.
We then have a canonical splitting $T^*(U \times X) = \pi_U^*T^*U \oplus \pi_X^*T^*X$ inducing a
splitting $\sA^{k}_{U \times X} = \bigoplus_{p + q = k} \pi_U^*\sA^p_U \otimes \pi_X^*\sA^q_X$.
Let $\sA^p_{U/X} = \pi_U^*\sA^p_U$ and $\sA^q_{X/U} = \pi_X^*\sA^q_X$.
Let $pr^{p, q}, \iota^{p, q}$ be as in \cref{defn:rest-diff} and $\iota^k, \ol{\iota}^k$,
as in \cref{prop:rest-diff-ext}. Denote:
\begin{enumerate}
\item $d_{U/X}^p := pr^{p + 1, q}d\iota^{p, q}$
\item $d_{X/U}^q := pr^{p, q + 1}d\iota^{p, q}$
\item $D_{U/X}^k = \iota^k \circ \bigoplus_{p + q = k} d_{U/X}^{p, q}$
\item $D_{X/U}^k = \ol{\iota}^k \circ \bigoplus_{p + q = k} d_{X/U}^{p, q}$
\end{enumerate}
At any $(u, x) \in U \times X$, we can consider coordinate neighbourhoods
$u \in U' \subset U, x \in X' \subset X$ so that $U' \times X'$ is a coordinate neighbourhood of
$(u, x)$. Let the coordinates be $u_i : U' \to \bR, x_j : X' \to \bR$ so that
$\set{du_i} \cup \set{dx_j}$ is a frame for $\sA^1_{U \times X}$. In this frame, we have:
\[
d\br{\sum_{I, J} f_{I, J} du_Idx_J}
= \sum_{I, J} \sum_{i, j} \br{
    \frac{\prt f_{I, J}}{\prt u_i} du_i \wedge du_Idx_J
    + \frac{\prt f_{I, J}}{\prt x_j} dx_j \wedge du_Idx_J
}
\]
where $I, J$ are multi-indices with $|I| + |J| = k$.
The above local formula
shows that:
\[
d = D_{U/X}^k + D_{X/U}^k
\]
Thus, we can apply \cref{prop:rest-diff-ext} to obtain differential spaces:
\begin{enumerate}
\item $(U \times X, \sA_{U \times X}^0, \sA_{U \times X}^{\bullet}, D_{U/X}^\bullet)$
\item $(U \times X, \sA_{U \times X}^0, \sA_{U \times X}^{\bullet}, D_{X/U}^\bullet)$
\item $(U \times X, \sA_{U \times X}^0, \sA_{U/X}^{\bullet}, d_{U/X}^{\bullet})$
\item $(U \times X, \sA_{U \times X}^0, \sA_{X/U}^{\bullet}, d_{X/U}^\bullet)$
\end{enumerate}
We will mainly use (iii) and (iv) above. The associated operators will be called the partial
exterior differentials of $U \times X$ along $X$.
\end{exm}

\begin{exm}[Partial Dolbeault Operators]\label{exm:prt-Dol-op}
Let $X$ be a complex manifold and $U$, a smooth manifold. Consider the splitting
$\sA^k_{U \times X} = \bigoplus_{l + m = k} \pi_U^*\sA^l_U \otimes \pi_X^*\sA^m_X$ as in
\cref{exm:prt-ext-diff}.
The complex structure on $X$ yields a further splitting:
\[
\sA^k_{U \times X} = \bigoplus_{l + m = k} \bigoplus_{p + q = m}
\pi_U^*\sA^l_U \otimes \pi_X^*\sA^{p, q}_X
\]
Denote by $\sA_{X/U}^{p, q}$ the sheaf $\pi_X^*\sA_{X}^{p, q}$.
Let $pr^{p, q}$ and $\iota^{p, q}$ now denote the projection
$\sA_{U \times X}^k \to \sA_{X/U}^{p, q}$ and the inclusion
$\sA_{X/U}^{p, q} \to \sA_{U \times X}^k$ respectively, where $k = p + q$.
Take $\iota^m$ and $\ol{\iota}^m$ to be the inclusions
$\bigoplus_{p + q = m} \sA^{p + 1, q}_{X/U} \to \sA^{m}_{X/U}$
and $\bigoplus_{p + q = m} \sA^{p, q + 1} \to \sA^m_{X/U}$.
Denote:
\begin{enumerate}
\item $\prt_{X/U}^{p, q} := pr^{p + 1, q}d_{X/U}\iota^{p, q}$
\item $\oprt_{X/U}^{p, q} := pr^{p, q + 1}d_{X/U}\iota^{p, q}$
\item $\prt_{X/U}^{k} := \iota^m \bigoplus_{p + q = m} \prt^{p, q}_{X/U}$
\item $\oprt_{X/U}^{k} := \ol{\iota}^m \bigoplus_{p + q = m} \prt^{p, q}_{X/U}$
\end{enumerate}
For any $(u, x) \in U \times X$, we have a holomorphic coordinate
neighbourhood $x \in X' \subset X$ with holomorphic coordinates $z_j : X' \to \bC$ and
anti-holomorphic coordinates $\ol{z}_j : X' \to \bC$. This gives real coordinates:
$x_j = \frac{1}{2}(z_j + \ol{z}_j), y_j = \frac{1}{2i}(z_j - \ol{z}_j) : X' \to \bR$.
We also have a real coordinate neighbourhood
$u \in U' \subset U$ with real coordinates $u_j : U' \to \bR$. Then,
$\set{u_i} \cup \set{x_j} \cup \set{y_j}$ are coordinates for $U' \times X'$.
From the local formula for the differential $d$ along with the definition of $d_{X/U}$ from
\cref{exm:prt-ext-diff} and the standard change between holomorphic and real coordinates on $X'$, we
have the following formula for $d_{X/U}$:
\[
d_{X/U}\br{\sum_{P, Q} f_{P, Q} dz_Pd\ol{z}_Q}
= \sum_{P, Q} \sum_{i, j} \br{
    \frac{\prt f_{P, Q}}{\prt z_i} dz_i \wedge dz_Pd\ol{z}_Q
    + \frac{\prt f_{P, Q}}{\prt \ol{z}_j} d\ol{z}_j \wedge dz_Pd\ol{z}_Q
}
\]
where $P, Q$ are multi-indices with $|P| = p, |Q| = q, p + q = k$.
From this formula, we see that:
\[
d_{X/U}^m = \prt_{X/U}^m + \oprt_{X/U}^m
\]
We can then apply \cref{prop:rest-diff-ext} to obtain differential spaces:
\begin{enumerate}
\item $(U \times X, \sA_{U \times X}^0, \sA_{X/U}^{\bullet}, \prt_{X/U}^\bullet)$
\item $(U \times X, \sA_{U \times X}^0, \sA_{X/U}^{\bullet}, \oprt_{X/U}^\bullet)$
\item $(U \times X, \sA_{U \times X}^0, \sA_{X/U}^{\bullet, 0}, \prt_{X/U}^{\bullet, 0})$
\item $(U \times X, \sA_{U \times X}^0, \sA_{X/U}^{0, \bullet}, \oprt_{X/U}^{0, \bullet})$
\end{enumerate}
We will make use of all of these differential spaces and the operators will be called
the partial Dolbeault operators on $U \times X$ along $X$.
\end{exm}

\begin{rmk}
The partial Dolbeault operators are special cases of the relative Dolbeault operators defined
in \cite[\S II.4]{KS60}.
\end{rmk}

We then discuss examples of morphisms of differential spaces that we will use.

\begin{exm}[Pullback of Differential Forms]\label{exm:ext-diff-morphism}
For any smooth map $f : Y \to Z$ of smooth manifolds, if we let $f^\sharp$ denote the usual
pullback of complex differential forms, we have a morphism of differential spaces:
\[
(f, f^\sharp) : (Y, \sA_Y^0, \sA_Y^{\bullet}, d)
\to (Z, \sA_Z^0, \sA_Z^{\bullet}, d)
\]
\end{exm}

\begin{exm}[Pullback of Partial Differential Forms]\label{exm:prt-ext-diff-morphism}
Let $f \times g : V \times Y \to U \times X$ be a smooth map of smooth manifolds.
Consider the pullback of differential forms along $f \times g : V \times Y \to U \times X$ and
denote it
$(f \times g)^{\sharp, k} : \sA^k_{U \times X} \to (f \times g)_*\sA^k_{V \times Y}$.
Take $(W, Z) = (U, X), (V, Y)$ and consider the splittings
$\sA^k_{W \times Z} \cong \bigoplus_{p + q = k} \sA^p_{W/Z} \oplus \sA^q_{Z/W}$ along with the
projections $pr^k_W : \sA_{W \times Z}^k \to \sA_{Z/W}^k$
and the inclusions $\iota^k_W : \sA_{Z/W}^k \to \sA_{W \times Z}^k$.
By a slight abuse of notation, use the same to denote the projections and inclusions
of the corresponding sub-bundles of the complexified tangent bundles.
Denote $h^{\sharp, k} := pr^k_V \circ (f \times g)^{\sharp, k} \circ \iota^k_U$.
We can then verify that the following diagrams commute:
\[\begin{small}\begin{tikzcd}[column sep=huge]
\sA_{X/U}^k \ar[r, "\iota^k_U"] \ar[d, "h^{\sharp, k}" left] &
\sA_{U \times X}^k \ar[d, "(f \times g)^{\sharp, k}"] \\
(f \times g)_*\sA_{Y/V}^k \ar[r, "(f \times g)_*\iota^k_V" below] &
(f \times g)_*\sA_{V \times Y}^k
\end{tikzcd}
\hspace{1.5em}
\begin{tikzcd}[column sep=huge]
\sA_{X/U}^k \ar[from=r, "pr^k_U" above] \ar[d, "h^{\sharp, k}" left] &
\sA_{U \times X}^k \ar[d, "(f \times g)^{\sharp, k}"] \\
(f \times g)_*\sA_{Y/V}^k \ar[from=r, "(f \times g)_*pr^k_V" below] &
(f \times g)_*\sA_{V \times Y}^k
\end{tikzcd}\end{small}\]
To see this, we first observe that the complexified tangent map
$T(f \times g) \otimes_\bR \bC : T(V \times Y)_\bC \to T(U \times X)_\bC$ is given on fibres over
$(v, y) \in V \times Y, (f(v), g(y)) \in U \times X$ by
$T_vf \otimes \bC \oplus T_yg \otimes \bC$. Then, for any $\omega \in \sA_{X/U}^k$,
we have:
\begin{align*}
((f \times g)^\sharp(\iota^k_U(\omega)))_{(v, y)}
=& \iota^k_U(\omega)_{(f(v), g(y))} \circ (T_vf \oplus T_yg)^{\oplus k} \\
=& \omega_{(f(v), g(y))} \circ pr^k_U \circ (T_vf \oplus T_yg)^{\oplus k} \\
=& \omega_{(f(v), g(y))} \circ T_yg^{\oplus k} \\
\end{align*}
Again, we observe that:
\begin{align*}
((f \times g)_*\iota^k_V(h^\sharp(\omega)))_{(v, y)}
=& (\iota^k_V \circ pr^k_V \circ (f \times g)^\sharp \circ \iota^k_U)(\omega)_{(v, y)} \\
=& ((f \times g)^\sharp \circ \iota^k_U)(\omega)_{(v, y)} \circ pr^k_V \circ \iota^k_V \\
=& ((f \times g)^\sharp \circ \iota^k_U)(\omega)_{(v, y)} \\
=& \iota^k_U(\omega)_{(f(v), g(y))} \circ (T_vf \oplus T_yg)^{\oplus k} \\
=& \omega_{(f(v), g(y))} \circ pr^k_U \circ (T_vf \oplus T_yg)^{\oplus k} \\
=& \omega_{(f(v), g(y))} \circ T_yg^{\oplus k}
\end{align*}
This shows that the first square commutes.
Next, for any $\tau \in \sA_{U \times X}^k$, we observe that:
\begin{align*}
   ((f \times g)_*pr_V^k \circ (f \times g)^\sharp)(\tau)_{(v, y)}
=& (f \times g)^\sharp(\tau)_{(v, x)} \circ \iota_V^k \\
=& \tau_{(f(v), g(y))} \circ (T_vf \oplus T_yg)^{\oplus k} \circ \iota_V^k \\
=& \tau_{(f(v), g(y))} \circ (0 \oplus T_yg)^{\oplus k}
\end{align*}
and:
\begin{align*}
(f^\sharp \circ pr_U^k)(\tau)_{(v, y)}
=& (pr_V^k \circ (f \times g)^\sharp \circ \iota_U^k \circ pr_U^k)(\tau)_{(v, y)} \\
=& ((f \times g)^\sharp \circ \iota_U^k \circ pr_U^k)(\tau)_{(v, y)} \circ \iota_V^k \\
=& (\iota_U^k \circ pr_U^k)(\tau)_{(f(v), g(y))} \circ (T_vf \oplus T_yg)^{\oplus k}
   \circ \iota_V^k \\
=& \tau_{(f(v), g(y))} \circ pr_U^k \circ \iota_U^k \circ (T_vf \oplus T_yg)^{\oplus k}
   \circ \iota_V^k \\
=& \tau_{(f(v), g(y))} \circ (T_vf \oplus T_yg)^{\oplus k} \circ \iota_V^k \\
=& \tau_{(f(v), g(y))} \circ (0 \oplus T_yg)^{\oplus k}
\end{align*}
showing the second square commutes. Thus, by \cref{prop:rest-diff-morphism}, we have a morphism
of differential spaces:
\[
(f \times  g, h^\sharp) : (V \times Y, \sA_{V \times Y}^0, \sA_{Y/V}^{\bullet}, d_{Y/V})
\to (U \times X, \sA_{U \times X}^0, \sA_{X/U}^{\bullet}, d_{X/U})
\]
We may take $U = V = \pt$ to recover \cref{exm:ext-diff-morphism}.
Furthermore, we may take $Y = X, g = \id_X$, which is the case we will use mostly use in this
paper.
\end{exm}

\begin{exm}[Pullback of Partial $(p, q)$--Forms]\label{exm:prt-Dol-morphism}
In the context of \cref{exm:prt-ext-diff-morphism}, let $X, Y$ be complex manifolds and
$g : Y \to X$, holomorphic. Then, for any $(v, y) \in V \times Y$, $T_{(v, y)}(f \times g)$ has a
further splitting, by the holomorphicity of $g$:
\[
T_{(v, y)}(f \times g) \otimes \bC = T_vf \otimes \bC \oplus T_yg \otimes \bC
= T_vf \otimes \bC \oplus T_yg^{0, 1} \oplus T_yg^{1, 0}
\]
where $T_yg^{0, 1}, T_g^{1, 0}$ are the restrictions of $T_yg \otimes \bC$ to $T_yY^{0, 1}$
and $T_yY^{1, 0}$ respectively.
Let $\eta_W^{p, q}, \mu_W^{p, q}$, now denote the
inclusions $\sA_{Z/V}^{p, q} \to \sA_{Z/W}^k$ and
projections $\sA_{Z/W}^k \to \sA_{Z/W}^{p, q}$
respectively.
Let $\xi^{\sharp, p, q}$ now denote
$\mu_V^{p, q} \circ h^{\sharp, k} \circ \eta_U^{p, q}$.
A very similar computation as in \cref{exm:prt-ext-diff-morphism}
shows that the following diagrams commute:
\[\begin{small}\begin{tikzcd}[column sep=huge]
\sA_{X/U}^{p, q} \ar[r, "\eta^{p, q}_U"] \ar[d, "\xi^{\sharp, p, q}" left] &
\sA_{X/U}^k \ar[d, "h^{\sharp, k}"] \\
(f \times g)_*\sA_{Y/V}^{p, q} \ar[r, "(f \times g)_*\eta^{p, q}_V" below] &
(f \times g)_*\sA_{V \times Y}^k
\end{tikzcd}
\hspace{1.5em}
\begin{tikzcd}[column sep=huge]
\sA_{X/U}^{p, q} \ar[from=r, "\mu^{p, q}_U" above] \ar[d, "\xi^{\sharp, p, q}" left] &
\sA_{X/U}^k \ar[d, "h^{\sharp, k}"] \\
(f \times g)_*\sA_{Y/V}^{p, q} \ar[from=r, "(f \times g)_*\mu^{p, q}_V" below] &
(f \times g)_*\sA_{V \times Y}^k
\end{tikzcd}\end{small}\]
Furthermore, the commutativity of the above squares also shows that:
\[
h^{\sharp, k} = \bigoplus_{p + q = k} \xi^{\sharp, p, q}
\]
by the fact the direct sums are products and coproducts of modules at the same time.
Thus, we may apply $\cref{prop:rest-diff-morphism}$ to obtain morphisms of differential
spaces:
\begin{enumerate}
\item $(f \times  g, \xi^{\sharp, \bullet, 0})
: (V \times Y, \sA_{V \times Y}^0, \sA_{Y/V}^{\bullet, 0}, \prt_{Y/V}^{\bullet, 0})
\to (U \times X, \sA_{U \times X}^0, \sA_{X/U}^{\bullet, 0}, \prt_{X/U}^{\bullet, 0})$

\item $(f \times  g, \xi^{\sharp, 0, \bullet})
: (V \times Y, \sA_{V \times Y}^0, \sA_{Y/V}^{0, \bullet}, \oprt_{Y/V}^{0, \bullet})
\to (U \times X, \sA_{U \times X}^0, \sA_{X/U}^{0, \bullet}, \oprt_{X/U}^{0, \bullet})$

\item $(f \times  g, h^{\sharp, \bullet})
: (V \times Y, \sA_{V \times Y}^0, \sA_{Y/V}^{\bullet}, \prt_{Y/V}^{\bullet})
\to (U \times X, \sA_{U \times X}^0, \sA_{X/U}^{\bullet}, \prt_{X/U}^{\bullet})$

\item $(f \times  g, h^{\sharp, \bullet})
: (V \times Y, \sA_{V \times Y}^0, \sA_{Y/V}^{\bullet}, \oprt_{Y/V}^{\bullet})
\to (U \times X, \sA_{U \times X}^0, \sA_{X/U}^{\bullet}, \oprt_{X/U}^{\bullet})$
\end{enumerate}
where $\prt_{Y/V}, \prt_{X/U}, \oprt_{Y/V}, \oprt_{X/U}$ are the partial Dolbeault operators of
\cref{exm:prt-Dol-op}.
\end{exm}

\subsection{$\lambda$--$d$--Connections}
\label{subsec:lambda-d-conn}

\begin{defn}[$\lambda$--$d$--Connections]\label{defn:conn}
Let $\br{Y, R, K, d}$ be a differential space and $\lambda \in R(Y)$, a global section.
A $\lambda$--$d$--connection on a finite rank locally free sheaf
$E$ on $(Y, R)$ is a morphism of sheaves of $\bF$--modules $\nabla : E \to E \otimes K$ satisfying
the Leibniz rule: that is,
for all open $Y' \subset Y$, $s \in E(Y'), r \in R(Y')$,
\[
\nabla_{Y'}(r \cdot s) = \lambda \cdot s \otimes d(r) + r \cdot \nabla_{Y'}(s)
\]
For two $R$--modules $E, F$ with $\lambda$--$d$--connections $\nabla, \nabla'$, respectively,
a morphism of connections $(E, \nabla) \to (F, \nabla')$ is a morphism of $R$--modules $f : E \to F$
making the following diagram commute:
\[\begin{tikzcd}
E \ar[r, "f"] \ar[d, "\nabla" left] & F \ar[d, "\nabla'"] \\
E \otimes K \ar[r, "f \otimes \id_K" below] & F \otimes K
\end{tikzcd}\]
\end{defn}

\begin{rmk}\label{rmk:0-d-conn}
In the context of \cref{defn:conn}, if $\lambda = 0$ or $d = 0$, then $\nabla$ is just an
$R$--linear map and is a $\lambda'$--$d'$--connection for any other $\lambda' \in R(Y)$
and differential $d'$ on $K$.
\end{rmk}

\begin{cor}\label{cor:conn-rest}
Let $(Y, R, K, d)$ be a differential space and $E$ a finite rank locally free $R$--module
with a $\lambda$--$d$--connection $\nabla$. Then, for any open subset $Y'$ of $Y$,
$\nabla|_{Y'}$ is a $\lambda$--$d|_{Y'}$--connection on $E|_{Y'}$.
\end{cor}

\begin{cor}\label{cor:conn-coll}
Let $(Y, R, K, d)$ be a differential space and $E$, a finite rank locally free $R$--module, such
that there exists an open cover $Y = \bigcup_{\alpha \in I} Y_\alpha$ and
$\lambda|_{Y_\alpha}$--$d|_{Y_\alpha}$--connections
$\nabla^\alpha : E|_{Y_\alpha} \to (E \otimes K)|_{Y_\alpha}$ satisfying that
$\nabla^\alpha|_{Y_\alpha \cap Y_\beta} = \nabla^\beta|_{Y_\alpha \cap Y_\beta}$.
Then, the unique sheaf map $\nabla : E \to E \otimes K$ given by gluing
\cite[\href{https://stacks.math.columbia.edu/tag/04TN}{Tag 04TN}]{stacks-project}
is a $\lambda$--$d$--connection.
\end{cor}
\begin{proof}
Let $W \subset Y$ be any open subset and consider $s \in E(W), r \in R(W)$. Consider the
open subsets $W_\alpha := W \cap Y_\alpha$ and the restrictions
$s^\alpha := s|_{W_\alpha}, r^\alpha := r|_{W_\alpha}$. We then have:
\[
\nabla(r \cdot s)|_{W_\alpha}
= \nabla^\alpha(r^\alpha \cdot s^\alpha)
= s^\alpha \otimes d(r^\alpha) + r^\alpha \cdot \nabla^\alpha(s^\alpha)
= \br{s \otimes d(r) + r \cdot \nabla(s)}|_{W_\alpha}
\]
for each $\alpha \in I$. If two sections are equal when restricted to all members of an open cover,
then they are equal. This shows that $\nabla$ satisfies the Leibniz rule over $W$, which was an
arbitrary open subset of $Y$. $\bF$--linearity is similar.
\end{proof}

\begin{defn}[Connection Form]\label{defn:conn-form}
Let $\nabla$ be a $\lambda$--$d$--connection on a locally free $R$--module $E$ of finite rank $n$
over a differential space $(Y, R, K, d)$. Let $(V, \phi_V)$ be a local trivialization of $E$
with a frame $\set{e^V_i}_{i = 1}^n$.
Then, we have a sheaf map $d^V : E|_{V} \to (E \otimes K)|_{V}$ defined by the formula:
\[
d^V\br{\sum_{i = 1}^n s^V_i e^V_i} := \sum_{i = 1}^n d(s^V_i) \otimes e^\alpha_i
\]
where $s^V_i \in R(V)$.
Let $\nabla^V$ denote the sheaf map
$\nabla|_{V} : E|_{Y_\alpha} \to (E \otimes K)|_{Y_\alpha}$.
We will call the sheaf map $A^V : E|_{V} \to (E \otimes K)_{V}$ defined by
$A^V := \nabla^V - \lambda d^V$ the connection form of $\nabla$ over $V$.
\end{defn}

\begin{prop}
In the context of \cref{defn:conn-form}, $A^V$ is $R|_{V}$--linear.
\end{prop}
\begin{proof}
We observe that for any local section $r$ of $R|_{V}$:
\[
A^V(re^V_i) = \nabla^V(re^V_i) - \lambda d^V(re^V_i)
= \lambda d(r) \otimes e^V_i + r\nabla^V(e^V_i) - \lambda d(r) \otimes e^V_i
= r \nabla^V(e^V_i)
\]
We then observe that:
\[
A^V(e^V_i) = \nabla^V(e^V_i) - \lambda d(1) \otimes e^V_i
= \nabla^V(e^V_i)
\]
where $d(1) = 0$ because $d(1) = d(1 \cdot 1) = d(1) \cdot 1 + 1 \cdot d(1) = d(1) + d(1)$.
Therefore,
\[
A^V(re^V_i) = r A^V(e^V_i)
\]
Of course, additivity follows from the fact that a pointwise difference of additive maps is again
additive.
\end{proof}

\begin{defn}[Connection Matrix]\label{defn:conn-matrix}
In the context of \cref{defn:conn-form}, there exist sections $A^V_{ij} \in K(V)$ such that
\[
A^V_{e_{j}} = \sum_{i} e^V_i \otimes A^V_{ij}
\]
so that for any general sections $s = \sum_{i} s^V_ie^V_i \in E(V)$, we have:
\[
A^V\br{\sum_{i} s^V_ie^V_i} = \sum_{i} e^V_i \otimes \br{\sum_{j} A^V_{ij} s^V_j}
\]
Then, $[A^V_{ij}]$ is called the connection matrix of $\nabla$ in the $\set{e^V_i}_i$ frame.
\end{defn}

\begin{prop}\label{prop:diff-frame-change}
In the context of \cref{defn:conn-form}, consider another local trivialization $(W, \phi^W)$
with a corresponding frame $\set{e^W_i}$, and transition functions
$g^{VW}$ and $g^{WV}$ as in \cref{rmk:frame-change}. Let $[g^{WV}_{ij}]$ be the matrix of
$g^{WV}$ in the $\set{e^V_i}_i$ frame. Consider the $R|_{V \cap W}$--linear map
$d(g^{WV}) : E|_{V \cap W} \to (E \otimes K)|_{V \cap W}$
defined over any open $W' \subset V \cap W$ by the formula:
\[
d(g^{WV})\br{\sum_{j = 1}^n s^V_i e^V_j}
= \sum_{i = 1}^n e^V_i \otimes \br{\sum_{j = 1}^n d\br{g^{WV}_{ij}} \wedge s^V_j}
\]
Then, the following identity holds:
\[
d^V = (g^{VW} \otimes \id_K)d(g^{WV}) + d^W
\]
\end{prop}
\begin{proof}
Let $s \in E(W')$. There are sections $s^V_i, s^W_i \in R(W), i = 1, \dots, n$
such that $s = \sum_{i = 1}^n s^V_i e^V_i = \sum_{i = 1}^n s^W_ie^W_i$.
We then have the following by \cref{rmk:frame-change}:
\[
s^V_i = \sum_{j = 1}^n g^{WV}_{ij} s^W_j
\]
This yields:
\begin{align*}
 & d^V(s) \\
=& \sum_{i = 1}^n d(s^V_i) \otimes e^V_i \\
=& \sum_{i = 1}^n d\br{\sum_{j = 1}^n g^{WV}_{ij}s^W_j} \otimes e^V_i \\
=& \sum_{i = 1}^n \sum_{j = 1}^n \br{d(g^{WV}_{ij}) \cdot s^W_j
                      + g^{VW}_{ij} \cdot d(s^W_j)} \otimes e^V_i \\
=& \sum_{i = 1}^n \sum_{j = 1}^n d(g^{WV}_{ij}) \cdot s^W_j \otimes e^V_i
   + \sum_{i = 1}^n \sum_{j = 1}^n d(s^W_j) \otimes g^{VW}_{ij} \cdot e^V_i \\
=& \sum_{i = 1}^n \sum_{j = 1}^n d(g^{WV}_{ij}) \cdot s^W_j
                                 \otimes g^{VW}(e^W_i)
   + \sum_{j = 1}^n d(s^W_j) \otimes
                    \br{\sum_{i = 1}^n g^{VW}_{ij} \cdot e^V_i} \\
=& (\id_K \otimes g^{VW})\br{
    \sum_{i = 1}^n \sum_{j = 1}^n d(g^{WV}_{ij}) \cdot s^W_j \otimes e^W_i}
   + \sum_{j = 1}^n d(s^W_j) \otimes e^W_j \\
=& ((\id_K \otimes g^{VW})d(g^{WV}) + d^W)(s)
\end{align*}
\end{proof}

\begin{prop}\label{prop:conn-form-compat}
In the context of \cref{defn:conn-form}, if $(W, \phi_W)$ is another local trivialization with $A^W$
the connection form of $\nabla$ over $W$, then the following identity holds:
\[
A^W = \lambda \cdot (g^{VW} \otimes \id_K) \circ d(g^{WV}) + A^V
\]
\end{prop}
\begin{proof}
Using \cref{prop:diff-frame-change}, we observe that:
\begin{align*}
     & (\lambda \cdot d^V + A^V)(s) = (\lambda \cdot d^W + A^W)(s) = \nabla(s) \\
\iff & A^W(s) = (\lambda \cdot (d^V - d^W)  + A^V)(s) \\
\iff & A^W(s) = (\lambda \cdot (g^{VW} \otimes \id_K)d(g^{WV}) + A^V)(s)
\end{align*}
\end{proof}

\begin{defn}[Descent Data for Connections]\label{defn:conn-descent-data}
Let $(Y, R, K, d)$ be a differential space and $E$, a locally free $R$--module of finite rank $n$.
Let $Y = \bigcup_{\alpha \in I} Y_\alpha$ with local trivializations $(Y^\alpha, \phi^\alpha)$
of $E$. Let be the corresonding frames $\set{e^\alpha_i}_{i = 1}^n$ for $E|_{Y_\alpha}$.
Denote the transition functions $g^{Y_\alpha, Y_\beta}$ as $g^{\alpha\beta}$ and the $(i, j)$--entry
of the matrix of such a transition function in the $\set{e^\beta_i}_{i = 1}^n$ frame as
$g^{\alpha\beta}_{ij}$.
Consider sections $A^\alpha_{ij} \in K(Y_\alpha)$ and let
$A^\alpha : E|_{Y_\alpha} \to (E \otimes K)|_{Y_\alpha}$ be $R|_{Y_\alpha}$--linear maps
defined by the formula:
\[
A^\alpha\br{\sum_{i = 1}^n s^\alpha_i e^\alpha_i} :=
    \sum_{i = 1}^n e^\alpha_i \otimes \br{\sum_{j = 1}^n A^\alpha_{ij} \wedge s^\alpha_j}
\]
That is, $[A^\alpha_{ij}]$ is the matrix of
$A^\alpha$ in the $\set{e^\alpha_i}_{i = 1}^n$ frame.
Suppose, for all $\alpha, \beta \in I$, $i, j = 1, \dots, n$,
the following identity holds over $Y_\alpha \cap Y_\beta$:
\[
A^\beta = \lambda \cdot (g^{\alpha\beta} \otimes \id_K) \circ d(g^{\beta\alpha}) + A^\alpha
\]
Then, we call the collection $\set[(\phi^\alpha, A^\alpha_{ij})]{\alpha \in I, i, j = 1, \dots, n}$
a descent datum for a $\lambda$--$d$--connection on $E$.
\end{defn}

\begin{rmk}\label{rmk:conn-descent-data-matrix}
In the context of \cref{defn:conn-descent-data},
considering the $(i, j)$--entries of the matrix of
$(g^{\alpha\beta} \otimes \id_K) \circ (\lambda \cdot d(g^{\beta\alpha})
+ A^\alpha \circ g^{\beta\alpha})$ in the $\set{e^\beta_k}_{k = 1}^n$ frame, the identity involving
$A^\alpha, A^\beta$ is precisely the entry-wise identity of sections:
\[
A^\beta_{ij}
= \lambda \cdot \sum_{k = 1}^n g^{\alpha\beta}_{ik} \wedge d(g^{\beta\alpha}_{kj})
  + \sum_{k = 1}^n \sum_{l = 1}^n
        g^{\alpha\beta}_{ik} \wedge A^\alpha_{kl} \wedge g^{\beta\alpha}_{lj}
\]
\end{rmk}

\begin{prop}\label{prop:conn-form-coll}
In the context of \cref{defn:conn-descent-data}, there exists a unique $\lambda$--$d$--connection
$\nabla : E \to E \otimes K$ whose connection form over $Y_\alpha$ is $A^\alpha$ for all
$\alpha \in I$.
\end{prop}
\begin{proof}
We will show that the maps
$\nabla^\alpha := \lambda d^\alpha + A^\alpha : E|_{Y_\alpha} \to (E \otimes K)|_{Y_\alpha}$
are $\lambda|_{Y_\alpha}$--$d|_{Y_\alpha}$--connections which agree on the intersections
$Y_{\alpha\beta} := Y_\alpha \cap Y_\beta$, and then apply \cref{cor:conn-coll}, which guarantees
both existence and uniqueness. The $\bF$--linearity follows from those of $\lambda \cdot d^\alpha$
and $A^\alpha$ separately. For the Leibniz rule, we observe that, for all
open $U \subset Y_\alpha$, $r, s^\alpha_i \in R(U), i = 1, \dots, n$ with
$s = \sum_{i = 1}^n s^\alpha_i e^\alpha_i \in E(U)$, we have:
\begin{align*}
 & \nabla^\alpha(r \cdot s) \\
=& (\lambda \cdot d^\alpha + A^\alpha)(s) \\
=& \lambda \cdot \sum_{i = 1}^n e^\alpha_i \otimes d(r \cdot s^\alpha_i)
    + r \cdot A^\alpha(s) \\
=& \lambda \cdot \sum_{i = 1}^n e^\alpha_i \otimes (d(r) \cdot s^\alpha_i + r \cdot d(s^\alpha_i))
    + r \cdot A^\alpha(s) \\
=& \lambda \cdot \sum_{i = 1}^n s^\alpha_i e^\alpha_i \otimes d(r)
    + r \cdot \lambda \cdot \sum_{i = 1}^n e^\alpha_i \otimes d(s^\alpha_i)
    + r \cdot A^\alpha(s) \\
=& \lambda \cdot s \otimes d(r)
    + r \cdot \br{\lambda \cdot d^\alpha(s)
    + A^\alpha(s)} \\
=& \lambda \cdot s \otimes d(r)
    + r \cdot \nabla^\alpha(s)
\end{align*}
We then need to show that for all $\alpha, \beta \in I$, all open
$W \subset Y_{\alpha\beta}$, and all $s \in E(W)$,
\begin{align*}
     & (\lambda d^\alpha + A^\alpha)(s) = (\lambda d^\beta + A^\beta)(s) \\
\iff & A^\beta(s) = (\lambda (d^\alpha - d^\beta) + A^\alpha)(s) \\
\iff & A^\beta(s) = (\lambda \cdot (g^{\alpha\beta} \otimes \id_K)d(g^{\beta\alpha}) + A^\alpha)(s)
\end{align*}
which is precisely the hypothesis that the matrix entries of $A^\alpha, A^\beta$ in the
$\set{e^\alpha_i}_i, \set{e^\beta_i}_i$ frames respectively are part of a descent datum.
\end{proof}

\begin{defn}[Connection on Higher Forms]\label{defn:conn-higher-forms}
Let $(Y, R, K, d)$ be a differential space with a locally free sheaf $E$ equipped with a
$\lambda$--$d$--connection $\nabla$. Define a map:
\[
\nabla^p_0 : E \times K^{\wedge p} \to E \otimes K^{\wedge (p + 1)}
           : (s, \omega) \mapsto \lambda \cdot s \otimes d(\omega) + \nabla(s) \wedge \omega
\]
where $\nabla(s) \wedge \omega$ is the image of
$\nabla(s) \otimes \omega \in (E \otimes K \otimes K^{\wedge p})(Y')$ for some open $Y' \subset Y$
under the map $\id_E \otimes (- \wedge -)$. We can check that the map $\nabla^p_0$ is additive
and $R$--balanced, providing a unique map $\nabla^p$ of sheaves of $\bF$--modules making the
following diagram commute, by the universal property of the tensor product:
\[\begin{tikzcd}
E \times K^{\wedge p} \ar[r, "\nabla^p_0"] \ar[d, "\otimes" left] & E \otimes K^{\wedge (p + 1)} \\
E \otimes K^{\wedge p} \ar[ru, "\nabla^p" below right]
\end{tikzcd}\]
This yields a diagram of sheaves of $\bF$--modules:
\[
0 \to E \to[\nabla^0] E \otimes K \to[\nabla^1] E \otimes K^{\wedge 2} \to[\nabla^2] \cdots
\]
We call the $\nabla^p$, the connections on higher forms defined by $\nabla$.
\end{defn}

\begin{prop}\label{prop:conn-morphism-to-diagram-morphism}
Let $(Y, R, K, d)$ be a differential space with finite rank locally free $R$--modules $E, F$
equipped with $\lambda$--$d$--connections $\nabla_E, \nabla_F$. Then, any morphism of connections
$f : \nabla_E \to \nabla_F$ extends to a morphism of the diagrams of sheaves of $\bF$--modules
of \cref{defn:conn-higher-forms}:
\[
f^\bullet = (f \otimes \id_{K^{\wedge \bullet}}) : (E \otimes K^{\wedge \bullet}, \nabla_E^\bullet)
\to (F \otimes K^{\wedge \bullet}, \nabla_F^\bullet)
\]
\end{prop}
\begin{proof}
We first show that the following diagram commutes:
\[\begin{tikzcd}
E \times K^{\wedge k} \ar[r, "\nabla^k_{E, 0}"] \ar[d, "f \times \id_{K^{\wedge k}}" left] &
E \otimes K^{\wedge (k + 1)} \ar[d, "f \otimes \id_{K^{\wedge (k + 1)}}"] \\
F \times K^{\wedge k} \ar[r, "\nabla^k_{F, 0}" below] & F \otimes K^{\wedge (k + 1)}
\end{tikzcd}\]
On a local section $(s, \omega)$, we have:
\begin{align*}
(f \otimes \id_{K^{\wedge (k + 1)}})(\nabla^k_{E, 0}(s, \omega))
=& (f \otimes \id_{K^{\wedge (k + 1)}})(\lambda s \otimes d(\omega) + \nabla_E(s) \wedge \omega) \\
=& \lambda f(s) \otimes d(\omega)
   + (f \otimes \id_{K^{\wedge (k + 1)}})(\nabla_E(s) \wedge \omega)
\end{align*}
On the other hand:
\begin{align*}
\nabla^k_{F, 0}((f \times \id)(s, \omega))
=& \nabla^k_{F, 0}(f(s), \omega) \\
=& \lambda f(s) \otimes d(\omega) + \nabla_F(f(s)) \wedge \omega \\
=& \lambda f(s) \otimes d(\omega) + (f \otimes \id_K)(\nabla_E(s)) \wedge \omega
\end{align*}
We can then show that:
\[
(f \otimes \id_K)(\nabla_E(s)) \wedge \omega
= (f \otimes \id_{K^{\wedge (k + 1)}})(\nabla_E(s) \wedge \omega)
\]
To see this, consider the connection matrix $[A_{ij}]$ of $\nabla_E$ in a frame $\set{e_i}_i$ over
an open $W \subset Y$. Let $s|_W = \sum_i s_i e_i$ for $s_i \in R(W)$. Then,
\begin{align*}
 & ((f \otimes \id_K)(\nabla_E(s)) \wedge \omega)|_W \\
=& (f \otimes \id_K)\br{\sum_i e_i \otimes d(s_i) + \sum_{i} e_i \otimes \br{\sum_j A_{ij}s_j}}
   \wedge \omega \\
=& \sum_i f(e_i) \otimes d(s_i) \wedge \omega
    + \sum_{i} f(e_i) \otimes \br{\sum_j A_{ij}s_j \wedge \omega} \\
=& (f \otimes \id_{K^{\wedge (k + 1)}})\br{\sum_i f(e_i) \otimes d(s_i) \wedge \omega
    + \sum_{i} f(e_i) \otimes \br{\sum_j A_{ij}s_j \wedge \omega}} \\
=& (f \otimes \id_{K^{\wedge (k + 1)}})\br{\br{\sum_i f(e_i) \otimes d(s_i)
    + \sum_{i} f(e_i) \otimes \br{\sum_j A_{ij}s_j}} \wedge \omega} \\
=& (f \otimes \id_{K^{\wedge (k + 1)}})\br{\nabla_E(s) \wedge \omega}|_W
\end{align*}
Since the two sides agree in all local frames, they must agree completely.
The universal property of tensor products now shows that the following diagram commutes:
\[\begin{tikzcd}
E \otimes K^{\wedge k} \ar[r, "\nabla^k_{E}"] \ar[d, "f \otimes \id_{K^{\wedge k}}" left] &
E \otimes K^{\wedge (k + 1)} \ar[d, "f \otimes \id_{K^{\wedge (k + 1)}}"] \\
F \otimes K^{\wedge k} \ar[r, "\nabla^k_{F}" below] & F \otimes K^{\wedge (k + 1)}
\end{tikzcd}\]
and we are done.
\end{proof}

\begin{defn}[Curvature]\label{defn:curvature}
In the context of \cref{defn:conn-higher-forms}, the sheaf morphisms
$\nabla^{p + 1} \circ \nabla^{p}$ are called the curvatures of
$\nabla$. We say that $\nabla$ is flat its curvatures are all zero. When $\nabla$ is flat, we have a
complex of sheaves of $\bF$--modules:
\[
0 \to E \to[\nabla^0] E \otimes K \to[\nabla^1] E \otimes K^{\wedge 2} \to[\nabla^2] \cdots
\]
called the connection complex of $\nabla$.
\end{defn}

\begin{rmk}\label{rmk:0-d-conn-curvature}
In the contexts of \cref{defn:conn-higher-forms} and \cref{defn:curvature}, if $\lambda = 0$
or $d = 0$, then $\nabla^k_0(s, \omega) = \nabla(s) \wedge \omega$ regardless of the choice of
$\lambda$ or $d$, which implies that $\nabla^k$ is also independent of these choices.
Then, the curvatures of $\nabla$ are also indpendent of these choices.
\end{rmk}

\begin{rmk}\label{rmk:conn-higher-forms-loc}
In the context of \cref{defn:conn-higher-forms}, if $(V, \phi_V)$ is a local trivialization of
$E \otimes K^{\wedge p}$ with a frame $\set{e^V_i \otimes \omega^V_j}_{i, j}$ and
$[A^V_{ij}]$ is the matrix of the connection form $A^V$ of $\nabla$ in this frame,
then, for any local section $s \in (E \otimes K^{\wedge p})(V)$, there are sections
$r^V_{ij} \in R(V)$ such that $s = \sum_{i, j} s^V_{ij} e^V_i \otimes \omega^V_j$.
Let $\tau^V_i := \sum_{j} r^V_{ij}\omega^V_j$ for brevity.
By unwrapping the definitions, we see that:
\begin{align*}
 & \nabla^p\br{\sum_i e^V_i \otimes \tau^V_i} \\
=& \lambda \cdot \sum_i e^V_i \otimes
     d(\tau^V_i)
     + \sum_i \nabla(e^V_i) \wedge \tau^V_i \\
=& \lambda \cdot \sum_i e^V_i \otimes
     d(\tau^V_i)
     + \sum_i \br{\sum_{k} A^V_{ki} \otimes e^V_k} \wedge \tau^V_i \\
=& \sum_i e^V_i \otimes \br{
     \lambda \cdot d(\tau^V_i)
     + \br{\sum_j A^V_{ij} \wedge \tau^V_j}
   }
\end{align*}
where the second-last step is obtained by collecting terms in the previous step.
Write $\rho^V_i := \lambda \cdot d(\tau^V_i) + \sum_{j} A^V_{ij} \wedge \tau^V_j$ for brevity again.
We can then compute curvatures:
\begin{align*}
 & \nabla^{p + 1}\br{\nabla^p\br{\sum_{i} e^V_i \otimes \tau^V_i}} \\
=& \nabla^{p + 1}\br{
     \sum_i e^V_i \otimes \rho^V_i
   } \\
=& \sum_i e^V_i \otimes \br{
     \lambda \cdot d(\rho^V_i)
     + \sum_j A^V_{ij} \wedge \rho^V_j
   }
\end{align*}
We further compute:
\begin{align*}
 & \lambda \cdot d(\rho^V_i) \\
=& \lambda \cdot d\br{\lambda \cdot d(\tau^V_i) + \sum_{j} A^V_{ij} \wedge \tau^V_j} \\
=& \lambda \cdot d(\lambda) \wedge d(\tau^V_i) + \lambda^2 \cdot d(d(\tau^V_i))
   + \lambda \cdot \sum_{j} \br{d(A^V_{ij}) \wedge \tau^V_j - A^V_{ij} \wedge d(\tau^V_j)} \\
=& \lambda \cdot d(\lambda) \wedge d(\tau^V_i)
   + \lambda \cdot \sum_{j} d(A^V_{ij}) \wedge \tau^V_j
   - \lambda \cdot \sum_{j} A^V_{ij} \wedge d(\tau^V_j)
\end{align*}
and:
\begin{align*}
 & \sum_j A^V_{ij} \wedge \rho^V_j \\
=& \sum_j A^V_{ij} \wedge \br{\lambda \cdot d(\tau^V_j) + \sum_{k} A^V_{jk} \wedge \tau^V_k} \\
=& \lambda \cdot \sum_j A^V_{ij} \wedge d(\tau^V_j)
   + \sum_j \sum_{k} A^V_{ij} \wedge A^V_{jk} \wedge \tau^V_k
\end{align*}
Putting these together, we get:
\begin{align*}
 & \nabla^{p + 1}\br{\nabla^p\br{\sum_{i} e^V_i \otimes \tau^V_i}} \\
=& \sum_i e^V_i \otimes \br{
     \lambda \cdot d(\lambda) \wedge d(\tau^V_i)
   + \lambda \cdot \sum_{j} d(A^V_{ij}) \wedge \tau^V_j
   + \sum_j \sum_{k} A^V_{ij} \wedge A^V_{jk} \wedge \tau^V_k
   } \\
=& \sum_i e^V_i \otimes \br{
   \lambda \cdot  d(\lambda) \wedge d(\tau^V_i)
   + \sum_j \br{\lambda \cdot d(A^V_{ij}) + \sum_{k} A^V_{ik} \wedge A^V_{kj}}
     \wedge \tau^V_j
   }
\end{align*}
where the last is obtained by collecting like terms in the previous step again.
\end{rmk}

\begin{prop}\label{prop:curvature-is-linear}
In the context of \cref{defn:conn-higher-forms}, if $\lambda \in \ker{d}$, then the curvature maps
$\nabla^{p + 1} \circ \nabla^p$, for each $p = 0, 1, 2, \dots$, are:
\begin{enumerate}
\item $R$--linear, and
\item zero when $\nabla^1 \circ \nabla^0$ is.
\end{enumerate}
In particular, $\nabla$ is flat if and only if point (ii) above holds.
\end{prop}
\begin{proof}
When $\lambda \in \ker{d}$, the last equation of \cref{rmk:conn-higher-forms-loc} becomes:
\begin{align*}
 & \nabla^{p + 1}\br{\nabla^p\br{\sum_{i} e^V_i \otimes \tau^V_i}} \\
=& \sum_i e^V_i \otimes \br{
     \sum_j \br{\lambda \cdot d(A^V_{ij}) + \sum_{k} A^V_{ik} \wedge A^V_{kj}}
     \wedge \tau^V_j
   }
\end{align*}
and it is clear that this expression is $R$--linear. When $p = 0$, we have
$\tau^V_j \in R(V)$ and taking $\tau^V_i = \delta_{ij}$ for $j \in \set{1, \dots, n}$ along with the
assumption that $\nabla^1 \circ \nabla^0 = 0$ yields:
\[
\lambda \cdot d(A^V_{ij}) + \sum_{k} A^V_{ik} \wedge A^V_{kj} = 0
\]
for all $i, j$. Hence, for $p > 0$ as well, $\nabla^{p + 1} \circ \nabla^p$ is zero in local frames
and hence globally.
\end{proof}

\subsection{Constructions Involving Connections}
\label{subsec:cons-conn}

\begin{prop}[Sum of Connections]\label{prop:conn-sum}
Let $(Y, R)$ be an $\bF$--space with finite rank locally free $R$--modules $E, K$.
Let $d_1, d_2$ be two differentials on $K$ and $\nabla_1, \nabla_2$ be
$\lambda_1$--$d_1$-- and $\lambda_2$--$d_2$--connections on $E$ respectively.
Then, the following hold:
\begin{enumerate}
\item $\nabla_1 + \nabla_2$ is a $1$--$(\lambda_1d_1 + \lambda_2d_2)$--connection on $E$.

\item $(\nabla_1 + \nabla_2)^k = \nabla_1^k + \nabla_2^k$, where the superscripts indicate
extensions to $E \otimes K^{\wedge k}$.

\item $(\nabla_1 + \nabla_2)^1 \circ (\nabla_1 + \nabla_2)^0
= \nabla_1^1 \circ \nabla_1^0 + \nabla_1^1 \circ \nabla_2^0 + \nabla_2^1 \circ \nabla_1^0
    + \nabla_2^1 \circ \nabla_2^0$

\item If $\nabla_1$ and $\nabla_2$ are flat, then $\nabla_1 + \nabla_2$ is flat if and only
if $\nabla_1^1 \circ \nabla_1^0 + \nabla_0^1 \circ \nabla_2^0 = 0$.
\end{enumerate}
\end{prop}
\begin{proof}
(i) is a straightforward verification of the defining properties of a connection.

(ii) We observe that:
\begin{align*}
 & s \otimes (\lambda_1d_1 + \lambda_2d_2)(\omega) + (\nabla_1 + \nabla_2)(s) \wedge \omega \\
=& (\lambda_1 s \otimes d_1(\omega) + \nabla_1(s) \wedge \omega)
  + (\lambda_2 s \otimes d_2(\omega) + \nabla_2(s) \wedge \omega)
\end{align*}
Now, by the definition of extensions of connections to $E \otimes K^{\wedge \bullet}
($\cref{defn:conn-higher-forms}) and the universal property of tensor products, we have:
\[
(\nabla_1 + \nabla_2)^k = \nabla_1^k + \nabla_2^k
\]

(iii), (iv) These are straightforward verifications, given (ii).
\end{proof}

\begin{prop}[Scaling Connections]\label{prop:conn-scale}
Let $(Y, R)$ be an $\bF$--space with finite rank locally free $R$--modules $E, K$.
Let $d$ be a differential on $K$ and $\nabla$, $\lambda$--$d$--connection on $E$.
Let $r \in R(Y)$ be a global section of $R$. Then the following hold:
\begin{enumerate}
\item $r \cdot \nabla$ is an $r\lambda$--$d$--connection.
In particular, any $\lambda$--$d$--connection is also a $1$--$\lambda \cdot d$--connection
and vice versa.

\item $(r \cdot \nabla)^k = r \cdot \nabla^k$

\item
$(r \cdot \nabla)^1 \circ (r \cdot \nabla)^0 = r\lambda \cdot \id_E \otimes d(r)
                                               + r^2 \cdot \nabla^1 \circ \nabla^0$

\item If $r \in \ker{d}$ and $\nabla$ is flat, then so is $r \cdot \nabla$.
\end{enumerate}
\end{prop}
\begin{proof}
(i) is a straightforward verification of the defining properties of a connection.

(ii) We observe that:
\[
r\lambda s \otimes d(\omega) + r\nabla(s) \wedge \omega
= r(\lambda s \otimes d(\omega) + \nabla(s) \wedge \omega)
\]
By the definition of $(r \cdot \nabla)^k, \nabla^k$ (\cref{defn:conn-higher-forms}) and the
universal property of tensor products, we have $(r \cdot \nabla)^k = r \cdot \nabla^k$.

(iii), (iv) These are straightforward verifications, given (ii).
\end{proof}

\begin{cor}[Module of Connections]
Let $(Y, R, K, d)$ be an differential space with a finite rank locally free $R$--module $E$.
Then, the set:
\[
C(E) = \set[(\lambda, \nabla)]{
    \lambda \in R(Y),
    \nabla \text{ is a } \lambda \text{--} d \text{--connection on } E
}
\]
is an $R(Y)$--module. In particular, it is an $\bF$--module.
\end{cor}

\begin{defn}[Dual Connection]\label{defn:dual-conn}
Let $(Y, R, K, d)$ be a differential space with a locally free $R$--module $E$ equipped with a
$\lambda$--$d$--connection $\nabla$. We define the dual of $\nabla$ to be the operator
$\nabla^\vee : E^\vee \to E^\vee \otimes K \cong \HHom(E, K)$ defined, for any opens
$Y'' \subset Y' \subset Y$, $\eta \in E^\vee(Y')$, $s \in E(Y'')$, by:
\[
\nabla^\vee(\eta) : E|_{Y''} \to K|_{Y''}
    : s \mapsto \lambda d(\eta(s)) - (\eta \otimes \id_K)(\nabla(s))
\]
\end{defn}

\begin{cor}\label{cor:dual-conn}
$\nabla^\vee$ is a $\lambda$--$d$--connection on $E^\vee$.
\end{cor}
\begin{proof}
$\bF$--linearity is immediate it is a difference of $\bF$--linear maps.
For the Leibniz rule, we observe that for any open $Y' \subset Y, \eta \in E^\vee(Y'),
s \in E(Y'), r \in R(Y')$:
\begin{align*}
\nabla^\vee(r\eta)(s)
=& \lambda d(r\eta(s)) - (r\eta \otimes \id_K)(\nabla(s)) \\
=& \lambda d(r) \cdot \eta(s) + \lambda r d(\eta(s)) - r(\eta \otimes \id_K)(\nabla(s)) \\
=& \lambda \eta(s)d(r) + r \nabla^\vee(s) \\
=& \lambda (\eta \otimes d(r))(s) + r \nabla^\vee(s) \\
=& (\lambda \eta \otimes d(r) + r \nabla^\vee)(s)
\end{align*}
\end{proof}

\begin{rmk}\label{rmk:dual-conn-matrix}
In the context of \cref{defn:dual-conn}, let $\set{e_i}_i$ be a frame for $E$ over some open subset
with dual frame $\set{\epsilon_i}_i$. Let $[A_{ij}]$ be the connection matrix of $\nabla$ in the
$\set{e_i}_i$ frame. Then, we observe that:
\[
\nabla^\vee(\epsilon_j)(e_i)
= \lambda d(\epsilon_j(e_i)) - \sum_{k} A_{ki} \otimes \epsilon_j(e_k)
= \lambda d(\delta_{ij}) - A_{ji}
= -A_{ji}
\]
This shows that the connection matrix of $\nabla^\vee$ in the $\set{\epsilon_i}_i$ frame
is the negative of the transpose of that of $\nabla$ in the $\set{e_i}_i$ frame.
\end{rmk}

\begin{prop}\label{prop:dual-conn-flat}
When $\lambda \in \ker{d}$, $\nabla^\vee$ is flat if $\nabla$ is.
\end{prop}
\begin{proof}
By \cref{rmk:dual-conn-matrix} and \cref{rmk:conn-higher-forms-loc} (dropping the $V$--superscripts
from the differential forms for convenience), we have:
\begin{align*}
 & (\nabla^{\vee, 1} \circ \nabla^{\vee, 0})\br{\sum_{i} s_i \epsilon_i } \\
=& \sum_{i} \epsilon_i \otimes
    \br{\sum_{j} \br{\lambda d(-A_{ji}) + \sum_k (-A_{ki}) \wedge (-A_{jk})}s_j} \\
=& - \sum_{i} \epsilon_i \otimes
    \br{\sum_{j} \br{\lambda d(A_{ji}) + \sum_k A_{jk} \wedge A_{ki}}s_j}
\end{align*}
The flatness of $\nabla$ shows that:
\[
\lambda d(A_{ji}) + \sum_k A_{jk} \wedge A_{ki} = 0
\]
similar to the proof of \cref{prop:curvature-is-linear},
since it is the $(j, i)$--entry of the matrix of $\nabla^1 \circ \nabla$ in the $\set{e_i}_i$ frame.
We can then apply \cref{prop:curvature-is-linear} to deduce that $\nabla^{p + 1} \circ \nabla^p = 0$
for $p > 0$.
\end{proof}

\begin{prop}\label{prop:dual-conn-sum}
In the context of \cref{defn:dual-conn}, consider two $\lambda_i$--$d_i$--connections
$\nabla_i, i = 1, 2$ on $E$. Then, the following identity of
$1$--$(\lambda_1d_1 + \lambda_2d_2)$--connections holds:
\[
(\nabla_1 + \nabla_2)^\vee = \nabla_1^\vee + \nabla_2^\vee
\]
\end{prop}
\begin{proof}
By definition:
\begin{align*}
 & (\nabla_1 + \nabla_2)^\vee(\eta)(s) \\
=& \lambda_1d_1(\eta(s)) + \lambda_2d_2(s) - (\eta \otimes \id_K)((\nabla_1 + \nabla_2)(s)) \\
=& \lambda_1d_1(\eta(s)) + \lambda_2d_2(s) - (\eta \otimes \id_K)(\nabla_1(s))
   - (\eta \otimes \id_K)((\nabla_2)(s)) \\
=& \lambda_1d_1(\eta(s)) - (\eta \otimes \id_K)(\nabla_1(s))
   + \lambda_2d_2(s) - (\eta \otimes \id_K)((\nabla_2)(s)) \\
=& \nabla_1^\vee(\eta)(s) + \nabla_2^\vee(\eta)(s) \\
=& (\nabla_1^\vee + \nabla_2^\vee)(\eta)(s)
\end{align*}
\end{proof}

\begin{defn}[Tensor Product of Connections]\label{defn:tensor-conn}
Let $\nabla$ and $\nabla'$ be $\lambda$--$d$--connections on $R$--modules $E$ and $F$ of finite
ranks $n, m$ respectively over a differential space $(Y, R, K, d)$.
By passing to a refinement if necessary, let $Y = \bigcup_{\alpha \in I} Y_\alpha$ be an open cover
over which both $E$ and $F$ trivialize.
Let $A^\alpha$ and $B^\alpha$ be the connection forms of $E$ and $F$ respectively over $Y_\alpha$.
Consider the maps
$C^\alpha := A^\alpha \otimes \id_F + \id_E \otimes B^\alpha : (E \otimes F)|_{Y_\alpha} \to
(E \otimes F \otimes K)|_{Y_\alpha}$.
Now, let $g^{\alpha\beta}$ and $h^{\alpha\beta}$ be the
transition functions of $E$ and $F$ respectively over $Y_\alpha \cap Y_\beta$
so that $\gamma^{\alpha\beta} := g^{\alpha\beta} \otimes h^{\alpha\beta}$ are the transition
functions of $E \otimes F$. It is immediate from the definitions that:
\[
d(\gamma^{\beta\alpha})
= d(g^{\beta\alpha}) \otimes h^{\beta\alpha} + g^{\beta\alpha} \otimes d(h^{\beta\alpha})
\]
so that:
\[
(\gamma^{\alpha\beta} \otimes \id_K)d(\gamma^{\beta\alpha})
= d(g^{\beta\alpha}) \otimes \id_F + \id_E \otimes d(h^{\beta\alpha})
\]
Using this, we can check that
$C^\beta = \lambda \cdot (\gamma^{\alpha\beta} \otimes \id_K)d(\gamma^{\beta\alpha}) + C^\alpha$
so that by \cref{prop:conn-form-coll}, the $C^\alpha$ yield a unique
$\lambda$--$d$--connection $\nabla \otimes \nabla' : E \otimes F \to E \otimes F \otimes K$
whose connection forms are the $C^\alpha$. We will call this the tensor product connection of
$\nabla$ and $\nabla'$.
\end{defn}

\begin{rmk}\label{rmk:tensor-conn-matrix}
Consider local frames $\set{e_i}_i, \set{f_j}_j$ for $E, F$ respectively such that
$\nabla, \nabla'$ have connection forms $A, B$ with connection matrices $[A_{ij}], [B_{kl}]$ in
these frames respectively. Let $C$ be the connection form of $\nabla \otimes \nabla'$ in the
$\set{e_i \otimes f_k}_{i, k}$ frame.
We observe that:
\[
C(e_j \otimes f_l) = A(e_j) \otimes f_l + e_j \otimes B(f_l)
= \sum_{i} e_i \otimes f_l \otimes A_{ij} + \sum_{k} e_j \otimes f_k \otimes B_{kl}
= \sum_{i, k} e_i \otimes f_k \otimes (\delta_{kl}A_{ij} + \delta_{ij}B_{kl})
\]
Let $C_{ijkl} := (\delta_{kl}A_{ij} + \delta_{ij}B_{kl})$.
By \cref{rmk:conn-higher-forms-loc}, we have:
\begin{align*}
 & (\nabla \otimes \nabla')^1((\nabla \otimes \nabla')^0(e_j \otimes f_l)) \\
=& \sum_{i, k} e_i \otimes f_k \otimes
     \sum_{j, l} \br{\lambda d(C_{ijkl}) + \sum_{p, q} C_{ipkq} \wedge C_{pjql}} \\
=& \sum_{i, k} e_i \otimes f_k \otimes
     \sum_{j, l} \br{\lambda \br{\delta_{kl}d(A_{ij}) + \delta_{ij}d(B_{kl})}
     + \sum_{p, q} (\delta_{kq} A_{ip} + \delta_{ip}B_{kq})
                  \wedge (\delta_{ql}A_{pj} + \delta_{pj} B_{ql})}
\end{align*}
By distributing the wedge products and sums, and taking only those summands
for which $\delta_{ab} = 1$ for various choices of $a, b$, this simplifies to:
\begin{align*}
 & (\nabla \otimes \nabla')^1((\nabla \otimes \nabla')^0(e_j \otimes f_l)) \\
=& \sum_{i, k} e_i \otimes f_k \otimes
     \sum_{j, l} \br{\delta_{kl}\br{\lambda d(A_{ij}) + \sum_{p} A_{ip} \wedge A_{pi}}
                 + \delta_{ij}\br{\lambda d(B_{kl}) + \sum_{q} B_{kq} \wedge B_{ql}}} \\
=& \sum_{i} \br{\br{\lambda d(A_{ij}) + \sum_{p} A_{ip} \wedge A_{pi}} \otimes e_i} \otimes f_k
   + \sum_{k} e_j \otimes
              \br{\br{\lambda d(B_{kl}) + \sum_{q} B_{kq} \wedge B_{ql}} \otimes f_{k}} \\
=& \nabla^1(\nabla^0(e_j)) \otimes f_k
   + e_j \otimes (\nabla')^1((\nabla')^0(f_{k}))
\end{align*}
That is, we have the following equality of $R$--linear maps:
\[
(\nabla \otimes \nabla)^1 \circ (\nabla \otimes \nabla')^0
= (\nabla^1 \circ \nabla^0) \otimes \id_F + \id_E \otimes ((\nabla')^1 \circ (\nabla')^0)
\]
\end{rmk}

\begin{cor}\label{cor:tensor-conn-flat}
$\nabla \otimes \nabla'$ is flat if both $\nabla$ and $\nabla'$ are.
\end{cor}
\begin{proof}
Apply \cref{rmk:tensor-conn-matrix} along with \cref{prop:curvature-is-linear}.
\end{proof}

\begin{prop}\label{prop:tensor-conn-elem}
In the context of \cref{defn:tensor-conn}, if $e \otimes f \in (E \otimes F)(U)$ is an elementary
tensor over an open subset $U \subset Y$, then
$(\nabla \otimes \nabla')(e \otimes f) = \nabla(e) \otimes f + e \otimes \nabla'(f)$.
\end{prop}
\begin{proof}
Consider an open cover $U = \bigcup_{\alpha} U_\alpha$ with local frames
$\set{e^\alpha_i}_i, \set{f^\alpha_j}_j$ such that
$e^\alpha := e|_{U_\alpha} = \sum_i s^\alpha_i e^\alpha_i,
f^\alpha := f|_{U_\alpha} = \sum_j t^\alpha_j f^\alpha_j$.
We observe that:
\begin{align*}
 & (\nabla \otimes \nabla')(e \otimes f)|_{U_\alpha} \\
=& (\nabla \otimes \nabla')((e \otimes f)|_{U_\alpha}) \\
=& d^\alpha\br{\sum_{i, j} s^\alpha_i t^\alpha_j e^\alpha_i \otimes f^\alpha_j}
   + C^\alpha(e^\alpha \otimes f^\alpha) \\
=& \sum_{i, j} d(s^\alpha_i t^\alpha_j) e^\alpha_i \otimes f^\alpha_j
   + C^\alpha(e^\alpha \otimes f^\alpha) \\
=& d^\alpha(e^\alpha) \otimes f^\alpha + e^\alpha \otimes d^\alpha(f^\alpha)
   +  A^\alpha(e^\alpha) \otimes f^\alpha + e^\alpha \otimes B^\alpha(f^\alpha) \\
=& (d^\alpha + A^\alpha)(e^\alpha) \otimes f^\alpha
   +  e^\alpha \otimes (d^\alpha + B^\alpha)(f^\alpha) \\
=& \nabla(e^\alpha) \otimes f^\alpha
   +  e^\alpha \otimes \nabla'(f^\alpha) \\
=& (\nabla(e) \otimes f + e \otimes \nabla'(f))|_{U_\alpha}
\end{align*}
This shows the two sides agree locally and hence, they also agree globally by the gluing
property of sections.
\end{proof}

\begin{prop}\label{prop:tensor-conn-sum}
Let $(Y, R)$ be an $\bF$--space with finite rank locally free $R$--modules $E_1, E_2, K$.
Consider two differentials $d_1, d_2$ on $K$, and for $i, j \in \set{1, 2}$,
$\lambda_j$--$d_j$--connections $\nabla_{i, j}$ on $E_i$.
Then, the following identity of $1$--$(\lambda_1d_1 + \lambda_2d_2)$--connections on
$E \otimes F$ holds:
\[
(\nabla_{1, 1} + \nabla_{1, 2}) \otimes (\nabla_{2, 1} + \nabla_{2, 2})
= \nabla_{1, 1} \otimes \nabla_{2, 1} + \nabla_{1, 2} \otimes \nabla_{2, 2}
\]
\end{prop}
\begin{proof}
We need to check that the types of connections on both sides of the identity align but this is
straightforward.
On elementary tensors $e \otimes f \in (E \otimes F)(U)$ over an open subset $U \subset Y$, by
\cref{prop:tensor-conn-elem}, we have:
\begin{align*}
 & ((\nabla_{1, 1} + \nabla_{1, 2}) \otimes (\nabla_{2, 1} + \nabla_{2, 2}))(e \otimes f) \\
=& (\nabla_{1, 1}(e) + \nabla_{1, 2}(e)) \otimes f
   + e \otimes (\nabla_{2, 1}(f) + \nabla_{2, 2}(f)) \\
=& \nabla_{1, 1}(e) \otimes f + \nabla_{1, 2}(e) \otimes f
   + e \otimes \nabla_{2, 1}(f) + e \otimes \nabla_{2, 2}(f) \\
=& \nabla_{1, 1}(e) \otimes f + e \otimes \nabla_{2, 1}(f)
   + \nabla_{1, 2}(e) \otimes f + e \otimes \nabla_{2, 2}(f) \\
=& (\nabla_{1, 1} \otimes \nabla_{2, 1})(e \otimes f)
   + (\nabla_{1, 2} \otimes \nabla_{2, 2})(e \otimes f) \\
=& (\nabla_{1, 1} \otimes \nabla_{2, 1} + \nabla_{1, 2} \otimes \nabla_{2, 2})(e \otimes f)
\end{align*}
This, in particular, shows that in a local frame, the two sides have the same connection forms
and are hence equal by \cref{prop:conn-form-coll}.
\end{proof}

\begin{prop}\label{prop:rest-conn-ext-prt}
In the context of \cref{prop:rest-diff-ext}, suppose we have a finite rank locally free $R$--module
$E$. For any morphisms $F : E \to E \otimes K$, $f : E \to E \otimes K^{j, 1 - j}$, denote the
composites $pr^{j, 1 - j} \circ F$ and $\iota^{j, 1 - j} \circ f$ by $F^{j, 1 - j}$ by
$f_{ext}$ respectively for $j = 0, 1$.
Then, the following hold:
\begin{enumerate}
\item If $\nabla$ is a $\lambda$--$d^{\bullet, 0}$--connection, then $\nabla_{ext}$
is a $\lambda$--$D^\bullet$--connection called the extension of $\nabla$ to $K$.
If $\nabla$ is flat, so is $\nabla_{ext}$.

\item If $\nabla$ is a $\lambda$--$D^\bullet$--connection, then $\nabla^{1, 0}$
is a $\lambda$--$d^{1, 0}$--connection called the restriction of $\nabla$ to $K^{1, 0}$.
If $\nabla$ is flat, so is $\nabla^{1, 0}$.

\item If $\nabla$ is a $\lambda$--$D^\bullet$--connection, then $\nabla^{0, 1}$ is
a $0$--$\ol{d}^{0, \bullet}$--connection, $\nabla^{0, 1}_{ext}$ is both a
$0$--$\ol{D}^\bullet$--connection and a $0$--$D^\bullet$--connection with extensions
$(\nabla^{0, 1}_{ext})^k$ to higher wedge powers (\cref{defn:conn-higher-forms}), independent of the
differential. If $\nabla$ is flat, then $\nabla^{0, 1}$ is flat as a
$0$--$\ol{d}^{0, \bullet}$--connection and $\nabla^{0, 1}_{ext}$ is flat as a
$0$--$D^\bullet$--connection and as a $0$--$\ol{D}^\bullet$--connection, and the following identity
holds:
\[
(\nabla^{1, 0}_{ext})^1 \circ (\nabla^{0, 1}_{ext})^0
    + (\nabla^{0, 1}_{ext})^1 \circ (\nabla^{1, 0}_{ext})^0 = 0
\]
where the superscripts are again as in \cref{defn:conn-higher-forms}.

\item If $\prt_E$ and $\theta$ are $\lambda$--$d^{\bullet, 0}$-- and
$0$--$\ol{d}^{0, \bullet}$--connections respectively, then
$\nabla = (\prt_E)_{ext} + \theta_{ext}$ is a $\lambda$--$D^\bullet$--connection. In addition,
$\nabla$ is flat if $\prt_E$ and $\theta$ are flat, and:
\[
(\prt_E)_{ext}^1 \circ \theta_{ext}^0 + \theta_{ext}^1 \circ (\prt_E)_{ext}^0 = 0
\]
where the superscripts are again as in \cref{defn:conn-higher-forms}.

\item There is a bijection between flat $\lambda$--$D^\bullet$--connections and
pairs $(\prt_E, \theta)$ satisfying the conditions of (iv).
\end{enumerate}
The analogous statements hold when we switch $D^\bullet, d^{\bullet, 0}, \iota^{1, 0}$ with
$\ol{D}^\bullet, \ol{d}^{0, \bullet}, \iota^{0, 1}$ respectively.
\end{prop}
\begin{proof}
(i) First observe that $\nabla_{ext}$ is $\bF$--linear because it is a composite of such maps.
Next, observe that:
\begin{align*}
\nabla_{ext}(rs)
=& (\id_E \otimes \iota^{1, 0})(\lambda s \otimes d^{1, 0}(r) + r\nabla(s)) \\
=& \lambda s \otimes \iota^{1, 0}(d^{1, 0}(r)) + r (\id_E \otimes \iota^{1, 0})(\nabla(s)) \\
=& \lambda s \otimes D^1(r) + r\nabla_{ext}(s)
\end{align*}
For flatness, we first recall how the extension of a connection to higher wedge powers is defined
(\cref{defn:conn-higher-forms}): the map $\nabla_{ext}^1$ is the unique map
$E \otimes K \to E \otimes K^{\wedge 2}$ provided by the universal property of tensor products
applied to the map:
\[
\nabla^1_{ext, 0} : E \times K \to E \otimes K^{\wedge 2}
                : (s, \omega) \mapsto s \otimes D(\omega) + \nabla_{ext}(s) \wedge \omega
\]
Then, by the fact that the wedge product of the direct summand $K^{1, 0}$ of $K$ is the wedge
product of $K$ restricted to $K^{1, 0}$, the following diagram commutes:
\[\begin{tikzcd}[column sep = huge]
E \times K^{1, 0} \ar[r, "\nabla(-) \otimes -" ] \ar[d, "\id_E \times \iota^{1, 0}" left] &
E \otimes K^{1, 0} \otimes K^{1, 0}
    \ar[r, "\id_E \otimes (- \wedge -)"]
    \ar[d, "\id_E \otimes \iota^{1, 0} \otimes \iota^{1, 0}" left] &
E \otimes (K^{1, 0})^{\wedge 2} \ar[d, "\id_E \otimes \iota^{2, 0}"] \\
E \times K \ar[r, "\nabla_{ext}(-) \otimes -" below] &
E \otimes K \otimes K \ar[r, "\id_E \otimes (- \wedge -)" below] &
E \otimes K^{\wedge 2}
\end{tikzcd}\]
This, in turn, implies, along with the definition of $D$, that the following diagram commutes:
\[\begin{tikzcd}
E \times K^{1, 0} \ar[r, "\nabla^1_0"] \ar[d, "\id_E \otimes \iota^{1, 0}" left] &
E \otimes (K^{1, 0})^{\wedge 2} \ar[d, "\id_E \otimes \iota^{2, 0}"] \\
E \times K \ar[r, "\nabla^1_{ext, 0}" below] &
E \otimes K^{\wedge 2}
\end{tikzcd}\]
By the universal property of tensor products, the following diagram then commutes:
\[\begin{tikzcd}
E \ar[r, "\nabla^0"] \ar[rd, "\nabla_{ext}^0" below left] &
E \otimes K^{1, 0} \ar[r, "\nabla^1"] \ar[d, "\id_E \otimes \iota^{1, 0}"] &
E \otimes (K^{1, 0})^{\wedge 2} \ar[d, "\id_E \otimes \iota^{2, 0}"] \\
& E \otimes K \ar[r, "\nabla_{ext}^1" below] & E \otimes K^{\wedge 2}
\end{tikzcd}\]
and the flatness of $\nabla_{ext}$ follows from that of $\nabla$ and
\cref{prop:curvature-is-linear}.
The proof for the case of $\ol{d}^{0, \bullet}$ and $\ol{D}^\bullet$ is symmetric.

(ii) $\bF$--linearity is similar to (i). Next, observe that:
\begin{align*}
\nabla^{1, 0}(rs)
=& (\id_E \otimes pr^{1, 0})(\lambda s \otimes D^{1}(r) + r\nabla(s)) \\
=& \lambda s \otimes pr^{1, 0}(D^{1}(r)) + r(\id_E \otimes pr^{1, 0})(\nabla(s)) \\
=& \lambda s \otimes d^{1, 0}(r) + r\nabla_{d^{1, 0}}(s)
\end{align*}
The proof of flatness is similar to (i) but with $\iota^{1, 0}$ replaced by $pr^{1, 0}$.
The proof of the analogous statement for $\ol{d}^{0, \bullet}$ and $\ol{D}^\bullet$ is symmetric.

(iii) First apply (i) and (ii) to see that $\nabla^{1, 0}_{ext}$ is a
$\lambda$--$D^\bullet$--connection.
Next, apply \cref{prop:conn-scale}(i) to see that $-\nabla^{1, 0}_{ext}$ is
a $-\lambda$--$D^\bullet$--connection.
Then, apply \cref{prop:conn-sum}(i) to see that
$\nabla^{0, 1}_{ext} = \nabla - \nabla^{1, 0}_{ext}$ is a $\lambda$--$0$--connection which makes it
both a $0$--$\ol{D}^\bullet$--connection and a $0$--$D^\bullet$--connection by \cref{rmk:0-d-conn}.
Finally, by (ii), we deduce that
$\nabla^{0, 1}$ is a $0$--$\ol{d}^{0, \bullet}$--connection.
Suppose, now, that $\nabla$ is flat. By (i) and (ii), $\nabla^{1, 0}_{ext}$ is flat, which, by
\cref{prop:conn-sum}(iv), implies $\nabla^{0, 1}_{ext}$ is flat as a $0$--$D^\bullet$--connection
and:
\[
(\nabla^{1, 0}_{ext})^1 \circ (\nabla^{0, 1}_{ext})^0
+ (\nabla^{0, 1}_{ext})^1 \circ (\nabla^{1, 0}_{ext})^0 = 0
\]
Note that $\nabla^{1, 0}_{ext}$ is also flat as a $0$--$\ol{D}^\bullet$--connection by
\cref{rmk:0-d-conn-curvature}.
Finally, apply (ii) to deduce that $\nabla^{1, 0}$ and $\nabla^{0, 1}$ are flat as
$\lambda$--$d^{\bullet, 0}$-- and $0$--$\ol{d}^{0, \bullet}$--connections respectively.

(iv) By (i), $(\oprt_E)_{ext}$ is a $\lambda$--$D^\bullet$--connection and
$\theta_{ext}$ is a $0$--$\ol{D}^\bullet$--connection and hence, also a $0$--$D^\bullet$--connection
by \cref{rmk:0-d-conn}. The rest follows from \cref{prop:conn-sum}.

(v) follows from (iii) and (iv).
\end{proof}

\begin{prop}\label{prop:rest-conn-ext-full}
In the context of \cref{prop:rest-diff-ext} and \cref{prop:rest-conn-ext-prt}, if $\nabla$ is
a $\lambda$--$d$--connection, then the following hold:
\begin{enumerate}
\item $\nabla^{1, 0}$ is a $1$--$d^{\bullet, 0}$--connection.
\item $\nabla^{1, 0}_{ext}$ is a $1$--$D^\bullet$--connection.
\item $\nabla^{0, 1}$ is a $1$--$\ol{d}^{0, \bullet}$--connection.
\item $\nabla^{0, 1}_{ext}$ is a $1$--$\ol{D}^\bullet$--connection.
\item $\nabla^k = (\nabla^{1, 0}_{ext})^k + (\nabla^{0, 1}_{ext})^k$
\item If $\lambda \in \ker{d}$, then $\nabla$ is flat if and only if
$\nabla^{1, 0}$ and $\nabla^{0, 1}$ are flat, and the following identity holds:
\[
(\nabla^{1, 0}_{ext})^1 \circ (\nabla^{0, 1}_{ext})^0
+ (\nabla^{0, 1}_{ext})^1 \circ (\nabla^{1, 0}_{ext})^0 = 0
\]
\end{enumerate}
\end{prop}
\begin{proof}
(i)
$\nabla^{1, 0} = (\id_E \otimes pr^{1, 0}) \circ \nabla$ is a $\bF$--linear because it is a
composite of such maps. We then observe that:
\begin{align*}
\nabla^{1, 0}(rs)
=& (\id_E \otimes pr^{1, 0})(\nabla(rs)) \\
=& (\id_E \otimes pr^{1, 0})(\lambda s \otimes d(r) + r\nabla(s)) \\
=& \lambda s \otimes pr^{1, 0}(d(r)) + r(\id_E \otimes pr^{1, 0})(\nabla(s)) \\
=& \lambda s \otimes d^{1, 0}(r) + r \nabla^{1, 0}(s)
\end{align*}

(ii) Apply \cref{prop:rest-conn-ext-prt}(i).

(iii) Similar to (i).

(iv) Similar to (ii).

(v)
We first observe that $\nabla^k = (\nabla^{1, 0}_{ext})^k + (\nabla^{0, 1}_{ext})^k$ for all
$k = 0, 1, 2, \dots$. For $k = 0$, this is immediate. For higher $k$, we observe that:
\begin{align*}
\nabla^k(s \otimes \omega)
=& s \otimes d(\omega) + \nabla(s) \wedge \omega \\
=& s \otimes (D + \ol{D})(\omega) + (\nabla^{1, 0}_{ext} + \nabla^{0, 1}_{ext})(s) \wedge \omega \\
=& (s \otimes D(\omega) + \nabla^{1, 0}(s) \wedge \omega)
   + (s \otimes \ol{D}(\omega) + \nabla^{0, 1}(s) \wedge \omega) \\
=& (\nabla^{1, 0}_{ext})^k(s \otimes \omega)
   + (\nabla^{0, 1}_{ext})^k(s \otimes \omega)
\end{align*}

(vi)
Then, we have:
\[
\nabla^1 \circ \nabla^0
= (\nabla^{1, 0}_{ext})^1 \circ (\nabla^{1, 0}_{ext})^0
  + (\nabla^{1, 0}_{ext})^1 \circ (\nabla^{0, 1}_{ext})^0
  + (\nabla^{0, 1}_{ext})^1 \circ (\nabla^{1, 0}_{ext})^0
  + (\nabla^{0, 1}_{ext})^1 \circ (\nabla^{0, 1}_{ext})^0
\]
If $\nabla$ is flat, this is zero. We then observe that the first summand, the last summand
and the remaining summands have images in $E \otimes K^{2, 0}, E \otimes K^{0, 2}$ and
$E \otimes K^{1, 1}$ respectively. Therefore, we have three separate identities:
\[
(\nabla^{1, 0}_{ext})^1 \circ (\nabla^{1, 0}_{ext})^0 = 0
\hspace{1.5em}
(\nabla^{1, 0}_{ext})^1 \circ (\nabla^{0, 1}_{ext})^0
  + (\nabla^{0, 1}_{ext})^1 \circ (\nabla^{1, 0}_{ext})^0 = 0
\hspace{1.5em}
(\nabla^{0, 1}_{ext})^1 \circ (\nabla^{0, 1}_{ext})^0 = 0
\]
This shows that $\nabla^{1, 0}_{ext}$ and $\nabla^{0, 1}_{ext}$
are flat by \cref{prop:curvature-is-linear}. We can apply \cref{prop:rest-conn-ext-prt}(ii) to
deduce that $\nabla^{1, 0}$ and $\nabla^{0, 1}$ are flat, and we are done.
\end{proof}

\begin{prop}\label{prop:tensor-conn-ext}
In the context of \cref{prop:rest-diff-ext} and \cref{prop:rest-conn-ext-prt}, consider two
$\lambda$--$d^{\bullet, 0}$--connections $\nabla_1$, $\nabla_2$ on $E_1, E_2$ respectively.
Then, the following identity of $\lambda$--$D^\bullet$--connections holds:
\[
(\iota^{1, 0} \circ \nabla_1) \otimes (\iota^{1, 0} \circ \nabla_2)
= \iota^{1, 0} \circ (\nabla_1 \otimes \nabla_2)
\]
The analogous statement holds if $d^{\bullet, 0}, D^\bullet$ are replaced with
$\ol{d}^{0, \bullet}, \ol{D}^\bullet$ respectively.
\end{prop}
\begin{proof}
On elementary tensors $e \otimes f \in (E \otimes F)(U)$ over an open subset $U \subset Y$, we have,
by \cref{prop:tensor-conn-elem}:
\begin{align*}
 & ((\iota^{1, 0} \circ \nabla_1) \otimes (\iota^{1, 0} \circ \nabla_2))(e \otimes f) \\
=& (\iota^{1, 0} \circ \nabla_1)(e) \otimes f + e \otimes (\iota^{1, 0} \circ \nabla_2)(f) \\
=& \iota^{1, 0}(\nabla_1(e)) \otimes f + e \otimes \iota^{1, 0}(\nabla_2(f)) \\
=& \iota^{1, 0}(\nabla_1(e) \otimes f) + \iota^{1, 0}(e \otimes \nabla_2(f)) \\
=& \iota^{1, 0}((\nabla_1 \otimes \nabla_2)(e \otimes f)) \\
=& (\iota^{1, 0} \circ (\nabla_1 \otimes \nabla_2))(e \otimes f)
\end{align*}
This, in particular, shows that in a local frame, the two sides have the same connection forms, and
are, hence, equal by \cref{prop:conn-form-coll}. The proof of the analogous statement for
$\ol{d}^{0, \bullet}, \ol{D}^\bullet$ is symmetric.
\end{proof}

\begin{prop}\label{prop:dual-conn-ext}
In the context of \cref{prop:rest-diff-ext} and \cref{prop:rest-conn-ext-prt},
for any $\lambda$--$d^{\bullet, 0}$--connection $\nabla$, the following identity of
$\lambda$--$D^\bullet$--connections holds:
\[
(\iota^{1, 0} \circ \nabla)^\vee = \iota^{1, 0} \circ \nabla^\vee
\]
The analogous statement holds when replace $d^{\bullet, 0}, D^\bullet$ with
$\ol{d}^{0, \bullet}, \ol{D}^\bullet$ respectively.
\end{prop}
\begin{proof}
We observe that:
\begin{align*}
 & (\iota^{1, 0} \circ \nabla)^\vee(\eta)(s) \\
=& \lambda D(\eta(s)) - (\eta \otimes \id_K)((\iota^{1, 0} \circ \nabla)(s)) \\
=& \iota^{1, 0}(\lambda d^{0, 0}(\eta(s)))
   - (\eta \otimes \id_K \circ \iota^{1, 0})(\nabla(s)) \\
=& \iota^{1, 0}(\lambda d^{0, 0}(\eta(s)))
   - \iota^{1, 0}((\eta \otimes \id_K)(\nabla(s))) \\
=& \iota^{1, 0}(\lambda d^{0, 0}(\eta(s)) + (\eta \otimes \id_K)(\nabla(s))) \\
=& \iota^{1, 0}(\nabla^\vee(\eta)(s)) \\
=& (\iota^{1, 0} \circ \nabla^\vee(\eta))(s)
\end{align*}
We then notice that $(\iota^{1, 0} \circ \nabla^\vee)(\eta)$ is equal as a map $E \to K$ to
the composite $E \to[\nabla^\vee(\eta)] K^{1, 0} \to[\iota^{1, 0}] K$, as required.
\end{proof}

\subsection{Pullbacks of Connections}
\label{subsec:pullback-conn}

\begin{defn}[Pullback Connection]\label{defn:pullback-conn}
Let $(f, f^\sharp) : (Y, R, K, d) \to (Z, S, L, d')$ be a morphism of differential spaces.
By a slight abuse of notation, let us use $f$ to denote both the morphism of differential spaces
and the continuous map of underlying spaces.
Let $E$ be a rank $n$ locally free sheaf on $(Z, S)$ with a $\lambda$--$d'$--connection
$\nabla$. Let $A^\alpha$ be the connection forms of $E$ over an open cover
$Z = \bigcup_{\alpha \in I} Z_\alpha$ such that $A^\alpha$ is given in the corresponding frame
$\set{e^\alpha_i}_{i = 1}^n$ by the formula:
\[
A^\alpha\br{\sum_{i = 1}^n s^\alpha_i e^\alpha_i} =
 \sum_{i = 1}^n e^\alpha_i \otimes \br{\sum_{j = 1}^n A^\alpha_{ij} \wedge s^\alpha_j}
\]
for sections $A^\alpha_{ij} \in L(Z_\alpha)$.
Then, if we write $Y_\alpha := f^{-1}Z_\alpha$,
$Y = \bigcup_{\alpha \in I} Y_\alpha$ is an open cover of $Y$ over which $f^*E$ trivializes
so that $\set{f^\sharp e^\alpha_i}_{i = 1}^n$ is a frame of $f^*E|_{Y_\alpha}$.
Consider the sections $f^\sharp A^\alpha_{ij} \in K(Y_\alpha)$. By the assumption that
$f$ is a morphism of differential spaces, we have that the $f^\sharp$ commute with
the differentials and preserve the wedge products. Furthermore, the matrices of the transition
functions of the pulled back sheaf -- call these $f^\sharp g^{\alpha\beta}$ -- are
$[f^\sharp g^{\alpha\beta}_{ij}]$ by \cref{cor:pullback-frame-change}.
These, together with \cref{rmk:conn-descent-data-matrix} show that the linear maps
$f^\sharp A^\alpha$ given by the matrices $[f^\sharp A^\alpha_{ij}]_{ij}$ form a descent datum for
an $f^\sharp\lambda$--$d$--connection on $f^*E$. That is,
\[
f^\sharp A^\beta
= f^\sharp \lambda (f^\sharp g^{\alpha \beta} \otimes \id) d(f^\sharp g^{\beta\alpha})
+ f^\sharp A^\alpha
\]
giving a connection $f^*\nabla$ on $f^*E$, which we will call the pullback of $\nabla$ along
$f$.
\end{defn}

\begin{prop}\label{prop:pullback-conn-unique}
In the context of \cref{defn:pullback-conn}, $f^* \nabla$ does not depend on the choice of open
cover.
\end{prop}
\begin{proof}
Choose another cover $Z = \bigcup_{b \in J} Z'_b$ with local trivializations
$\phi^b : S|_{Z'_b}^{\oplus n} \to[\cong] E|_{Z'_b}$ and frames $\set{e^b_i}_{i}$. Let $B^b$ be the
connection forms of $\nabla$ over $Z'_b$.
Let $\nabla'$ be the pullback of $\nabla$ constructed using this cover.
Then, for each $\alpha \in I$ and $b \in J$, let $g^{\alpha b}$ be the transition function
with $g^{\alpha b}(e^b_i) = e^{\alpha}_i$.
Let $[B^b_{ij}]$ be the matrix of $B^b$ in the $\set{e^b_{i}}_{i}$ frame.
Since $B^b$ and $A^\alpha$ are connection forms of the
same connection on $Z$, \cref{prop:conn-form-compat} yields:
\[
B^b = \lambda \cdot (g^{b\alpha} \otimes \id)d(g^{\alpha b}) + A^\alpha
\]
By applying $f^\sharp$ to the matrices entry-wise, this shows that:
\[
f^\sharp B^b
= f^\sharp\lambda \cdot (f^\sharp g^{b\alpha} \otimes \id)d(f^\sharp g^{\alpha b})
+ f^\sharp A^\alpha
\]
which is precisely the condition that $\nabla'$ and $f^*\nabla$ agree locally, as in the proof
of \cref{prop:conn-form-coll}.
\end{proof}

\begin{prop}\label{prop:pullback-conn-curvature}
In the contexts of \cref{defn:conn-higher-forms} and \cref{defn:pullback-conn}, when
$\lambda \in \ker{d}$, we have
$(f^*\nabla)^{k + 1} \circ (f^*\nabla)^k = f^*(\nabla^{k + 1} \circ \nabla^{k})$,
where the right side is the pullback of the $R$--linear (by \cref{prop:curvature-is-linear})
map $\nabla^{k + 1} \circ \nabla^k$. In particular, if $\nabla$ is flat, then so is $f^*\nabla$.
\end{prop}
\begin{proof}
It suffices to verify this in local trivializations and the local statement follows from the
formulas in \cref{rmk:conn-higher-forms-loc} along with the fact that $f^\sharp$ commutes with $d$
and preserves wedge products.
\end{proof}

\begin{prop}\label{prop:pullback-conn-dual}
In the context of $\cref{defn:pullback-conn}$, let $\alpha : f^*(E^\vee) \to (f^*E)^\vee$ be the
canonical natural isomorphism. Then, $\alpha$ is a morphism of connections
$f^*(\nabla^\vee) \to (f^*\nabla)^\vee$.
\end{prop}
\begin{proof}
Let $\set{e_i}_i$ be a local frame of $E$ with dual frame $\set{\epsilon_i}_i$. Let
$\set{\epsilon'_i}_{i}$ be the dual frame of the pulled back frame $\set{f^\sharp e_i}_i$
Then, $\alpha$ sends $f^\sharp \epsilon_i$ to $\epsilon_i'$. Now, let $A = [A_{ij}]$ be the
connection matrix of $\nabla$ in the $\set{e_i}_i$ frame.
Then, by \cref{rmk:dual-conn-matrix} and the definition of pullback connections,
the $(i, j)$--entries of the connection matrix of $f^*(\nabla^\vee)$ in the
$\set{f^\sharp \epsilon_i}_i$ frame are $f^\sharp(-A_{ji}) = -f^\sharp(A_{ji})$ and those of
$(f^*\nabla)^\vee$ in the $\set{\epsilon_i'}_i$ frame are the same.
This implies:
\begin{align*}
 & (f^*\nabla)^\vee\br{\alpha\br{f^\sharp \epsilon_j}} \\
=& (f^*\nabla)^\vee\br{\epsilon_j'} \\
=& \sum_{i} \epsilon'_i \otimes (-f^\sharp(A_{ji})) \\
=& (\alpha \otimes \id_K)\br{\sum_{i} f^\sharp \epsilon_i \otimes (-f^\sharp(A_{ji}))} \\
=& (\alpha \otimes \id_K)(f^*(\nabla^\vee)(f^\sharp \epsilon_j))
\end{align*}
Since all local sections are generated by the frame, $(f^*\nabla)^\vee \circ \alpha$ and
$(\alpha \otimes \id_K) \circ f^*(\nabla^\vee)$ agree in this frame. Since the frame was arbitrary,
this agreement holds globally.
\end{proof}

\begin{prop}\label{prop:pullback-conn-tensor}
In the context of $\cref{defn:pullback-conn}$, let $F$ be another finite rank locally free
$S$--module equipped with a $\lambda$--$d'$--connection $\nabla'$.
Let $\beta : f^*(E \otimes F) \to f^*E \otimes f^*F$ be the canonical natural isomorphism.
Then, $\beta$ is a morphism of connections
$f^*(\nabla \otimes \nabla') \to f^*\nabla \otimes f^*\nabla'$.
\end{prop}
\begin{proof}
Let $\set{e_i}_i, \set{f_j}_j$ be local frames of $E, F$ respectively.
Then, $\beta$ sends $f^\sharp(e_i \otimes f_j)$ to $f^\sharp(e_i) \otimes f^\sharp(f_j)$.
Now, let $A = [A_{ij}], B = [B_{ij}]$ be the connection matrices of $\nabla, \nabla'$ in the
$\set{e_i}_i, \set{f_j}_j$ frames respectively.
Then, by \cref{rmk:tensor-conn-matrix} and the definition of pullback connections,
the $(i, j)$--entries of the connection matrix of $f^*(E \otimes F)$ in the
$\set{f^\sharp(e_i \otimes f_j)}_{i, j}$ frame are
$f^\sharp(\delta_{jq}A_{ip} + \delta_{ip}B_{jq}) = \delta_{jq}f^\sharp(A_{ip})
+ \delta_{ip}f^\sharp(B_{jq})$, and those of $f^*\nabla \otimes f^*\nabla'$ in
the $\set{f^\sharp(e_i) \otimes f^\sharp(f_j)}_{i, j}$ frame are the same.
This implies:
\begin{align*}
 & (f^*\nabla \otimes f^*\nabla)\br{\beta\br{f^\sharp(e_p \otimes f_q)}} \\
=& (f^*\nabla \otimes f^*\nabla)\br{f^\sharp(e_p) \otimes f^\sharp(f_q)} \\
=& \sum_{a, b} f^\sharp(e_i) \otimes f^\sharp(f_j)
               \otimes (\delta_{jq} f^\sharp(A_{ip}) + \delta_{ip} f^\sharp(B_{jq})) \\
=& (\beta \otimes \id_K)\br{\sum_{a, b} f^\sharp(e_i \otimes f_j)
       \otimes (\delta_{jq} f^\sharp(A_{ip}) + \delta_{ip} f^\sharp(B_{jq}))} \\
=& (\beta \otimes \id_K)((f^*(\nabla \otimes \nabla')(e_p \otimes f_q))
\end{align*}
Since all local sections are generated by the frame,
$(f^*\nabla \otimes f^*\nabla') \circ \beta$ and
$(\beta \otimes \id_K) \circ f^*(\nabla \otimes \nabla')$ agree in this frame.
Since the frame was arbitrary, this agreement holds globally.
\end{proof}

\begin{rmk}\label{rmk:pullback-conn-dual-tensor-nat}
In light of \cref{prop:pullback-conn-dual}, we will not
distinguish between $f^*(\nabla^\vee)$ and $(f^*\nabla)^\vee$, and will simply write it as
$f^*\nabla^\vee$. Similarly, because of \cref{prop:pullback-conn-tensor}, we will not distinguish
between $f^*\nabla \otimes f^*\nabla'$ and $f^*(\nabla \otimes \nabla')$.
\end{rmk}

\subsection{Examples of Connections}
\label{subsec:exm-conn}

\begin{exm}[Ordinary Connections]
For any smooth manifold $X$, consider a smooth complex vector bundle $E \to X$. Then,
a $1$--$d_{\bC}$--connection on $E^{sec}$, where $d_\bC$ is the complex exterior differential of
$X$ (\cref{exm:ext-diff}), is just a complex linear connection on $E$.
\end{exm}

\begin{exm}[Holomorphic Structures on Vector Bundles]\label{exm:cmplx-struct-bun}
For any complex manifold $X$, consider a smooth complex vector bundle $E \to X$. Then, a holomorphic
structure on $E$ is a flat $1$--$\oprt_X^{0, \bullet}$--connection $\oprt_E$ on $E^{sec}$,
where $\oprt_X^{p, q}$ is as in \cref{exm:Dol-op}.
\end{exm}

\begin{exm}[Higgs Bundles]\label{exm:Higgs-bun}
Let $X$ be a complex manifold with smooth complex vector bundle $E \to X$.
Consider a flat $1$--$\oprt^\bullet$--connection $D''$ on $E^{sec}$.
By taking
$d^{\bullet, 0} = \prt^{\bullet, 0}, \ol{d}^{0, \bullet} = \oprt^{0, \bullet},
D^\bullet = \prt^\bullet, \ol{D}^\bullet = \oprt^\bullet$, the Dolbeault operators, in
\cref{prop:rest-conn-ext-prt}(iii), we get a flat $1$--$\oprt^{0, \bullet}$--connection
$\oprt_E = (D'')^{0, 1}$ and a flat $1$--$pr^{1, 0} \oprt^{\bullet}$--connection
$\theta = (D'')^{1, 0}$ satisfying:
\[
(\oprt_E)^1_{ext} \circ \theta^0_{ext} + \theta^1_{ext} \circ (\oprt_E)^0_{ext} = 0
\]
Notice that $pr^{1, 0}\oprt^{\bullet} = 0$ is the zero map $E \to E \otimes \sA_{X}^{1, 0}$,
so that $\theta$ is also a $0$--$\prt^{\bullet, 0}$--connection by \cref{rmk:0-d-conn}.
By \cref{prop:rest-conn-ext-prt}, the triple $(E, \oprt_E, \theta)$ is precisely the data of a Higgs
bundle as defined in \cite[Example 10]{GR15}. From this point onwards, by a Higgs bundle, we will
refer to a pair of the form $(E, D'')$ as above.
\end{exm}

\begin{exm}[Smooth Families of Connections]\label{exm:conn-family}
In the context of \cref{exm:prt-ext-diff},
consider a $1$--$d_{X/U}$--connection $\nabla$ on $E^{sec}$ for a smooth complex vector bundle
$p : E \to U \times X$. Then, for each $u \in U$, we have the inclusion
$\iota_u : X \cong \set{u} \times X \to U \times X$. By taking $V = \set{u}$, $Y = X$, $f = \iota_u$
and $g = \id_X$ in \cref{exm:prt-ext-diff-morphism}, we have a morphism of differential spaces:
\[
(\iota_u, \iota_u^\sharp) : (X, \sA^0_X, \sA^\bullet_X, d) \to
    (U \times X, \sA^0_{U \times X}, \sA^\bullet_{X/U}, d_{X/U})
\]
We then get, by pullback, a $1$--$d$--connection $\nabla_u := \iota_u^*\nabla$ on $E_u^{sec}$ where
$E_u := \iota_u^*E$. Thus, the pair $(E, \nabla)$ is a smooth family of vector bundles on $X$
equipped with ordinary connections parametrized by $U$. By \cref{prop:pullback-conn-curvature},
if $\nabla$ is flat, then so is $\nabla_u$ for each $u \in U$.
\end{exm}

\begin{exm}[Smooth Families of Holomorphic Vector Bundles]\label{exm:hol-bun-family}
In the context of \cref{exm:prt-Dol-op}, consider a smooth vector bundle $p : E \to U \times X$
with a $1$--$\oprt_{X/U}^{0, \bullet}$--connection $\oprt_{E/U}$ on $E^{sec}$.
Then, for each $u \in U$, we have the inclusion $\iota_u : X \cong \set{u} \times X \to U \times X$.
By taking $V = \set{u}$, $Y = X$, $f = \iota_u$ and $g = \id_X$ in \cref{exm:prt-Dol-morphism},
we have a morphism of differential spaces:
\[
(\iota_u, \iota_u^{\sharp, 0, \bullet})
    : (X, \sA^0_X, \sA^{0, \bullet}_X, \oprt_X^{0, \bullet}) \to
      (U \times X, \sA^0_{U \times X}, \sA^{0, \bullet}_{X/U}, \oprt_{X/U}^{0, \bullet})
\]
We then get, by pullback, a $1$--$\oprt_{X}^{0, \bullet}$--connection
$\oprt_{E_u} := \iota_u^*\oprt_{E/U}$ on $E_u^{sec}$ where $E_u := \iota_u^*E$.
By \cref{prop:pullback-conn-curvature}, if $\oprt_{E/U}$ is flat, then so is
$\oprt_{E_u}$ for each $u \in U$ so that, the pair $(E_u, \oprt_{E_u})$ is a holomorphic vector
bundle on $X$. We will call the pair $(E, \oprt_{E/U})$ a smooth family of holomorphic vector
bundles on $X$ parametrized by $U$ in this case.
\end{exm}

\begin{exm}[Smooth Families of Higgs Bundles]\label{exm:Higgs-bun-family}
In the context of \cref{exm:prt-Dol-op}, consider a smooth vector bundle $p : E \to U \times X$
equipped with a flat $1$--$\oprt_{X/U}^{\bullet}$--connection $D''$. For any
$u \in U$ and $\iota_u : X \cong \set{u} \times X \to U \times X$, the
inclusion, we can take $V = \set{u}$, $Y = X$, $f = \iota_u$, $g = \id_X$ in
\cref{exm:prt-Dol-morphism} to obtain a morphism of differential spaces:
\[
(\iota_u, \iota_u^{\bullet})
: (X, \sA_{X}^0, \sA_{X}^{\bullet}, \oprt_{X}^\bullet)
\to (U \times X, \sA_{U \times X}^0, \sA_{X/U}^{\bullet}, \oprt_{X/U}^{\bullet})
\]
Then, by \cref{prop:pullback-conn-curvature}, $D''_u := \iota_u^*D''$ is a flat
$1$--$\oprt_X^{\bullet}$--connection on $E_u^{sec}$ where $E_u := \iota_u^*E$, so that the pair
$(E_u, D''_u)$ is a Higgs bundle just as in \cref{exm:Higgs-bun}.
\end{exm}

\begin{prop}\label{prop:conn-slicewise}
Let $(U \times X, R, K, \delta)$ be any of the differential spaces of
\cref{exm:prt-ext-diff} or \cref{exm:prt-Dol-op}.
Then for any $\lambda$--$\delta$--connection $\nabla$ on a vector bundle $E$ and any
$s \in (E \otimes K^{\wedge k})(W)$ for any open $W \subset U \times X$, we have:
\[
\nabla^k(s)(u, x) = (\nabla_u)^k(s_u)(x)
\]
where $\nabla_u$ and $s_u$ are the pullbacks of $\nabla$ and $s$ along the inclusion
$\set{u} \times X \hto U \times X$, and the right hand side is identified with its image
under the inclusion $(E \otimes K^{\wedge k})|_{\set{u} \times X} \hto E \otimes K^{\wedge k}$.
In particular, $\nabla$ is flat if and only if $\nabla_u$ is flat for all $u \in U$.
\end{prop}
\begin{proof}
First, suppose $K = \sA_{X/U}^1, \delta = d_{X/U}$. Consider $s$ to be an elementary tensor
$a \otimes \omega$. Then,
\begin{align*}
 & \nabla^k(a \otimes \omega)(u, x) \\
=& (a \otimes d_{X/U}(\omega) + \nabla(a) \wedge \omega)(u, x) \\
=& (a \otimes d_{X/U}(\omega))(u, x) + (\nabla(a) \wedge \omega)(u, x) \\
=& a(u, x) \otimes d_{X/U}(\omega)(u, x) + \nabla(a)(u, x) \wedge \omega(u, x) \\
=& a_u( x) \otimes d_{X/U}(\omega)(u, x) + \nabla(a)(u, x) \wedge \omega_u(x)
\end{align*}
From the local formula for $d_{X/U}$, we can see that $d_{X/U}(\omega)(u, x) = d(\omega_u)(x)$.
We can also show that $\nabla(a)(u, x) = \nabla_u(a_u)(x)$. For this,
pick an open neighbourhood $(u, x) \in W' \subset W$ over which $E$ admits a local frame
$\set{e_i}_i$ and $a|_{W'} = \sum_{i} a_i e_i$. Then, taking $A$ to be the connection form of
$\nabla$ in the $\set{e_i}_i$ frame, we get:
\begin{align*}
\nabla(a)(u, x)
=& \br{\sum_i e_i \otimes d_{X/U}(a_i) + \sum_i a_iA(e_i)}(u, x) \\
=& \sum_i e_i(u, x) \otimes d_{X/U}(a_i)(u, x) + \sum_i a_i(u, x)A(e_i)(u, x) \\
=& \sum_i e_i(u, x) \otimes d_{X/U}(a_i)(u, x) + \sum_i a_i(u, x)A(e_i)(u, x) \\
=& \sum_i e_{i, u}(x) \otimes d(a_{i, u})(x) + \sum_i a_{i, u}(x)A_u(e_{i, u})(x) \\
=& \br{\sum_i e_{i, u} \otimes d(a_{i, u}) + \sum_i a_{i, u}A_u(e_{i, u})}(x) \\
=& \nabla_u(a_u)(x)
\end{align*}
Putting these back in the first computation, we get:
\begin{align*}
 & \nabla^k(a \otimes \omega)(u, x) \\
=& a_u(x) \otimes d_{X/U}(\omega)(u, x) + \nabla(a)(u, x) \wedge \omega_u(x) \\
=& a_u(x) \otimes d(\omega_u)(x) + \nabla_u(a_u)(x) \wedge \omega_u(x) \\
=& (a_u \otimes d(\omega_u) + \nabla_u(a_u) \wedge \omega_u)(x) \\
=& \nabla_u(a_u \otimes \omega_u)(x) \\
=& \nabla_u((a \otimes \omega)_u)(x)
\end{align*}
Since every tensor can locally be written as a sum of elementary tensors, we are done.
The cases of the other differential spaces are similar.
\end{proof}

\section{Moduli Stacks of Connections}
\label{sec:mod-st-conn}

Using the language of \cref{sec:diff-conn}, we will now define our moduli stacks of interest.

\begin{conv}
For any differential space $(Y, R, K, d)$, when we speak about $\lambda$--$d$--connections,
we will assume $\lambda$ is in $\ker{d}$.
Furthermore, we will refer to a pair $(E, \nabla_E)$ where $E$ is a
finite rank locally free $R$--module and $\nabla_E$ is a $\lambda$--$d$--connection on $E$,
as simply a $\lambda$--$d$--connection, as opposed to a sheaf equipped with a connection.
We will write pairs of this form simply as $\nabla_E$ for brevity, when there is no ambiguity.
\end{conv}

\begin{notn}
For a commutative ring $\bF$, we will denote the category of $\bF$--modules as $\bF\Mod$.
Given an Abelian category $\mcA$ such as $\bF\Mod$, We will denote the category of cochain complexes
in $\mcA$ as $\Ch(\mcA)$ and the subcategory thereof consisting of the non-negatively graded
cochain complexes as $\Ch^{\geq 0}(\mcA)$.
\end{notn}

The route we will take starts with defining a differential graded category of flat connections
on a differential $\bF$--space. For this, we require Hom complexes between two flat connections.
Given two flat connections $\nabla_E, \nabla_F$, we can form a Hom
flat connection $[\nabla_E, \nabla_F]$, whose complex of derived global sections will be taken to
be the Hom complex with domain $\nabla_E$ and codomain $\nabla_F$. With these Hom complexes,
flat connections form a differential graded category. The homotopy category of this differential
graded is the ordinary category of connections with morphisms as defined in \cref{defn:conn}.
However, a bit more is true: taking the
complex $\dbr{\nabla_E, \nabla_F}$ of higher forms of $[\nabla_E, \nabla_F]$ as Hom objects
gives an enrichment of flat connections in the category of complexes of sheaves of $\bF$--modules,
and the differential graded category structure is obtained by applying the derived global sections
functor Hom-wise. This picture is developed in detail in \cref{subsec:Hom-conn-cmplx}.
Given a morphism of differential spaces $f : Y \to Z$, the pullback of connections along $f$
described in \cref{subsec:pullback-conn} is a $\bF\dg$--functor between $\bF\dg$--categories of
connections on the two differential spaces.
We study this functoriality of pullback in \cref{subsec:functoriality-pullback}.
Next, we will show that the assignment:
\[
\text{differential $\bF$--space } Y \mapsto \text{$\bF\dg$--category of flat connections on $Y$}
\]
is also functorial: that is, it is a pseudofunctor from the $1$--category of differential
$\bF$--spaces to the $2$--category of $\bF\dg$--categories.
The action on morphisms of differential spaces is, as expected, by pullback of connections.
Given two composeable morphisms $f, g$ of differential spaces, the pullback  $\bF\dg$--functor for
$g \circ f$ and the composite of the pullback functors for $f$ and $g$ are naturally isomorphic.
At the same time, pullback along the identity morphism of differential spaces is naturally
isomorphic to the identity functor of the $\bF\dg$--category of connections. These natural
isomorphisms are simply the natural isomorphisms of the corresponding functors of locally free
sheaves, and hence they satisfy the coherence conditions needed to make the above mapping a
pseudofunctor. Restricting this pseudofunctor to manifolds of the form $U \times X$ with suitable
choices of partial differentials or partial Dolbeault operators, for a fixed complex manifold $X$,
gives differential graded prestacks --- pseudofunctors:
\[
\set{\text{manifolds}} \to \set{\bC\dg\text{--categories}}
\]
parametrizing Higgs bundles and flat connections. We can compose these
$\dg$--prestacks with the homtopy category functor followed by the maximal subgroupoid functor,
and then take the Grothendieck construction of the psedudofunctor valued in groupoids so obtained,
to get honest prestacks --- categories fibred in groupoids over the category of manifolds ---
parametrizing Higgs bundles and flat connections. These constructions are explained in detail
in \cref{subsec:prest-conn}. Finally, in \cref{subsec:geometricity}, we show that the prestacks
we obtained are, in fact, diffeological stacks: stacks over the site of smooth manifolds, equipped
with representable, epimorphic subductions --- that is, atlases --- from diffeological spaces,
considered via their embedding into the $2$--category of stacks.

\subsection{Hom Connections and Complexes}
\label{subsec:Hom-conn-cmplx}

\begin{defn}[Hom Connection]\label{defn:hom-conn}
Let $(Y, R, K, d)$ be a differential space over a base ring $\bF$ equipped with
$\lambda$--$d$--connections $\nabla_E, \nabla_F$. We denote by
$\CHom{E}{F}$ the $\lambda$--$d$--connection $\nabla_E^\vee \otimes \nabla_F$
(on $E^\vee \otimes F$),
and call it the Hom connection from $\nabla_E$ to $\nabla_F$.
By \cref{prop:dual-conn-flat} and \cref{cor:tensor-conn-flat}, when $\nabla_E$ and $\nabla_F$ are
flat we obtain a complex of sheaves of $\bF$--modules:
\[\begin{tikzcd}[column sep = huge]
0 \ar[r, "0"] &
E^\vee \otimes F
    \ar[r, "{[\nabla_E, \nabla_F]^0}"] &
E^\vee \otimes F \otimes K
    \ar[r, "{[\nabla_E, \nabla_F]^1}"] &
E^\vee \otimes F \otimes K^{\wedge 2}
    \ar[r, "{[\nabla_E, \nabla_F]^2}"] &
\cdots
\end{tikzcd}\]
which we will denote as $\CCHom{E}{F}$ and call the sheaf Hom complex from $\nabla_E$ to
$\nabla_F$.
\end{defn}

\begin{rmk}\label{rmk:hom-conn-desc}
In the context of \cref{defn:hom-conn}, let $U \subset Y$ be an open subset and let
$\alpha \otimes b \in (E^\vee \otimes F)(U)$ be an elementary tensor. Then, by by unwrapping
\cref{defn:dual-conn} and \cref{defn:tensor-conn}, for any open $U \subset Y$,
$\alpha \in E^\vee(U), b \in F(U), s \in E(U)$ we get:
\begin{align*}
 & \CHom{E}{F}(\alpha \otimes b)(s) \\
=& (\nabla_{E}^\vee(\alpha) \otimes b + \alpha \otimes \nabla_F(b))(s) \\
=& b \otimes \nabla_E^\vee(\alpha)(s) + \alpha(s) \cdot \nabla_F(b) \\
=& b \otimes (\lambda d(\alpha(s)) - (\alpha \otimes \id_K)(\nabla_E(s)))
   + \alpha(s) \cdot \nabla_F(b) \\
=& \lambda \cdot b \otimes d(\alpha(s)) + \alpha(s) \nabla_F(b)
   - b \otimes (\alpha \otimes \id_K)(\nabla_E(s)) \\
=& \nabla_F(\alpha(s) \cdot b)
   - (\alpha \otimes b \otimes \id_K)(\nabla_E(s)) \\
=& \nabla_F((\alpha \otimes b)(s))
   - (\alpha \otimes b \otimes \id_K)(\nabla_E(s)) \\
=& (\nabla_F \circ (\alpha \otimes b) - ((\alpha \otimes b) \otimes \id_K) \circ \nabla_E)(s)
\end{align*}
Since $\CHom{E}{F}$ is $R$--balanced in the $\alpha$ and $b$ arguments and agrees with the map
\[
E^\vee \otimes F \to E^\vee \otimes F \otimes K
: \beta \mapsto \nabla_F \circ \beta - (\beta \otimes \id_K) \circ \nabla_E
\]
on elementary tensors, the universal property of tensor products implies that the two maps
agree on all inputs.
Now, suppose $\beta \in (E^\vee \otimes F)(U), \omega \in K^{\wedge}(U)$. Then,
\begin{align*}
 & \CHom{E}{F}^k(\beta \otimes \omega) \\
=& \lambda \cdot \beta \otimes d(\omega) + \CHom{E}{F}(\beta) \wedge \omega \\
=& \lambda \cdot \beta \otimes d(\omega)
   + (\nabla_F \circ \beta - (\beta \otimes \id_K) \circ \nabla_E) \wedge \omega \\
=& \lambda \cdot \beta \otimes d(\omega)
   + (\nabla_F \circ \beta) \wedge \omega - ((\beta \otimes \id_K) \circ \nabla_E) \wedge \omega
\end{align*}
\end{rmk}

\begin{thm}\label{thm:conn-cat-enriched-in-sheaves}
With sheaf Hom complexes as in \cref{defn:hom-conn} taken as morphism objects, the collection of
flat $\lambda$--$d$--connections $\nabla_E$ form a category enriched in
$\Ch^{\geq 0}(\Sh(X, \bF\Mod))$, the category of bounded below complexes of sheaves of
$\bF$--modules.
\end{thm}
\begin{proof}
We first define composition maps of the form:
\[
c_{E, F, G} : \CCHom{F}{G} \otimes_{\ul{\bF}} \CCHom{E}{F}
    \to[] \CCHom{E}{G}
\]
where $\ul{\bF}$ is the constant sheaf (or, sheafification of the constant presheaf) valued at
$\bF$.
The tensor product forming the domain of $c$ is graded as follows
\[
\br{\CCHom{F}{G} \otimes_{\ul{\bF}} \CCHom{E}{F}}^k
    = \bigoplus_{p + q = k} (F^\vee \otimes G \otimes K^{\wedge p})
                            \otimes_{\ul{\bF}} (E^\vee \otimes F \otimes K^{\wedge q})
\]
and its differential, call it $D$, is the unique map of $\bF$--modules corresponding to the
$\bF$--bilinear map:
\[
(s, t) \mapsto \CHom{F}{G}(s) \otimes_{\ul{\bF}} t
               + (-1)^{\deg(s)} s \otimes_{\ul{\bF}} \CHom{E}{F}(t)
\]
where $s, t$ are local sections of
$F^\vee \otimes G \otimes K^{\wedge p}, E^\vee \otimes F \otimes K^{\wedge q}$ respectively.
We then define $c_{E, F, G}^k$ as the following composition:
\begin{align*}
 & \br{\CCHom{F}{G} \otimes_{\ul{\bF}} \CCHom{E}{F}}^k \\
\to[=] & \bigoplus_{p + q = k} (F^\vee \otimes G \otimes K^{\wedge p})
                            \otimes_{\ul{\bF}} (E^\vee \otimes F \otimes K^{\wedge q}) \\
\to & \bigoplus_{p + q = k} (F^\vee \otimes G \otimes K^{\wedge p})
                            \otimes (E^\vee \otimes F \otimes K^{\wedge q}) \\
\to[\cong]& F^\vee \otimes F \otimes E^\vee \otimes G \otimes
            \bigoplus_{p + q = k} K^{\wedge p} \otimes K^{\wedge q} \\
\to & R \otimes E^\vee \otimes G \otimes K^{\wedge k} \\
\to[\cong] & \CCHom{E}{G}^k
\end{align*}
where the second map is the canonical map changing the tensor product from over $\ul{\bF}$ to $R$,
and the second last map is $\tr \otimes \id_{E^\vee \otimes G} \otimes (- \wedge -)$.
On elementary tensors, this map is given by:
\[
c_{E, F, G} : (a \otimes \omega) \otimes_{\ul{\bF}} (b \otimes \tau) \mapsto
(a \circ b) \otimes (\omega \wedge \tau)
\]
We define the identity maps as follows:
\[
e_E : \ul{\bF}[0] \to \CCHom{E}{E}
    : 1_{\ul{\bF}(U)} \mapsto \id_{E|_U} \in (E^\vee \otimes E)(U) \cong \End(E)(U)
\]
for an open $U \subset Y$.
From these definitions, we can see that associativity and unitality of composition follows
from those of composition of sheaf maps and wedge products of sections of $K^\bullet$.

It remains to check that these composition and identity maps commute with the differentials.
For the identity, map we simply need to check that $\id_{E|_U} \in \ker{\CHom{E}{E}}$, which follows
from \cref{rmk:hom-conn-desc}:
\[
\CHom{E}{E}(\id_E) = \nabla_E \circ \id_E - (\id_E \otimes \id_K) \circ \nabla_E = 0
\]
For the composition maps, we have to check that the following diagram commutes:
\[\begin{tikzcd}
(\CCHom{F}{G} \otimes_{\ul{\bF}} \CCHom{E}{F})^k \ar[r, "D"] \ar[d, "c_{E, F, G}" left] &
(\CCHom{F}{G} \otimes_{\ul{\bF}} \CCHom{E}{F})^{k + 1} \ar[d, "c_{E, F, G}"] \\
\CCHom{E}{G}^k \ar[r, "\CHom{E}{G}^k" below] &
\CCHom{E}{G}^{k + 1}
\end{tikzcd}\]
It suffices to check that the two paths of maps agree in a local trivialization.
So, consider local frames $\set{e_i}, \set{f_j}, \set{g_k}$ for $E, F, G$ respectively with dual
frames $\set{\epsilon_i}, \set{\varphi_j}, \set{\gamma_k}$ respectively.
Let the connection forms and matrices of $\nabla_E, \nabla_F, \nabla_G$ be $A, B, C$ and
$[A_{aa'}], [B_{bb'}], [C_{cc'}]$ respectively with duals denoted by superscripts $(-)^\vee$. For
convenience, we will suppress ``$\otimes$''
from the notation but retain ``$\otimes_{\ul\bF}$''.
Now, it suffices to check the agreement of the maps on elementary tensors of the form
$\varphi_j g_k \omega \otimes_{\ul{\bF}} \epsilon_i f_{j'} \tau$.
Using the formulas from \cref{rmk:conn-higher-forms-loc}, \cref{rmk:dual-conn-matrix} and
\cref{rmk:tensor-conn-matrix}, we then compute:
\begin{align*}
 & \CHom{E}{G}^{k}\br{c_{E, F, G}\br{
     \varphi_m g_l \omega \otimes_{\ul{\bF}} \epsilon_j f_{n} \tau
   }} \\
=& \delta_{m n} \CHom{E}{G}^{k}\br{\epsilon_jg_l(\omega \wedge \tau)} \\
=& \delta_{m n} \br{
     \lambda \epsilon_j g_l d(\omega \wedge \tau)
     + \sum_{i, k} \epsilon_i g_k (\delta_{k l}A^\vee_{i j} + \delta_{i j} C_{kl})
                                  \wedge \omega \wedge \tau
   } \\
=& \delta_{m n} \br{
     \lambda \epsilon_j g_l \br{
       d(\omega)  \wedge \tau
       + (-1)^{\deg(\omega)} \wedge \omega \wedge d(\tau)
     }
     + \sum_{i, k} \epsilon_i g_k (-\delta_{k l}A_{j i} + \delta_{i j} C_{kl})
                                  \wedge \omega \wedge \tau
   }
\end{align*}
On the other hand, we have:
\begin{align*}
 & c_{E, F, G}\br{D\br{\varphi_m g_l \omega \otimes_{\ul{\bF}} \epsilon_j f_{n} \tau}} \\
=& c_{E, F, G}\br{
     \CHom{F}{G}\br{\varphi_m g_l \omega} \otimes_{\ul\bF} \epsilon_j f_{n} \tau
     + (-1)^{\deg(\omega)}
     \varphi_m g_l \omega \otimes_{\ul\bF} \CHom{E}{F}\br{\epsilon_j f_{n} \tau}
   } \\
=& c_{E, F, G}\br{\CHom{F}{G}\br{\varphi_m g_l \omega} \otimes_{\ul\bF} \epsilon_j f_{n} \tau}
     + (-1)^{\deg(\omega)} c_{E, F, G} \br{
       \varphi_m g_l \omega \otimes_{\ul\bF} \CHom{E}{F}\br{\epsilon_j f_{n} \tau}
     }
\end{align*}
We compute the summands separately:
\begin{align*}
 & c_{E, F, G}\br{\CHom{F}{G}\br{\varphi_m g_l \omega} \otimes_{\ul\bF} \epsilon_j f_{n} \tau} \\
=& c_{E, F, G}\br{
     \br{
       \lambda \varphi_m g_l d(\omega)
       + \sum_{i, k} \phi_ig_k (\delta_{k l}B^\vee_{i m} + \delta_{i m}C_{k l}) \wedge \omega
     }
     \otimes_{\ul\bF} \epsilon_j f_{n} \tau
   } \\
=&   \delta_{mn} \lambda \epsilon_j g_l d(\omega) \wedge \tau
     + \sum_{i, k} \delta_{i n} \epsilon_j g_k (- \delta_{k l}B_{m i} + \delta_{i m}C_{k l})
                                               \wedge \omega \wedge \tau \\
=&   \delta_{mn} \lambda \epsilon_j g_l d(\omega) \wedge \tau
     + \sum_{k} \epsilon_j g_k (- \delta_{k l}B_{m n} + \delta_{n m}C_{k l})
                               \wedge \omega \wedge \tau \\
=&   \delta_{mn} \lambda \epsilon_j g_l d(\omega) \wedge \tau
     + \br{-\epsilon_j g_l B_{m n}
           + \delta_{n m}\sum_{k} \epsilon_j g_k C_{k l}} \wedge \omega \wedge \tau
\end{align*}
Denote this as section as $\rho$.
We also have:
\begin{align*}
 & c_{E, F, G}\br{\varphi_m g_l \omega \otimes_{\ul\bF} \CHom{E}{F}\br{\epsilon_j f_{n} \tau}} \\
=& c_{E, F, G}\br{
     \phi_m g_l \omega \otimes_{\ul\bF}
     \br{
       \lambda \epsilon_j f_n d(\tau)
       + \sum_{i, k} \epsilon_i f_k \br{\delta_{k n} A^\vee_{i j} + \delta_{i j} B_{k n}}
                                    \wedge \tau
     }
   } \\
=& \delta_{mn} \lambda \epsilon_j g_l \omega \wedge d(\tau)
   + \sum_{i, k} \delta_{mk} \epsilon_i g_l
                 \omega \wedge \br{\delta_{k n} A^\vee_{i j} + \delta_{i j} B_{k n}} \wedge \tau \\
=& \delta_{mn} \lambda \epsilon_j g_l \omega \wedge d(\tau)
   + (-1)^{\deg(\omega)} \sum_{i, k} \delta_{m k} \epsilon_i g_l
                 \br{\delta_{k n} A^\vee_{i j} + \delta_{i j} B_{k n}} \wedge \omega \wedge \tau \\
=& \delta_{mn} \lambda \epsilon_j g_l \omega \wedge d(\tau)
   + (-1)^{\deg(\omega)} \sum_{i} \epsilon_i g_l
                 \br{\delta_{m n} A^\vee_{i j} + \delta_{i j} B_{m n}} \wedge \omega \wedge \tau \\
=& \delta_{mn} \lambda \epsilon_j g_l \omega \wedge d(\tau)
   + (-1)^{\deg(\omega)} \br{\delta_{m n} \sum_{i} \epsilon_i g_l A^\vee_{i j}
                             + \epsilon_j g_l B_{m n}} \wedge \omega \wedge \tau
\end{align*}
Denote this section as $\xi$. These computations show that:
\[
c_{E, F, G}\br{D\br{\varphi_m g_l \omega \otimes_{\ul{\bF}} \epsilon_j f_{n} \tau}}
= \rho + (-1)^{\deg(\omega)} \xi
= \CHom{E}{G}^{k}\br{c_{E, F, G}\br{\varphi_m g_l \omega \otimes_{\ul{\bF}} \epsilon_j f_{n} \tau}}
\]
as required.
\end{proof}

\begin{thm}\label{thm:dg-cat-conn}
In the context of \cref{thm:conn-cat-enriched-in-sheaves}, the collection of flat
$\lambda$--$d$--connections admits differential graded category structures with the following
$\Hom$ complexes:
\begin{enumerate}
\item the complex of global sections of the sheaf Hom complexes, in which case the homotopy category
is the ordinary category of $\lambda$--$d$--connections, or
\item the complex of derived global sections of the sheaf Hom complexes
\end{enumerate}
\end{thm}
\begin{proof}
(i) The global sections functor
$\Gamma : \Ch^{\leq 0}(\Sh(X, \bF\Mod)) \to \Ch^{\leq 0}(\bF\Mod)$ is well-known to be lax monoidal,
providing a change of enrichment functor by \cite[Lemma 3.4.3]{Rie14}:
\[
\Gamma_* : \Ch^{\geq 0}(\Sh(X, \bF\Mod))-\Cat \to \Ch^{\geq 0}(\bF\Mod)-\Cat
\]
The first differential graded category structure then follows from
\cref{thm:conn-cat-enriched-in-sheaves}. For any two objects $\nabla_E, \nabla_F$,
$H^0(\Gamma(X, \CCHom{E}{F}])^\bullet) = \ker{\CHom{E}{F}^0}$. By \cref{rmk:hom-conn-desc},
we see that the homotopy category is precisely the ordinary category of $\lambda$--$d$--connections.

(ii) We can use the following construction for the derived global sections. Given a complex
$(A^\bullet, D)$ of $\bF$--module sheaves, we take its Godement resolution
$(A^\bullet, D) \to[\simeq] (\tilde{A}^\bullet, \tilde{D})$, where each term $\tilde{A}^n$ is a
flasque sheaf, making it a $\Gamma$--acyclic complex quasi-isomorphic to $(A^\bullet, D)$, so that
the complex of derived global sections of $(A^\bullet, D)$ is:
\[
\RGl(X, A^\bullet, D) = \Gamma(X, \tilde{A}^\bullet, \tilde{D}^\bullet)
\]
We recall how $(\tilde{A}^\bullet, \tilde{D})$ is obtained \cite[\S 3.3, 5.1]{RR15}:
we consider the ringed space $(X_{disc}, R_{disc})$ where $X_{disc}$ is the underlying set of $X$
equipped with the discrete topology and $R_{disc}$ is the sheaf of rings defined by
$R_{disc}(U) = \prod_{x \in U} R_x$ with restrictions induced by those of $R$. Note that
$R_{disc} = f^{-1}R$.
There is a morphism of ringed spaces $f : (X_{disc}, R_{disc}) \to (X, R)$ yielding a cosimplicial
cochain complex:
\[
G^n(A^\bullet, D) = (f_*f^{-1})^n(A^\bullet, D)
\]
Note that we use $f^{-1}$ because the differentials of complexes are not $R$--linear.
Then, we take $(\tilde{A}^\bullet, \tilde{D})$ to be the simple complex
$s(\set{G^n(A^\bullet, D)}_{n \geq 0})$ associated with this
cosimplicial cochain complex --- this is the totalization of the Moore double complex of the
cosimplicial cochain complex.
Now, the mapping
\[
G^\bullet : \Ch^{\geq 0}(\Sh(X, \bF\Mod)) \to cs\Ch^{\geq 0}(\Sh(X, \bF\Mod))
\]
is monoidal --- this is because $f^{-1}$ preserves tensor products in general, and we can verify
that so does $f_*$ in this specific case \cite[Lemma 3.11, Proposition 3.12]{RR17}. Then,
the mapping
\[
s : sc\Ch^{\leq 0}(\Sh(X, \bF\Mod)) \to \Ch^{\leq 0}(\Sh(X, \bF\Mod))
\]
has a lax monoidal structure given by the Alexander-Whitney map --- the formal dual
of the argument in \cite[\href{https://kerodon.net/tag/00S4}{Tag 00S4}]{kerodon} works for any
Abelian category in place of Abelian groups.
Finally, the ordinary global sections functor
$\Gamma : \Ch^{\leq 0}(\Sh(X, \bF\Mod)) \to \Ch^{\leq 0}(\bF\Mod)$ is well-known to be lax monoidal.
The derived global sections functor can then be written as a composition:
\[
\RGl = \Gamma \circ s \circ G^\bullet : \Ch^{\geq 0}(\Sh(X, \bF\Mod)) \to \Ch^{\geq 0}(\bF\Mod)
\]
and is hence lax monoidal --- we should note that we are not considering the functor on the derived
categories and the lax monoidality is with respect to the ordinary tensor product
$\otimes_{\ul\bF}$.
Thus, we again get a change of enrichment functor \cite[Lemma 3.4.3]{Rie14}:
\[
\RGl_* : \Ch^{\geq 0}(\Sh(X, \bF\Mod))-\Cat \to \Ch^{\geq 0}(\bF\Mod)
\]
The $\Ch^{\geq 0}(\bF\Mod)$--category structure now follows from
\cref{thm:conn-cat-enriched-in-sheaves}.
\end{proof}

\begin{notn}
We will denote the differential graded category structure of \cref{thm:dg-cat-conn}(i) by
$\Conn_{\lambda}(Y, R, K, d)$ and the one of \cref{thm:dg-cat-conn}(ii) by
$\RConn_{\lambda}(Y, R, K, d)$ or simply, $\Conn_{\lambda, d}(Y)$ and $\RConn_{\lambda, d}(Y)$
respectively, when $R, K, d$ are clear from context.
\end{notn}

\begin{warn}\label{warn:dg-cat-conn}
In the context of \cref{thm:dg-cat-conn}, the homotopy category of $\RConn_{\lambda}(Y)$ may not be
the ordinary category of $\lambda$--$d$--connections as the complex of derived global sectoins is
not, in general, quasi-isomorphic to the complex of ordinary
global sections. However, in the case that $(Y, R) = (X, \Cinf_X)$ is a smooth manifold, all sheaves
of $\Cinf_X$--modules are flasque by the existence of bump functions, making $\RGl = \Gamma$ in this
case, and the two differential graded category structures coincide with the homotopy category being
the ordinary category of $\lambda$--$d$--connections, regardless of what $d$ is.
\end{warn}

\begin{rmk}
For a differential graded category where the $\Hom$ complexes are concentrated in non-negative
degrees, the homotopy category is isomorphic to the underlying category (see
\cite[Remarks 1.3.1.4 and 1.3.1.5]{HA}).
\end{rmk}

\subsection{Functoriality of Pullback}
\label{subsec:functoriality-pullback}

\begin{prop}\label{prop:pullback-conn-higher-forms}
In the contexts of \cref{defn:pullback-conn} and \cref{defn:conn-higher-forms},
consider the map
\[
\eta^k_0
: E \times L^{\wedge k} \to f_*(f^*E \otimes K)
: (s, \omega) \mapsto \eta(s) \otimes f^\sharp(\omega)
\]
where $\eta : E \to f_*f^*E$ is the adjunction unit.
We can see that $\eta^k_0$ is $R$--bilinear and thus, by the universal property of tensor products,
gives an $R$--linear map $\eta^k$ making the following diagram commute:
\[\begin{tikzcd}
E \times L^{\wedge k} \ar[r, "\eta^k_0"] \ar[d, "\otimes" left] & f_*(f^*E \otimes K) \\
E \otimes L^{\wedge k} \ar[ru, "\eta^k" below right]
\end{tikzcd}\]
We then have a commutative diagram of sheaves of $\bF$--modules:
\[\begin{tikzcd}[column sep = large]
0 \ar[r] \ar[d, equal] &
E \ar[r, "\nabla^0"] \ar[d, "\eta^0"] &
E \otimes K \ar[r, "\nabla^1"] \ar[d, "\eta^1"] &
E \otimes K^{\wedge 2} \ar[r, "\nabla^2"] \ar[d, "\eta^2"] & \cdots \\
0 \ar[r] & f_*f^*E \ar[r, "f_*f^*\nabla^0" below] &
f_*(f^*E \otimes K) \ar[r, "f_*f^*\nabla^1" below] &
f_*(f^*E \otimes K^{\wedge 2}) \ar[r, "f_*f^*\nabla^2" below] & \cdots
\end{tikzcd}\]
In particular, when $\nabla$ is flat, we get morphisms of cochain complexes of $\bF$--modules:
\begin{enumerate}
\item $\Gamma(Z, E \otimes L^{\wedge \bullet}, \nabla^\bullet)
    \to \Gamma(Y, f^*E \otimes K^{\wedge \bullet}, f^*\nabla^\bullet)$
\item $\RGl(Z, E \otimes L^{\wedge \bullet}, \nabla^\bullet)
    \to \RGl(Z, f_*(f^*E \otimes K^{\wedge \bullet}), f_*f^*\nabla^\bullet)$
\item $\RGl(Z, E \otimes L^{\wedge \bullet}, \nabla^\bullet)
    \to \RGl(Y, f^*E \otimes K^{\wedge \bullet}, f^*\nabla^\bullet)$
\end{enumerate}
\end{prop}
\begin{proof}
Consider an open $Z' \subset Z$ over which both $E$ and $L$ trivialize with frames
$\set{e_i}_i, \set{\omega_j}_j$ respectively. Let $A$ be the connection form in this local
trivialization and let $f^\sharp A$ be the connection form of the pulled back connection.
Observe that $\eta(e_i) = f^\sharp(e_i)$, and then apply the formula for $\nabla^{k + 1}$ from
\cref{rmk:conn-higher-forms-loc}, and that for $\eta^k$ from the statement,
to obtain, for any section
$s = \sum_{i} e_i \otimes \tau_j \in (E \otimes L)^{\wedge k}(Z')$, the following
equality:
\begin{align*}
 & \eta^{k + 1} \br{\nabla^k\br{\sum_{i} e_i \otimes \tau_j}} \\
=& \eta^{k + 1}\br{\sum_{i} e_i \otimes \br{\lambda \cdot d(\tau_i)
                   + \sum_{j} A_{ij} \wedge \tau_j}} \\
=& \sum_{i} \eta(e_i) \otimes f^\sharp\br{\lambda \cdot d(\tau_i)
                                  + \sum_{j} A_{ij} \wedge \tau_j} \\
=& \sum_{i} f^\sharp(e_i) \otimes \br{\lambda \cdot d(f^\sharp \tau_i)
                                  + \sum_{j} f^\sharp A_{ij} \wedge f^\sharp \tau_j}
\end{align*}
Similarly, we have:
\begin{align*}
 & f_*f^*\nabla^k \br{\eta^k\br{\sum_{i} e_i \otimes \tau_j}} \\
=& f^*\nabla^k \br{\sum_{i} \eta(e_i) \otimes f^\sharp \tau_i} \\
=& \sum_{i} f^\sharp(e_i) \otimes \br{\lambda \cdot d(f^\sharp \tau_i)
                                  + \sum_{j} f^\sharp A_{ij} \wedge f^\sharp \tau_j}
\end{align*}
This shows that $\eta^{k + 1} \circ \nabla^k$ and $f_*f^*\nabla^k \circ \eta^k$ agree in every local
trivialization, and hence on every open.

The morphism of cochain complexes in (i) is obtained by applying the global sections functor to the
given commutative diagram:
\[
\Gamma(\eta^\bullet) : \Gamma(Z, E \otimes L^{\wedge \bullet}, \nabla^\bullet)
    \to \Gamma(Y, f^*E \otimes K^{\wedge \bullet}, f^*\nabla^\bullet)
\]
For (ii), we may apply the derived global sections functor instead:
\[
\RGl(\eta^\bullet) : \RGl(Z, E \otimes L^{\wedge \bullet}, \nabla^\bullet)
    \to \RGl(Y, f^*E \otimes K^{\wedge \bullet}, f^*\nabla^\bullet)
\]
For the morphism of cochain complexes in (ii), we first recall that we have a natural transformation
of functors:
\[
\rho : \id_{\Ch^{\geq 0}(\Sh(Y, \bF\Mod))} \To s \circ G^\bullet
\]
where $s$ is the simple complex functor (totalization of the Moore complex of a cosimplicial cochain
complex) and $G^\bullet$ is the Godement resolution functor \cite[434]{RR15}. Since,
$s \circ G^\bullet$ provides resolutions where each term is a flasque sheaf,
$f_* \circ s \circ G^\bullet$ computes $\rmR f_*$. Thus, by horizontal composition, we get a natural
transformation: $f_* \star \rho : f_* \to f_* \circ s \circ G^\bullet = \rmR f_*$.
Let $H(f^*E \otimes K^{\wedge \bullet}, f^*\nabla^\bullet)$ denote
$s(G^\bullet(f^*E \otimes K^{\wedge \bullet}, f^*\nabla^\bullet))$.
Taking the component of $\rho$ at the complex
$(f^*E \otimes K^{\wedge \bullet}, f^*\nabla^\bullet)$ gives a morphism of complexes of sheaves of
$\bF$--modules:
\begin{equation}\label{eqn:pushforward-to-der-pushforward}
f_*(f^*E \otimes K^{\wedge \bullet}, f^*\nabla^\bullet)
\to f_*(H((f_*(f^*E \otimes K^{\wedge \bullet}), f_*f^*\nabla^\bullet)))
\end{equation}
Since each term of $H(f^*E \otimes K^{\wedge \bullet}, f^*\nabla_E^\bullet)$ is flasque, so are the
terms of its pushforward, and thus:
\begin{align*}
\RGl(Z, f_*(H(f^*E \otimes K^{\wedge \bullet}, f^*\nabla^\bullet)))
=& \Gamma(Z, f_*(H(f^*E \otimes K^{\wedge \bullet}, f^*\nabla^\bullet))) \\
=& \Gamma(Y, H(f^*E \otimes K^{\wedge \bullet}, f^*\nabla^\bullet)) \\
=& \RGl(Y, f^*E \otimes K^{\wedge \bullet}, f^*\nabla^\bullet)
\end{align*}
Applying the derived global sections functor to the morphism of complexes of sheaves of
$\bF$--modules of \ref{eqn:pushforward-to-der-pushforward} provides a morphism of
complexes of $\bF$--modules:
\[
\RGl(f_*(\rho^\bullet)) : \RGl(Z, f_*(f^*E \otimes K^{\wedge \bullet}), f_*f^*\nabla^\bullet)
\to \RGl(Y, f^*E \otimes K^{\wedge \bullet}, f^*\nabla^\bullet)
\]
We compose this with the morphism of (ii) to get the morphism of (iii):
\[
\RGl(f_*(\rho^\bullet) \circ \eta^\bullet) : \RGl(Z, E \otimes L^{\wedge \bullet}, \nabla^\bullet)
    \to \RGl(Y, f^*E \otimes K^{\wedge \bullet}, f^*\nabla^\bullet)
\]
\end{proof}

\begin{rmk}\label{rmk:hom-conn-pb-nat}
For any morphism of differential spaces $(f, f^\sharp)$ and any connections $\nabla_E, \nabla_F$
over its codomain, we have canonical natural isomorphisms:
\[
\xi_{E, F} : f^*[\nabla_E, \nabla_F] = f^*(\nabla_E^\vee \otimes \nabla_F)
\to[\cong] (f^*\nabla_E)^\vee \otimes f^*\nabla_F = [f^*\nabla_E, f^*\nabla_F]
\]
by \cref{prop:pullback-conn-dual} and \cref{prop:pullback-conn-tensor}.
By \cref{prop:conn-morphism-to-diagram-morphism}, there is an isomorphism of complexes of sheaves of
$\bF$--modules:
\[
f^*\dbr{\nabla_E, \nabla_F} \to \dbr{f^*\nabla_E, f^*\nabla_F}
\]
By the naturality $\xi$, this isomorphism of complexes is also natural in $E$ and $F$.
Hence, we will make no distinction between the two sides of these isomorphisms. In particular, we
will write morphisms into and out of these interchangeably. This causes no problems for the
arguments by the naturality of the isomorphism.
\end{rmk}

\begin{prop}\label{prop:pullback-func}
The mapping $\nabla_E \mapsto f^*\nabla_E$ along with morphisms in (i) and (iii) of
\cref{prop:pullback-conn-higher-forms} assemble into differential graded functors respectively:
\begin{enumerate}
\item $f^* : \Conn_{\lambda, d'}(Z) \to \Conn_{f^\sharp \lambda, d}(Y)$
\item $f^* : \RConn_{\lambda, d'}(Z) \to \RConn_{f^\sharp \lambda, d}(Y)$
\end{enumerate}
\end{prop}
\begin{proof}
(i)
By the naturality of the lax monoidal structure maps of $\Gamma$, the following diagram commutes:
\begin{equation}\label{eqn:gl-sec-lax-mon}
\begin{tikzcd}
\Gamma\dbr{\nabla_F, \nabla_G} \otimes_{\ul\bF} \Gamma\dbr{\nabla_E, \nabla_F}
    \ar[r] \ar[d, "\Gamma(\eta) \otimes_{\ul\bF} \Gamma(\eta)" left] &
\Gamma(\dbr{\nabla_F, \nabla_G} \otimes_{\ul\bF} \dbr{\nabla_E, \nabla_F})
    \ar[d, "\Gamma(\eta \otimes_{\ul\bF} \eta)"] \\
\Gamma f_* \dbr{f^*\nabla_F, f^*\nabla_G} \otimes_{\ul\bF} \Gamma f_*\dbr{f^*\nabla_E, f^*\nabla_F}
    \ar[r] &
\Gamma f_*(\dbr{f^*\nabla_F, f^*\nabla_G} \otimes_{\ul\bF} \dbr{f^*\nabla_E, f^*\nabla_F})
\end{tikzcd}
\end{equation}
where the horizontal maps are the lax monoidal structure maps of $\Gamma$.
Next, we can show that the following diagram commutes:
\begin{equation}\label{eqn:pullback-func-comp}
\begin{tikzcd}[column sep = huge]
\CCHom{F}{G} \otimes_{\ul\bF} \CCHom{E}{F}
    \ar[r, "c"] \ar[d, "\eta \otimes_{\ul\bF} \eta" left] &
\CCHom{E}{G}
    \ar[dd, "\eta"] \\
f_*\dbr{f^*\nabla_F, f^*\nabla_G} \otimes_{\ul\bF} f_*\dbr{f^*\nabla_E, f^*\nabla_F}
    \ar[d] & \\
f_*(\dbr{f^*\nabla_F, f^*\nabla_G} \otimes_{\ul\bF} \dbr{f^*\nabla_E, f^*\nabla_F})
    \ar[r, "f_*(c)" below] &
f_*\dbr{f^*\nabla_E, f^*\nabla_G}
\end{tikzcd}
\end{equation}
This can be shown by working in local frames,
using \cref{rmk:hom-conn-desc}, the description of
connection matrices of pulled back connections as in \cref{defn:pullback-conn}, the properties of
morphisms of differentials as in \cref{defn:diff-on-sh} along
with formulas similar to the proof of \cref{thm:conn-cat-enriched-in-sheaves}.
Applying $\Gamma$ to
diagram \ref{eqn:pullback-func-comp} and pasting with diagram \ref{eqn:gl-sec-lax-mon} along the
$\Gamma(\eta \otimes_{\ul\bF} \eta)$ edge shows that the following diagram commutes:
\[\begin{tikzcd}
\Gamma\dbr{\nabla_F, \nabla_G} \otimes_{\ul\bF} \Gamma\dbr{\nabla_E, \nabla_F}
    \ar[r, "c"] \ar[d, "\eta \otimes_{\ul\bF} \eta" left] &
\Gamma(\dbr{\nabla_E, \nabla_G})
    \ar[d, "\eta"] \\
\Gamma \dbr{f^*\nabla_F, f^*\nabla_G} \otimes_{\ul\bF} \Gamma \dbr{f^*\nabla_E, f^*\nabla_F}
    \ar[r, "c" below] &
\Gamma \dbr{f^*\nabla_E, f^*\nabla_G}
\end{tikzcd}\]

Now, consider the following diagram:
\begin{equation}\label{eqn:pullback-func-unit}\begin{tikzcd}
\ul\bF[0] \ar[r, "e"] \ar[d] & \dbr{\nabla_E, \nabla_E} \ar[d, "\eta"] \\
f_*\ul\bF[0] \ar[r, "f_*e" below]& f_*\dbr{f^*\nabla_E, f^*\nabla_E}
\end{tikzcd}\end{equation}
where horizontal maps sends $1$ to the sections of $E^\vee \otimes E$ and $f^*E^\vee \otimes f^*E$
corresponding to the identity maps, and the left vertical map is the canonical one.
It is straightforward to see that this diagram commutes by passing to a local frame. Passing to
global sections, we have the following commutative diagram of complexes:
\[\begin{tikzcd} &
\Gamma(Z, \ul\bF[0]) \ar[r, "e"] \ar[dd] & \Gamma(Z, \dbr{\nabla_E, \nabla_E}) \ar[dd, "\eta"] \\
\bF[0] \ar[ru] \ar[rd] & & \\ &
\Gamma(Z, f_*\ul\bF[0]) \ar[r, "f_*e" below] & \Gamma(Z, f_*\dbr{f^*\nabla_E, f^*\nabla_E})
\end{tikzcd}\]

(ii)
Let the morphism of \cref{prop:pullback-conn-higher-forms}(iii) be denoted
$\sigma := \RGl(f_*(\rho^\bullet) \circ \eta^\bullet)$.
Then, we wish to show the commutativity of the following diagram:
\begin{equation}\label{eqn:pullback-func-der}\begin{small}\begin{tikzcd}
\RGl(Z, \CCHom{F}{G}) \otimes_{\ul\bF} \RGl(Z, \CCHom{E}{F})
    \ar[d] \ar[r, "\sigma \otimes_{\ul\bF} \sigma"] &
\RGl(Y, \dbr{f^*\nabla_F, f^*\nabla_G}) \otimes_{\ul\bF}
    \RGl(Y, \dbr{f^*\nabla_E, f^*\nabla_F})
    \ar[d, "" right] \\
\RGl(Z, \CCHom{F}{G} \otimes_{\ul\bF} \CCHom{E}{F})
    \ar[d, "\RGl(c_{E, F, G})" left] &
\RGl(Y, \dbr{f^*\nabla_F, f^*\nabla_G} \otimes_{\ul\bF}
        \dbr{f^*\nabla_E, f^*\nabla_F})
    \ar[d, "\RGl(c_{f^*E, f^*F, f^*G})" right] \\
\RGl(Z, \CCHom{E}{G}) \ar[r, "\sigma" below] &
\RGl(Y, \dbr{f^*\nabla_E, f^*\nabla_G})
\end{tikzcd}\end{small}\end{equation}
where the vertical maps on the top row are the lax monoidal structure maps on $\RGl$, as in the
proof of \cref{thm:dg-cat-conn}.
By the functoriality of $\otimes_{\ul\bF}$, we have:
\[
\sigma \otimes_{\ul\bF} \sigma = (\RGl(f_*(\rho)) \otimes_{\ul\bF} \RGl(f_*(\rho))) \circ
                                 (\RGl(\eta) \otimes_{\ul\bF} \RGl(\eta))
\]
We also observe that there is a map:
\begin{align*}
\CCHom{F}{G} \otimes_{\ul\bF} \CCHom{E}{F}
\to[\eta \otimes_{\ul\bF} \eta] & f_*\dbr{f^*\nabla_F, f^*\nabla_G} \otimes_{\ul\bF}
        f_*\dbr{f^*\nabla_E, f^*\nabla_F} \\
\to & f_*(\dbr{f^*\nabla_F, f^*\nabla_G} \otimes_{\ul\bF}
        \dbr{f^*\nabla_E, f^*\nabla_F})
\end{align*}
where the $\eta$ are as in \cref{prop:pullback-conn-higher-forms}, the second map is the lax
monoidal structure on $f_*$, and the third map is the natural map $\rho$ to the resolution, again
as in the proof of \cref{prop:pullback-conn-higher-forms}. Let us call this map $\mu$.
Let us also denote $H = s \circ G^\bullet$ for the resolution functor.

The naturality of $\rho$ implies that the following square commutes:
\begin{equation}\label{eqn:resolution-nat}
\begin{tikzcd}[column sep = huge]
f_*(\dbr{f^*\nabla_F, f^*\nabla_G} \otimes_{\ul\bF} \dbr{f^*\nabla_E, f^*\nabla_F})
    \ar[r, "f_*(c)" below] \ar[d, "f_*(\rho)" left] &
f_*\dbr{f^*\nabla_E, f^*\nabla_G} \ar[d, "f_*(\rho)"] \\
f_*(H\br{\dbr{f^*\nabla_F, f^*\nabla_G} \otimes_{\ul\bF} \dbr{f^*\nabla_E, f^*\nabla_F}})
    \ar[r, "f_*(H\br{c})" below] &
f_*(H\br{\dbr{f^*\nabla_E, f^*\nabla_G}})
\end{tikzcd}
\end{equation}
Taking derived global sections of diagrams \ref{eqn:pullback-func-comp} and
\ref{eqn:resolution-nat} and pasting them together shows the commutativity of the
following diagram:
\begin{equation}\label{eqn:pullback-func-der-1}\begin{tikzcd}[column sep = huge]
\RGl(Z, \CCHom{F}{G} \otimes_{\ul\bF} \CCHom{E}{F})
    \ar[r, "\RGl(f_*(\rho) \circ \mu)"]
    \ar[d, "\RGl(c_{E, F, G})" left] &
\RGl(Y, \dbr{f^*\nabla_F, f^*\nabla_G} \otimes_{\ul\bF}
        \dbr{f^*\nabla_E, f^*\nabla_F})
    \ar[d, "\RGl(c_{f^*E, f^*F, f^*G})" right] \\
\RGl(Z, \CCHom{E}{G}) \ar[r, "\sigma = \RGl(f_*(\rho) \circ \eta)" below] &
\RGl(Y, \dbr{f^*\nabla_E, f^*\nabla_G})
\end{tikzcd}\end{equation}

By the naturality of lax monoidal structure maps, the following diagram commutes:
\[\begin{small}\begin{tikzcd}[column sep = small]
\RGl\dbr{\nabla_F, \nabla_G} \otimes_{\ul\bF} \RGl\dbr{\nabla_E, \nabla_F}
    \ar[r] \ar[d, "\RGl\eta \otimes_{\ul\bF} \RGl\eta" left] &
\RGl(\dbr{\nabla_F, \nabla_G} \otimes_{\ul\bF} \dbr{\nabla_E, \nabla_F})
    \ar[d, "\RGl(\eta \otimes_{\ul\bF} \eta)"] \\
\RGl f_*(\dbr{f^*\nabla_F, f^*\nabla_G}) \otimes_{\ul\bF} \RGl f_*(\dbr{f^*\nabla_E, f^*\nabla_F})
    \ar[r]
    \ar[d, "\RGl f_*(\rho) \otimes_{\ul\bF} \RGl f_*(\rho)" left] &
\RGl(f_*\dbr{f^*\nabla_F, f^*\nabla_G} \otimes_{\ul\bF} f_*\dbr{f^*\nabla_E, f^*\nabla_F})
    \ar[d, "\RGl(f_*(\rho) \otimes_{\ul\bF} f_*(\rho))"] \\
\RGl f_* H\dbr{f^*\nabla_F, f^*\nabla_G} \otimes_{\ul\bF} \RGl f_* H\dbr{f^*\nabla_E, f^*\nabla_F}
    \ar[r] &
\RGl(f_* H\dbr{f^*\nabla_F, f^*\nabla_G} \otimes_{\ul\bF} f_* H\dbr{f^*\nabla_E, f^*\nabla_F})
\end{tikzcd}\end{small}\]
where the horizontal maps are the lax monoidal structure maps of $\RGl$.
Similarly, we have a further commutative diagram:
\[\begin{small}\begin{tikzcd}[column sep = small]
\RGl(\dbr{\nabla_F, \nabla_G} \otimes_{\ul\bF} \dbr{\nabla_E, \nabla_F})
    \ar[rd, shift left, "\RGl(\mu)" above right]
    \ar[d, "\RGl(\eta \otimes_{\ul\bF} \eta)" left] & \\
\RGl(f_*\dbr{f^*\nabla_F, f^*\nabla_G} \otimes_{\ul\bF} f_*\dbr{f^*\nabla_E, f^*\nabla_F})
    \ar[r]
    \ar[d, "\RGl(f_*(\rho) \otimes_{\ul\bF} f_*(\rho))" left] &
\RGl(f_*(\dbr{f^*\nabla_F, f^*\nabla_G} \otimes_{\ul\bF} \dbr{f^*\nabla_E, f^*\nabla_F}))
    \ar[d, "\RGl(f_*(\rho \otimes \rho))"] \\
\RGl(f_* H\dbr{f^*\nabla_F, f^*\nabla_G} \otimes_{\ul\bF} f_* H\dbr{f^*\nabla_E, f^*\nabla_F})
    \ar[r] &
\RGl(f_*(H\dbr{f^*\nabla_F, f^*\nabla_G} \otimes_{\ul\bF} H\dbr{f^*\nabla_E, f^*\nabla_F}))
    \ar[d] \\ &
\RGl(f_*(H(\dbr{f^*\nabla_F, f^*\nabla_G} \otimes_{\ul\bF} \dbr{f^*\nabla_E, f^*\nabla_F})))
\end{tikzcd}\end{small}\]
where the horizontal maps are $\RGl$ applied to the lax monoidal structure maps of $f_*$
and the bottom vertical map is $\RGl \circ f_*$ applied to the lax monoidal structure map of
$H$. Pasting these last two diagrams along common edges, and using the identification
$\RGl \circ f_* \circ H \to[\simeq] \RGl$, we have the commutativity of the following diagram:
\begin{equation}\label{eqn:pullback-func-der-2}\begin{small}\begin{tikzcd}
\RGl(Z, \CCHom{F}{G}) \otimes_{\ul\bF} \RGl(Z, \CCHom{E}{F})
    \ar[d] \ar[r, "\sigma \otimes_{\ul\bF} \sigma"] &
\RGl(Y, \dbr{f^*\nabla_F, f^*\nabla_G}) \otimes_{\ul\bF}
    \RGl(Y, \dbr{f^*\nabla_E, f^*\nabla_F})
    \ar[d, "" right] \\
\RGl(Z, \CCHom{F}{G} \otimes_{\ul\bF} \CCHom{E}{F})
    \ar[r, "\RGl(f_*(\rho) \circ \mu)" below] &
\RGl(Y, \dbr{f^*\nabla_F, f^*\nabla_G} \otimes_{\ul\bF}
        \dbr{f^*\nabla_E, f^*\nabla_F})
\end{tikzcd}\end{small}\end{equation}
Pasting diagrams \ref{eqn:pullback-func-der-1} and \ref{eqn:pullback-func-der-2} yields the
commutativity of diagram \ref{eqn:pullback-func-der}.

Next, we have the following commutative diagram:
\[\begin{tikzcd}[column sep = huge] &
\RGl \ul\bF[0] \ar[r, "\RGl e"] \ar[d] &
\RGl \dbr{\nabla_E, \nabla_E} \ar[d, "\RGl \eta"] \\
\bF[0] \ar[ru, shift left] \ar[r] \ar[rd, shift right, shorten >=1.5em] &
\RGl f_*\ul\bF[0] \ar[r, "\RGl f_*(e)" below] \ar[d, "\RGl f_*(\rho)" left] &
\RGl f_*\dbr{f^*\nabla_E, f^*\nabla_E} \ar[d, "\RGl f_*(\rho)"] \\ &
\RGl(\ul\bF[0]) = \RGl f_*H(\ul\bF[0]) \ar[r, "\RGl f_*(H(e))" below] &
\RGl f_*H\dbr{f^*\nabla_E, f^*\nabla_E}
\end{tikzcd}\]
where the top right square commutes because \ref{eqn:pullback-func-unit} commutes, the bottom
right square commutes by the naturality of $\rho$, and the top left and bottom left triangles
commute by the lax monoidal structure on $\RGl$.
\end{proof}

\begin{rmk}\label{rmk:RConn-is-Conn-smooth-man}
When $(Y, R) = (X, \Cinf_X), (Z, S) = (X', \Cinf_{X'})$ are smooth manifolds, then all
$\Cinf_X$-- and $\Cinf_{X'}$--modules are flasque and $\RGl$ and $\Gamma$ coincide. In this case,
the two functors of \cref{prop:pullback-func} coincide.
\end{rmk}

\begin{thm}\label{thm:dg-cat-conn-complete}
Consider a differential space $(Y, R, K, d)$ where $Y$ is a smooth manifold and $R = \Cinf_Y$.
Then, $\Conn_{\lambda, d}(Y)$ is a complete $\bC\dg$--category in the sense of
\cite[\S 3]{HiggsLocSys}. In particular, $\Conn_{\lambda, d}(Y)$ is quasi-equivalent to its
completion.
\end{thm}
\begin{proof}
Consider two objects $\nabla_E, \nabla_F$ along with some class
$[\phi] \in H^1\dbr{\nabla_E, \nabla_F} = \Ext^1(\nabla_E, \nabla_F)$.
Then consider the sheaf $E \oplus F$ with differential defined by:
\[
\nabla_{E \oplus F, \phi}(e, f)
= (\nabla_E(e) + \phi(f), \nabla_F(f))
\]
The same arguments as in the case of the $\bC\dg$--category of ordinary connections on
vector bundles on a manifold show that this is a flat $\lambda$--$d$--connection on $E \oplus F$
whose extension class is $[\phi]$.
\end{proof}

\subsection{Prestacks of Connections}
\label{subsec:prest-conn}

\begin{notn}
Given a ring $\bF$, by $\bF\dg-\Cat$, we will denote the $2$--category of differential graded
categories over $\bF$, or equivalently, categories enriched in the monoidal category
$\Ch(\bF\Mod)$ of cochain complexes over $\bF$.
\end{notn}

\begin{prop}\label{prop:pullback-comp-nat}
Let $(X, Q, J, d_J) \to[(f, f^\sharp)] (Y, R, K, d_K) \to[(g, g^\sharp)] (Z, S, L, d_L)$
be a composeable chain of morphisms of differential spaces. Then, consider any
$\lambda$--$d_L$--connection $\nabla_E$. The natural isomorphism
$\alpha_E : (g \circ f)^*(E) \to (f^* \circ g^*)(E)$ is also an isomorphism of connections
$(g \circ f)^*\nabla_E \to (f^* \circ g^*)\nabla_E$ and, in particular,
$\alpha_E \in \ker{\sbr{(g \circ f)^*\nabla_E, f^*g^*\nabla_E}}$. Furthermore, the mappings
\[
\hat{\alpha}_E : \ul\bF[0] \to \dbr{(g \circ f)^*\nabla_E, f^*g^*\nabla_E}
               : 1 \mapsto \alpha_E
\]
as $\nabla_E$ varies, assemble to a natural isomorphism of differential graded functors:
\[\begin{tikzcd}[column sep = 0pt] &
\Conn_{g^\sharp \lambda, d_K}(Y)
    \ar[from=rd, "g^*" above right] & \\
\Conn_{f^\sharp g^\sharp \lambda, d_J}(X) \ar[from=ru, "f^*" above left]
    \ar[from=rr, "(g \circ f)^*" below, ""{name=A, above}] & &
\Conn_{\lambda, d_L}(Z)
\ar[Rightarrow, from=A, to=1-2, ]
\end{tikzcd}\]
\end{prop}
\begin{proof}
Let $\set{e_i}_i$ be a local frame for $E$ and $\set{\epsilon_i}_i$, its dual frame.
Let $[A_{ij}]$ be the connection matrix of $\nabla_E$ in this frame.
Then we get pulled back frames
$\set{f^\sharp g^\sharp e_i}_i$ of $f^*g^*E$ and $\set{(g \circ f)^\sharp e_i}_i$ of
$(g \circ f)^*E$. The isomorphism $\alpha_E$ sends $(g \circ f)^\sharp e_i$ to
$f^\sharp g^\sharp e_i$. This also implies $f^\sharp g^\sharp(\epsilon_i) \circ \alpha_E
= (g \circ f)^\sharp \epsilon_i$
The same holds for the $p_j$ and the $\phi_j$.
Next, by definition, on local sections of $L^{\wedge \bullet}$, $(g \circ f)^\sharp$ is the
composite
$L^{\wedge \bullet} \to[g^\sharp] g_*K^{\wedge \bullet} \to[g_*f^\sharp] g_*f_*J^{\wedge \bullet}$
which on an open $V \subset Z$ is just the map
$L^{\wedge \bullet}(V)
\to[g^\sharp] K^{\wedge \bullet}(g^{-1}(V))
\to[f^\sharp] J^{\wedge \bullet}(f^{-1}g^{-1}(V))$. Using these observations along with connection
matrices, we can show that $\alpha_E$ is a morphism of connections:
\begin{align*}
 & (\alpha_E \otimes \id_L)((gf)^*\nabla_E((gf)^\sharp(e_j))) \\
=& (\alpha_E \otimes \id_L)\br{\sum_{i} (gf)^\sharp(e_i) \otimes f^\sharp g^\sharp A_{ij}} \\
=& \sum_{i} \alpha_E((gf)^\sharp(e_i)) \otimes (gf)^\sharp A_{ij} \\
=& \sum_{i} f^\sharp(g^\sharp(e_i)) \otimes f^\sharp g^\sharp A_{ij} \\
=& f^*g^*\nabla_E(f^\sharp(g^\sharp(e_j))) \\
=& f^*g^*\nabla_E(\alpha_E((gf)^\sharp(e_j)))
\end{align*}
Of course, a morphism of connections that is an isomorphism of the underlying sheaves is an
isomorphism of connections.

Now, let $\set{p_j}_j$ be local frames for another finite rank locally free $S$--module $F$
and $\set{\phi_j}_j$, its dual frames.
We can then show that the following diagram commutes:
\[\begin{tikzcd}
\ul\bF[0] \otimes_{\ul\bF} \dbr{\nabla_{E}, \nabla_{F}}
    \ar[r, "\tilde{\alpha}_F \otimes_{\ul\bF} \eta"] &
(gf)_*\dbr{(gf)^*\nabla_F, f^*g^*\nabla_F} \otimes_{\ul\bF}
(gf)_*\dbr{(gf)^*\nabla_E, (gf)^*\nabla_F}
    \ar[d, "c'"] \\
\dbr{\nabla_{E}, \nabla_{F}}
    \ar[u, "\cong" left]
    \ar[d, "\cong" left] &
(gf)^*\dbr{(gf)^*\nabla_E, f^*g^*\nabla_F} \\
\dbr{\nabla_{E}, \nabla_{F}} \otimes_{\ul\bF} \ul\bF[0]
    \ar[r, "\eta \otimes_{\ul\bF} \tilde{\alpha}_E" below] &
(gf)_*\dbr{f^*g^*\nabla_E, f^*g^*\nabla_F}
  \otimes_{\ul\bF} (gf)_*\dbr{(gf)^*\nabla_E, f^*g^*\nabla_E}
    \ar[u, "c'" right]
\end{tikzcd}\]
where the $\eta$ are as in \cref{prop:pullback-conn-higher-forms},
the $\tilde\alpha$ are the $(gf)_*\hat\alpha$ composed with the canonical maps
$\ul\bF[0] \to (gf)_*\ul\bF[0]$, and
the $c'$ are the composition maps $(gf)_*c$ composed with the lax monoidal structure maps of
the pushforward functor $(gf)_*$. To see this, we first observe that:
\begin{align*}
 & c'((\tilde\alpha_F \otimes_{\ul\bF[0]} \eta)(
        1 \otimes_{\ul\bF[0]} (\epsilon_i p_j \omega)
   ) \\
=& c(\alpha_F \otimes_{\ul\bF[0]} (gf)^\sharp(\epsilon_i) (gf)^\sharp(p_j) (gf)^\sharp(\omega)) \\
=& (gf)^\sharp(\epsilon_i) \alpha_E((gf)^\sharp(p_j)) (gf)^\sharp(\omega) \\
=& (gf)^\sharp(\epsilon_i) f^\sharp(g^\sharp(p_j)) (gf)^\sharp(\omega)
\end{align*}
Next, we observe that:
\begin{align*}
 & c'((\eta \otimes_{\ul\bF[0]} \tilde\alpha_F)(
        (\epsilon_i p_j \omega) \otimes_{\ul\bF[0]} 1
   ) \\
=& c(f^\sharp(g^\sharp(\epsilon_i)) f^\sharp(g^\sharp(p_j)) f^\sharp(g^\sharp(\omega))
     \otimes_{\ul\bF[0]} \alpha_E) \\
=& (f^\sharp(g^\sharp(\epsilon_i)) \circ \alpha_E) f^\sharp(g^\sharp(p_j))
   (gf)^\sharp(\omega) \\
=& (gf)^\sharp(\epsilon_i) f^\sharp(g^\sharp(p_j)) (gf)^\sharp(\omega)
\end{align*}
This shows that the above diagram commutes.
Applying $\Gamma$ to the diagram and using its lax monoidal structure now shows that
$\tilde{\alpha}$ is a $\bC\dg$--natural transformation.
\end{proof}

\begin{prop}\label{prop:pullback-unit-nat}
Let $(Y, R, K, d)$ be a differential space. Then, let $\id_Y$ denote the identity morphism
of differential space and the underlying ringed space. Consider any $\lambda$--$d$--connection
$\nabla_E$. The natural isomorphism $\beta_E : E \to \id_Y^*E$ is also an isomorphism of connections
$\nabla_E \to \id_Y^*\nabla_E$, and, in particular,
$\beta_E \in \ker{\sbr{\nabla_E, \id_Y^*\nabla_E}}$. Furthermore, the mappings:
\[
\hat{\id}_E : \ul\bF[0] \to \Gamma(\dbr{\nabla_E, \id_Y^*\nabla_E}) : 1 \mapsto \beta_E
\]
as $\nabla_E$ varies, assemble to a natural isomorphism of differential graded functors:
\[\begin{tikzcd}
\Conn_{\lambda, d}(Y)
    \ar[from=r, bend left, "\id" below, ""{above, name=A}]
    \ar[from=r, bend right, "\id_Y^*" above, ""{below, name=B}] &
\Conn_{\lambda, d}(Y)
\ar[Rightarrow, from=A, to=B]
\end{tikzcd}\]
\end{prop}
\begin{proof}
Consider the same frames as in \cref{prop:pullback-comp-nat}. Then, $\beta_E$ sends $\set{e_i}_i$ to
the pulled back frame $\set{\id_Y^\sharp e_i}_i$. We will first show that $\beta_E$ is a morphism of
connections:
\begin{align*}
 & (\beta_E \otimes \id_K)(\nabla_E(e_j)) \\
=& (\beta_E \otimes \id_K)\br{\sum_{i} e_i \otimes A_{ij}} \\
=& \sum_{i} \id_Y^\sharp(e_i) \otimes \id_Y^\sharp(A_{ij}) \\
=& \id_Y^*\nabla_E(\id_Y^\sharp(e_j)) \\
=& \id_Y^*\nabla_E(\beta_E(e_i))
\end{align*}
Next, we consider the following diagram:
\[\begin{tikzcd}
\ul\bF[0] \otimes_{\ul\bF} \dbr{\nabla_{E}, \nabla_{F}}
    \ar[r, "\tilde{\beta}_F \otimes_{\ul\bF} \id"] &
\dbr{\nabla_F, id_Y^*\nabla_F} \otimes_{\ul\bF}
\dbr{\nabla_E, \nabla_F}
    \ar[d, "c'"] \\
\dbr{\nabla_{E}, \nabla_{F}}
    \ar[u, "\cong" left]
    \ar[d, "\cong" left] &
\dbr{\nabla_E, id_Y^*\nabla_F} \\
\dbr{\nabla_{E}, \nabla_{F}} \otimes_{\ul\bF} \ul\bF[0]
    \ar[r, "\eta \otimes_{\ul\bF} \tilde{\beta}_E" below] &
\dbr{id_Y^*\nabla_E, id_Y^*\nabla_F}
  \otimes_{\ul\bF} \dbr{\nabla_E, id_Y^*\nabla_E}
    \ar[u, "c'" right]
\end{tikzcd}\]
where the maps are all obtained just as in \cref{prop:pullback-comp-nat}. We can show that this
diagram commutes in a frame as follows. For this, we first observe that:
\begin{align*}
 & c'((\tilde\beta_F \otimes_{\ul\bF[0]} \id)(
        1 \otimes_{\ul\bF[0]} (\epsilon_i p_j \omega)
   ) \\
=& c(\beta_F \otimes_{\ul\bF[0]} \epsilon_i p_j \omega) \\
=& \epsilon_i \beta_F(p_j) \omega \\
=& \epsilon_i \id_Y^\sharp(p_j) \omega
\end{align*}
Next, we observe that:
\begin{align*}
 & c'((\eta \otimes_{\ul\bF} \tilde{\beta}_E)(\epsilon_i p_j \omega \otimes_{\ul\bF} 1)) \\
=& c'(\id_Y^\sharp(\epsilon_i) \id_Y^\sharp(p_j) \omega \otimes_{\ul\bF} \beta_E) \\
=& (\id_Y^\sharp(\epsilon_i) \circ \beta_E) \id_Y^\sharp(p_j) \omega \\
=& \epsilon_i \id_Y^\sharp(p_j) \omega
\end{align*}
where last equality holds because
$\id_Y^\sharp(\epsilon_i)(\beta_E(e_j))
= \id_Y^\sharp(\epsilon_i)(\id_Y^\sharp(e_j))
= \delta_{ij}$, since the dual of the pulled back frame is equal to the pulled back dual frame.
This shows that the above diagram commutes. Applying $\Gamma$ to this diagram and using its lax
monoidal structure shows that $\tilde{\beta}$ is a $\bC\dg$--natural transformation.
\end{proof}

\begin{thm}\label{thm:Conn-functor}
Let $\fD$ be the category whose objects are tuples $(Y, R, K, d_K, \lambda)$, where first four
entries form a differential $\bF$--space $(Y, R, K, d_K)$, and $\lambda \in R(Y)$ is a global
section of $R$ with $\lambda \in \ker{d}$; and, whose morphisms
$f = (f_0, f^\sharp) : (Y, R, K, d_K, \lambda) \to (Z, S, L, d_L, \lambda')$ are morphisms of the
underlying differential spaces such that $f^\sharp \lambda' = \lambda$.
Then, the mapping:
\[\begin{array}{ccccc}
\Conn &:& \fD^\op &\to& \bF\dg-\Cat \\
      &:& (Y, R, K, d_K, \lambda) &\mapsto& \Conn_{\lambda, d_K}(Y) \\
      &:& (f, f^\sharp)  &\mapsto& (f^* : \Conn_{\lambda', d_L}(Z) \to \Conn_{\lambda, d_K}(Y))
\end{array}\]
along with the natural transformations $\alpha$ of \cref{prop:pullback-comp-nat} and
$\beta$ of \cref{prop:pullback-unit-nat} is a pseudofunctor.
\end{thm}
\begin{proof}[Proof sketch]
After \cref{prop:pullback-func}, \cref{prop:pullback-comp-nat} and \cref{prop:pullback-unit-nat}, it
only remains to verify the pseudofunctor coherence conditions are satisfied by the natural
transformations $\alpha$ and $\beta$. By the description in local frames of
$g^*, f^*, (gf)^*, \id_Y^*$, $\alpha$ and $\beta$, it is not difficult to see that these coherence
conditions follow from those satisfied by the $\alpha$ and $\beta$ as natural transformations of
pullback functors of sheaves.
\end{proof}

\begin{rmk}
It should not be significantly more difficult to replace $\Gamma$ with $\RGl$ and $\Conn$ with its
analogue $\RConn$ defined using the $\RConn_{\lambda, d}(Y)$ place of $\Conn_{\lambda, d}(Y)$, in
the results of this section so far.
This would involve using the lax monoidal structure of $\RGl$ and some properties of the
Godement resolution as we did in the results of \cref{subsec:Hom-conn-cmplx}. However, we will
not pursue this generality further in the remainder of the paper, as our main results will concern
smooth manifolds --- see \cref{rmk:RConn-is-Conn-smooth-man}.
\end{rmk}

\begin{defn}\label{defn:diff-sp-cat}
Let $\Diff$ denote the category of differential spaces.
In the context of \cref{thm:Conn-functor}, if $\lambda \in \bF$, then, we get a functor:
\[
I_\lambda : \Diff \to \fD : (Y, R, K, d) \mapsto (Y, R, K, d, \lambda)
\]
by considering the image of $\lambda$ under the structure map $\bF \to R(Y)$. We denote the
composition:
\[
\Diff^\op \to[I_\lambda] \fD^\op \to[\Conn] \bF\dg-\Cat
\]
by $\Conn_\lambda$.
\end{defn}

\begin{defn}[$\dg$--Prestacks]
Given any category $\sC$, we will call a pseudofunctor $F : \sC^\op \to \bF\dg-\Cat$, an
$\bF\dg$--prestack over $\sC$. We will also call them $dg$--prestacks, when $\bF$ is clear from
context. Let $h : \bF\dg-\Cat \to \Cat$ denote the functor sending a $\bF\dg$--category to its
homotopy category. Then, we will call the composition of pseudofunctors $hF$, the ordinary
prestack associated to $F$.
\end{defn}

\begin{rmk}
Instead of the homotopy category, we could also use the underlying category in the previous
definition. However, in our case, all $\dg$--categories will have Hom complexes concentrated in
non-negative degrees and hence will homotopy categories that are the same as the underlying
categories.
\end{rmk}

\begin{defn}[$\dg$--Prestack of Connections]\label{defn:mod-prest-conn}
Consider the contexts of \cref{thm:Conn-functor} and \cref{defn:diff-sp-cat}.
Let $X$ be a fixed smooth manifold. Then, by \cref{exm:prt-ext-diff} and
\cref{exm:prt-ext-diff-morphism}, we have a functor:
\[
\diff_{X} : \Mfd \to \Diff : U \mapsto (U \times X, \sA_{U \times X}^0, \sA^\bullet_{X/U}, d_{X/U})
\]
For any $\lambda \in\bC$, we call the composition of pseudofunctors:
\[
\Mfd^\op \to[\diff_{X}^\op] \Diff^\op \to[I_\lambda^\op] \fD^\op \to[\Conn] \bC\dg-\Cat
\]
the $\dg$--prestack of complex flat $\lambda$--$d_\bC$--connections on $X$ and denote it
$\mcM_{\lambda}(X)$. For $\lambda = 1$, we will simply call this the $\dg$--prestack of
complex flat connections and denote it as $\mcM_{dR}(X)$.
\end{defn}

\begin{rmk}
From the definition, it is immediate that $\mcM_\lambda(X)(U)$ is the $\dg$--category of flat
$\lambda$--$d_{X/U}$--connections on $U \times X$ or smooth families of finite rank
locally free $\Cinf_X$--modules equipped with flat $\lambda$--$d_{\bC}$--connections parametrized by
$U$ as in \cref{exm:conn-family}.
Furthermore, $\mcM_{\lambda}(X)(U)$ acts on morphisms of manifolds $f : V \to U$ by pullback of
connections along $f \times \id_X$.
Of course, here we are using the equivalence of categories of finite rank locally free
$\Cinf_X$--modules and of smooth vector bundles.
\end{rmk}

\begin{defn}[$\dg$--Prestack of Higgs Bundles]\label{defn:mod-prest-Higgs}
Consider the contexts of \cref{thm:Conn-functor} and \cref{defn:diff-sp-cat}.
Let $X$ be a fixed complex manifold. Then, by \cref{exm:prt-Dol-op} and
\cref{exm:prt-Dol-morphism}, we have a functor:
\[
\Dol_{X} : \Mfd \to \Diff :
    U \mapsto (U \times X, \sA^0_{U \times X}, \sA^\bullet_{X/U}, \oprt^\bullet_{X/U})
\]
We call the composition of pseudofunctors:
\[
\Mfd^\op \to[\Dol_X^\op] \Diff^\op \to[I_1^\op] \fD^\op \to[\Conn] \bC\dg-\Cat
\]
the $\dg$--prestack of finite rank locally free Higgs sheaves and denote it as
$\mcM_{Dol}(X)$.
\end{defn}

\begin{rmk}
By construction, $\mcM_{Dol}(X)(U)$ is the $\dg$--category of $1$--$\oprt^\bullet_{X/U}$--connections
on $U \times X$, which as in \cref{exm:Higgs-bun-family}, are smooth families of Higgs bundles on
$X$ parametrized by $U$. The action on morphisms of manifolds $f : V \to U$ is again by pullback
along $f \times \id_X$.
\end{rmk}

\begin{notn}
For a morphism of smooth manifolds $f : V \to U$, we will denote pullbacks of connections on
$U \times X$ with respect to various differentials to those on $V \times X$ along $f \times \id_X$
by $f^<$ and $f_>$ instead of $(f \times \id_X)^*$ and $(f \times \id_X)_*$ respectively, for
brevity.
\end{notn}

\begin{prop}\label{prop:htpy-cat-func}
The mapping $h : \bF\dg-\Cat \to \Cat$ sending an $\bF\dg$--category to its homotopy category is a
strict functor.
\end{prop}
\begin{proof}
This follows from the fact that taking the $i$--th cohomology is a functor
$H^i : \Ch(\bF\Mod) \to \bF\Mod$ for all $i$.
\end{proof}

\begin{prop}\label{prop:ext-compl-func}
The mapping $(-)^{ext} : \bF\dg-\Cat \to \bF\dg-\Cat$ sending an $\bF\dg$--category $\sA$ to its
completion $\sA^{ext}$ under extensions in the sense of \cite[\S 3]{HiggsLocSys} is a strict
functor.
\end{prop}
\begin{proof}
We first recall the construction of $\sA^{ext}$ for any $\bF\dg$--category $\sA$. For this, we
consider the $\bF\dg$--category $\ol\sA$ whose objects are pairs $(A, \eta)$, where $A$ is an object
of $\sA$ and $\eta \in \Hom_{\sA}^1(A, A)$ with $d\eta + \eta^2 = 0$. The complex
$\Hom_{\ol\sA}^\bullet((U, \eta), (U', \eta'))$ has the same underlying graded $\bF$--module
as $\Hom_\sA^\bullet(U, U')$ but with differential:
\[
\ol{d}(f) = d(f) + \eta'f - (-1)^{\deg(f)} f \eta
\]
With this definition, $\sA \to \ol\sA : A \mapsto (A, 0)$ is a fully-faithful embedding.
$\sA^{ext}$ is then the full sub-category of $\ol\sA$ consisting of objects that can be obtained
from objects of $\sA$ by finitely many extensions. For a $\bF\dg$--functor $F : \sA \to \sB$,
the functor $F^{ext} : \sA^{ext} \to \sB^{ext}$ is given by $F^{ext}(A, \eta) = (F(A), F(\eta))$,
where one verifies that $(F(A), F(\eta)) \in \sB^{ext}$ by the properties of an $\bF\dg$--functor.
From this description, the functoriality of $(-)^{ext}$ is immediate.
\end{proof}

\begin{defn}[Prestack of Higgs Bundles]\label{defn:Dol-prest}
Given a complex manifold $X$, we obtain a composite pseudofunctor:
\[\begin{tikzcd}[column sep=huge]
h\mcM^{\simeq}_{Dol}(X) : \Mfd^\op \ar[r, "\mcM_{Dol}(X)"] & \bC\dg-\Cat \ar[r, "h"] &
    \Cat \ar[r, "(-)^\simeq"] & \Grpd
\end{tikzcd}\]
where the last map sends a category to its core or maximal subgroupoid.
By the Grothendieck construction, we then obtain a prestack
\[
\sM_{Dol}(X) = \int h\mcM^{\simeq}_{Dol}(X) \to \Mfd
\]
which we call the prestack of Higgs bundles on $X$.
\end{defn}

\begin{rmk}
The underlying category of the prestack $\sM_{Dol}(X)$ has the following concrete description:
\begin{enumerate}
\item Objects are tuples $(U, E, \nabla_E)$ where:
    \begin{enumerate}[label=(\alph*)]
    \item $U$ is a smooth manifold, and
    \item $(E, \nabla_E)$ is a smooth family of Higgs bundles on $X$ parametrized by $U$.
    \end{enumerate}
\item Morphisms $(V, F, \nabla_F) \to (U, E, \nabla_E)$ are tuples $(f, \alpha)$ where:
    \begin{enumerate}[label=(\alph*)]
    \item $f : V \to U$ is a smooth map, and
    \item $\alpha : f^<(E, \nabla_E) \to (F, \nabla_F)$ is a morphism of families of Higgs bundles
    in the sense of \cref{defn:conn}.
    \end{enumerate}
\end{enumerate}
\end{rmk}

\begin{defn}[Prestack of Flat Bundles]\label{defn:dR-prest}
Given a complex manifold $X$, we obtain a composite pseudofunctor:
\[\begin{tikzcd}[column sep=huge]
h\mcM^{\simeq}_{dR}(X) : \Mfd^\op \ar[r, "\mcM_{dR}(X)"] & \bC\dg-\Cat \ar[r, "h"] &
    \Cat \ar[r, "(-)^\simeq"] & \Grpd
\end{tikzcd}\]
where the last map sends a category to its core or maximal subgroupoid.
By the Grothendieck construction, we then obtain a prestack
\[
\sM_{dR}(X) = \int h\mcM^{\simeq}_{dR}(X) \to \Mfd
\]
which we call the prestack of flat bundles on $X$.
\end{defn}

\begin{rmk}
The underlying category of the prestack $\sM_{dR}(X)$ has the following concrete description:
\begin{enumerate}
\item Objects are tuples $(U, E, \nabla_E)$ where:
    \begin{enumerate}[label=(\alph*)]
    \item $U$ is a smooth manifold, and
    \item $(E, \nabla_E)$ is a smooth family of flat connections on $X$ parametrized by $U$.
    \end{enumerate}
\item Morphisms $(V, F, \nabla_F) \to (U, E, \nabla_E)$ are tuples $(f, \alpha)$ where:
    \begin{enumerate}[label=(\alph*)]
    \item $f : V \to U$ is a smooth map, and
    \item $\alpha : f^<(E, \nabla_E) \to (F, \nabla_F)$ is a morphism of connections
    in the sense of \cref{defn:conn}.
    \end{enumerate}
\end{enumerate}
\end{rmk}

\begin{rmk}\label{rmk:ext-compl-equiv-prest}
We could also consider the pseudofunctors:
\[\begin{tikzcd}[column sep=large]
h\mcM_{Dol}^{ext, \simeq}(X) : \Mfd^\op \ar[r, "\mcM_{Dol}(X)"] & \bC\dg-\Cat \ar[r, "(-)^{ext}"] &
    \bC\dg-\Cat \ar[r, "h"] & \Cat \ar[r, "(-)^\simeq"] & \Grpd
\end{tikzcd}\]
and
\[\begin{tikzcd}[column sep=large]
h\mcM_{dR}^{ext, \simeq}(X) : \Mfd^\op \ar[r, "\mcM_{dR}(X)"] & \bC\dg-\Cat \ar[r, "(-)^{ext}"] &
    \bC\dg-\Cat \ar[r, "h"] & \Cat \ar[r, "(-)^\simeq"] & \Grpd
\end{tikzcd}\]
However, by \cref{thm:dg-cat-conn-complete}, they are equivalent to $h\mcM_{Dol}(X)$ and
$h\mcM_{dR}(X)$ respectively.
\end{rmk}

\subsection{Geometricity}
\label{subsec:geometricity}

\begin{thm}\label{thm:mod-st-Dol-dR}
For a complex manifold $X$, the prestacks $\sM_{Dol}(X)$ and $\sM_{dR}(X)$ of \cref{defn:Dol-prest}
and \cref{defn:dR-prest} respectively are stacks with respect to the usual open cover topology on
the category of smooth manifolds.
\end{thm}
\begin{proof}
We will show that the conditions of \cite[Definition 2.5.1]{AlperModuli}.
Consider an open cover $U = \bigcup_{i \in I} U_i$ of a smooth manifold $U$, and smooth families
of flat bundles $(E_i, \nabla_{E_i})$ on $X$ parametrized by $U_i$ along with isomorphisms
$\alpha_{i, j} : (E_i, \nabla_{E_i}) \to (E_j, \nabla_{E_j})$ satisfying the cocycle condition
$\alpha_{j, k} \circ \alpha_{i, j} = \alpha_{i, k}$ on intersections $U_i \cap U_j \cap U_k$.
By the gluing of vector bundles, we get a vector bundle $E$ on $U \times X$ with isomorphisms
$\phi_i : E|_{U_i} \to E_i$ such that $\phi_j = \alpha_{i, j} \circ \phi_i$ on intersections
$U_i \cap U_j$. The isomorphism $\phi_i$ provides a sheaf map
$\nabla_i = (\phi_i^{-1} \otimes \id) \nabla_{E_i} \circ \phi_i$ and it is straighforward to verify
that $(E|_{U_i}, \nabla_i)$ is a smooth family of flat bundles on $X$ parametrized by $U_i$.
Furthermore, the maps $\nabla_i, \nabla_j$ agree on $U_i \cap U_j$. We can then apply
\cref{cor:conn-coll} to obtain a family of flat bundles $(E, \nabla)$ on $X$ parametrized by $U$.
Of course, by construction $\phi$ is an isomorphism
$(E|_{U_i}, \nabla|_{U_i}) \to (E_i, \nabla_{U_i})$ is an isomorphism.

Next, consider another smooth manifold $W$ and smooth family $(F, \nabla_F)$ of flat bundles on $X$
parametrized by $W$, along with a morphism
$(f_i, \alpha_i) : (U_i, E|_{U_i}, \nabla_E|_{U_i}) \to (W, F, \nabla_F)$ in
$\sM_{dR}(X)$, such that
$(f_i|_{U_i \cap U_j}, \alpha_i|_{U_i \cap U_j})
= (f_j|_{U_i \cap U_j}, \alpha_{j}|_{U_i \cap U_j})$. By the gluing of morphisms of vector bundles,
we obtain a unique morphism $\phi : f^<F \to E$ with $\phi|_{E_{U_i}} = \alpha_i$.
Then, the fact that $\phi$ is a morphism of families of flat bundles can be checked locally, and, by
hypothesis, the $\alpha_i$ are morphisms of families of flat bundles.

This shows that $\sM_{dR}(X)$ is a stack. The argument for $\sM_{Dol}(X)$ is similar.
\end{proof}

\begin{prop}\label{prop:diffeological-atlas-simple}
Let $\mcY$ be a diffeological stack \cite[Definition 2.9 and following paragraphs]{RV18}.
Then, $\mcY$ admits an atlas or presentation $A \to \mcY$ where $A$ is a disjoint union of open
subsets of real Euclidean spaces.
\end{prop}
\begin{proof}
Choose a diffeological atlas $A \to \mcY$: that is, a representable and epimorphic subduction.
Then, consider the collection of all plots $p_i : P_i \to A, i \in I$. Then,
$P := \coprod_{i} P_i$ is a diffeological space since it is a disjoint union of open subsets of
Euclidean spaces. The morphism $p : P \to A$ induced by the universal property of coproducts
can be seen to be representable by diffeological spaces. For this, let $M$ be a smooth manifold
and $f : M \to A$ be a diffeologically smooth map and observe that
$P \times_A M = \coprod_i P_i \times_A M$ is a disjoint union of smooth manifolds and is, hence a
diffeological space. Next, it is immediate from the definition of subduction
\cite[paragraph before Example 2.6]{RV18} that $p$ is one.
Then, it follows from the defintion of epimorphism of stacks \cite[Definition 2.3]{BX11} that any
subduction of diffeological spaces is an epimorphism of the associated stacks. The composite
$P \to A \to \mcY$ is thus an atlas and its domain is a disjoint union of open subsets of real
Euclidean spaces.
\end{proof}

\begin{prop}\label{prop:diffeological-rep-morphism}
Let $f : \mcX \to \mcY$ be a morphism of stacks on the site of smooth manifolds that is
representable by diffeological spaces in the sense of \cite[Definition 2.9]{RV18}. Then, if
$\mcY$ is diffeological, so is $\mcX$.
\end{prop}
\begin{proof}
By \cref{prop:diffeological-atlas-simple}, choose an atlas $P = \coprod_i P_i \to \mcY$ where
each $P_i$ is an open subset of some real Euclidean space. Then,
$P \times_\mcY \mcX = \coprod_i P_i \times_\mcY \mcX \to \mcX$ is
a representable, epimorphic subduction because these properties are stable under pullback.
Therefore, it suffices to show that $P \times_\mcY \mcX$ is a diffeological spaces.
For this, we first observe that each $P_i \times_\mcY \mcX$ is a diffeological space since $f$ is
representable by diffeological spaces, by hypothesis, and $P_i$ is a smooth manifold, by
construction. Then, $P \times_\mcY \mcX$, being a disjoint union of diffeological spaces, is also a
diffeological space.
\end{proof}

\begin{thm}\label{thm:mod-st-Dol-dR-diff}
Let $\sM_{\Vect}(X)$ denote the moduli stack of vector bundles on a complex manifold $X$.
The forgetful mappings of stacks:
\[\begin{array}{ccccccc}
\sM_{dR}(X) & \to & \sM_{\Vect}(X) & : & (U, E, \nabla_E) & \mapsto & (U, E) \\
\sM_{Dol}(X) & \to & \sM_{\Vect}(X) & : & (U, E, D''_E) & \mapsto & (U, E)
\end{array}\]
are representable by diffeological spaces. In particular, the stacks $\sM_{dR}(X)$ and
$\sM_{Dol}(X)$ are diffeological stacks in the sense of \cite[Definition 2.9]{RV18}.
\end{thm}
\begin{proof}
We need to show that for any smooth manifold $U$ and any map $c_E : U \to \sM_\Vect(X)$
corresponding to a smooth vector bundle $E \to U \times X$, the fibre products
$P_{Dol, E} := \sM_{Dol}(X) \times_{\sM_\Vect(X)} U$ and
$P_{dR, E} := \sM_{dR}(X) \times_{\sM_\Vect(X)} U$
are representable by diffeological spaces. The stack $P_{dR, E}$ has the following concrete
description by \cite[\href{https://stacks.math.columbia.edu/tag/0040}{Tag 0040}]{stacks-project}:
\begin{itemize}
\item Objects are $(v, F, \nabla_F, \alpha)$ where:
    \begin{itemize}
    \item $v : V \to U$ is a smooth map from a smooth manifold $V$,
    \item $(F, \nabla_F)$ is a smooth family of flat bundles on $X$ parametrized by $V$, and
    \item $\alpha : v^<E \to F$ is an isomorphism of vector bundles.
    \end{itemize}
\item Morphisms $(w, G, \nabla_G, \beta) \to (v, F, \nabla_F, \alpha)$ are tuples $(f, a)$ where:
    \begin{itemize}
    \item $f : W \to V$ is a smooth map commuting with the maps to $U$, and
    \item $a : f^<(F, \nabla_F) \to (G, \nabla_G)$ is an isomorphism of smooth families of flat
    bundles on $X$ parametrized by $W$, such that the following diagram of vector bundles commutes:
    \[\begin{tikzcd}
    f^<v^<E \ar[r, "f^<\alpha"] \ar[d, "\cong" left] & f^<F \ar[d, "a"] \\
    w^<E \ar[r, "\beta" below] & G
    \end{tikzcd}\]
    \end{itemize}
\end{itemize}
We notice that by the commutativity of the above diagram, if $(f, a)$ and $(f, b)$ are two morphism
with the same source and target, then $a = b$. Thus, $P_{dR, E}$ is a prestack whose fibres are
setoids.

We will show that this prestack (in setoids) is represented by the diffeological space
defined by the following fibre product:
\[\begin{tikzcd}
S_E \ar[r] \ar[d] \ar[rd, phantom, "\lrcorner" very near start] &
\Cinf(X, \End(E) \otimes \sA_{X/U}^1) \ar[d, "p_*"] \\
U \ar[r, "\hat{\id_U}" below] &
\Cinf(X, U \times X)
\end{tikzcd}\]
where the right vertical map is given by post-composition by the vector bundle projection
$p : \End(E) \otimes \sA^1_{X/U} \to U \times X$,
and the bottom horizontal map is defined by $u \mapsto (x \mapsto (u, x))$. Of course, this fibre
product represents the $2$--fibre product when all vertices are considered as the corresponding
prestacks.
First, we choose any fixed smooth family of flat $1$--$d_{X/U}$--connection $\nabla_E$ --- choose
any Hermitian metric on $E$ and the associated flat $1$--$d_{X/U}$--connection provided by
\cref{prop:prt-Chern-conn}, for example.
Now, consider a diffeologically smooth map $s : V \to S_E$ from a smooth manifold $V$ given by two
diffeologically smooth maps $s_0 : V \to \Cinf(X, \End(E) \otimes \sA_{X/U}^1), s_1 : V \to U$ such
that $\hat{\id}_U \circ s_1 = p \circ s_0$.
This data corresponds to a section $\hat{s}$ of
$\End(s_1^<E) \otimes s_1^<\sA_{X/U}^1 \to V \times X$.
We have a canonical isomorphism
$\sA^1_{X/V} = \pi_X^*\sA^1_X \cong (s_1 \times \id_X)^*\pi_X^*\sA^1_X = s_1^<\sA^1_{X/U}$,
yielding an isomorphism
$\End(s_1^<E) \otimes s_1^<\sA_{X/U}^1 \cong \End(s_1^<E) \otimes \sA_{X/V}^1$, which, in turn,
yields a section $\tilde{s}$ of $\End(s_1^<E) \otimes \sA_{X/V}^1$. This $\tilde{s}$
is a morphism of vector bundles $s_1^<E \to s_1^<E \otimes \sA_{X/V}^1$.
Define an object $\Psi_E(s) := (s_1, s_1^<E, s_1^<\nabla_E + \tilde{s}, \id_{s_1^<E})$ of
$P_{dR, E}$, where $s_0$ corresponds to a section of $s_1^<E \to V \times X$.
Then, consider another diffeologically smooth map $t : W \to S_E$ given by two
diffeologically smooth maps $t_0 : W \to C^\infty(X, E), t_1 : W \to U$ such that
$\hat{\id}_U \circ t_1 = p \circ t_0$, and a smooth map $f : W \to V$ over $S_E$. Then, we get a
unique morphism $\Psi_E(f) = (f, a) : \Psi_E(s) \to \Psi_E(t)$ by the discussion in the previous
paragraph. The mapping $\Psi_E$ can be checked to be a morphism of prestacks $S_E \to P_{dR, E}$.
It is fully faithful by construction.
It then suffices to show that $\Psi_E$ is essentially surjective. We first observe that every object
$(v, F, \nabla_F, \alpha)$ in $P_{dR, E}$ is isomorphic to the object
$(v, v^<E, v^<\nabla_E, \id_{v^<E})$. Then, consider the section
$\tilde{s}_0 := (\alpha^{-1} \otimes \id) \circ \nabla_F \circ \alpha - v^<\nabla_E$ of
$\End(v^<E) \otimes \sA^1_{X/V}$. By the isomorphism $\sA^1_{X/V} \cong v^<\sA^1{X/U}$ discussed
above, we obtain a corresponding section $\hat{s}_0$ of $\End(v^<E) \otimes v^<\sA^1_{X/U}$
which, in turn gives a map $s'_0 : V \times X \to \End(E) \otimes \sA^1_{X/U}$ such that
$p \circ s'_0 = v \times \id_X$. Setting $s_0$ to be the maps
$V \to \Cinf(X, \End(E) \otimes \sA^1_{X/ U})$ corresponding to $s'_0$ and $s_1 = v$, we obtain
a map $s = (s_0, s_1) : V \to S_E$. By construction, $\Psi_E(s) = (v, v^<E, v^<\nabla_E, \id)$.

The case of $\sM_{Dol, E}(X) \to \sM_\Vect(X)$ is simpler. We have a similar fibre product of stacks
$P_{Dol, E} = \sM_{Dol}(X) \times_{\sM_{\Vect}(X)} U$, whose objects and morphisms are analogous,
with families of Higgs bundles in place of flat connections. An analogous construction as above but
with the zero map in place of the fixed connection $\nabla_E$ provies an equivalence of prestacks
$\Psi'_E : S_E \to P_{Dol, E}$.

For showing that the stacks $\sM_{Dol}(X)$ and $\sM_{dR}(X)$ are diffeological,
we first observe that we have an equivalence of stacks
$\sM_\Vect(X) \simeq \coprod_{n \in \bN} \HHom(X, B\Gl_n(\bC))$, where $B\Gl_n(\bC)$ is the quotient
stack $[\pt/\Gl_n(\bC)]$. The point $\pt$ is a diffeological spaces and provides a diffeological
atlas $\pt \to B\Gl_n(\bC)$, so that $\sM_\Vect(X)$ is a diffeological stack.
\cref{prop:diffeological-rep-morphism} in combination with
\cite[Theorem 2]{RV18HomStack} then prove that $\sM_{Dol}(X)$ and $\sM_{dR}(X)$ are diffeological.
\end{proof}

\begin{rmk}\label{rmk:Frechet-Lie-groupoid-weak-presentation}
$\HHom(X, B\Gl_n(\bC))$ is presented as a diffeological stack by the Fr\'echet Lie groupoid
$\Map(X, B\Gl_n(\bC))_1 \rightrightarrows \Map(X, B\Gl_n(\bC))_0$ of
\cite[2]{RV18HomStack}. One might think that this means $\sM_\Vect(X)$ is a Fr\'echet differentiable
stack, but the mapping $\Map(X, B\Gl_n(\bC))_0 \to \HHom(X, B\Gl_n(\bC))$, while representable by
diffeological spaces, might fail to be representable by Fr\'echet manifolds, as the Fr\'echet Lie
groupoid only weakly presents the stack. As a result, it is not clear if one can upgade
\cref{thm:mod-st-Dol-dR-diff} to a statement about Fr\'echet differentiable stacks.
\end{rmk}

\begin{rmk}
The proof \cref{thm:mod-st-Dol-dR-diff} was inspired by \cite[Theorem 6.13.]{GeomNonAbHodgeFilt}.
\end{rmk}

\section{Harmonic Bundles}
\label{sec:harmonic}

In non-Abelian Hodge theory, the correspondence between semistable Higgs bundles and flat bundles
is mediated by harmonic bundles. The correspondence, in our case, will be mediated by families
thereof. This section is dedicated to developing the theory of families of harmonic bundles.
This starts with defining a relative or partial version of the Chern connection for a Hermitian
metric. Families of harmonic bundles are then families of smooth bundles that are equipped with both
the structure of a family of Higgs bundles and of a family of flat bundles, with a partial Chern
connection relating these structures in the usual way. The usual facts about harmonic bundles
then carry over to families. This is the subject of \cref{subsec:harmonic-bun-families}.
The standard constructions of duals and tensor products applied to the Higgs part and the flat
connection part of a harmonic bundle family provide suitable notions of duals and tensor products,
and eventually, Hom objects of families of harmonic bundles. That is, tensor products and duals
of families of Higgs bundles and flat bundles related by a harmonic metric are related by the tensor
product and dual metrics respectively. We show these facts in \cref{subsec:cons-harmonic}.
Finally, in \cref{subsec:dg-prest-harmonic} we show that
the $\dg$--category structure of families of Higgs bundles induces one on the collection of families
of harmonic bundles, in the same way that the $\dg$--category structure of Higgs bundles induces the
one on the collection of harmonic bundles.
At the same time, we will see that the pullbacks of families of Higgs bundles and of flat
bundles gives a pullback of families of harmonic bundles: if families of the two structures
are related by a harmonic metric, then their pullbacks are related by the pulled back harmonic
metric.
This gives rise to a $\dg$--prestack of harmonic bundles,
which will mediate the non-Abelian Hodge correspondence at the level of diffeological moduli stacks.

\subsection{Smooth Families of Harmonic Bundles}
\label{subsec:harmonic-bun-families}

\begin{prop}\label{prop:prt-Chern-conn}
Let $(E \to U \times X, \oprt_{E/U})$ be a smooth family of holomorphic vector bundles as in
\cref{exm:hol-bun-family} and $H$, a Hermitian metric on $E$. Then, there exists a unique
flat $1$--$d_{X/U}$--connection $\nabla_{H/U}$ on $E$ such that for each $u \in U$
the pulled back $1$--$d$--connection $\nabla_{H/U, u}$ on $E_u$ is the Chern connection of the
restriction $H_u$ of $H$ to $E_u$.
\end{prop}
\begin{proof}
Consider an open cover $U \times X = \bigcup_{\alpha \in I} W_\alpha$ so that $E|_{W_\alpha}$
trivializes with a local frame $\set{e^\alpha_i}_i$ and transition functions
$g^{\alpha\beta}$ for each $\alpha, \beta \in I$. Let $H_\alpha$ be the matrix of $H$ in this frame.
Let us also suppose we have holomorphic and antiholomorphic coordinates $z_i, \bar{z}_i$ for
$\pi_X(W_\alpha) \subset X$.
Then, for each $\alpha \in I$, set:
\[
A^\alpha_{ij} := \sum_k (\bar{H}_\alpha^{-1})_{ik} \prt_{X/U}((\bar{H}_\alpha)_{kj})
= \sum_{p} \sum_k (\bar{H}_\alpha^{-1})_{ik} \frac{\prt (\bar{H}_\alpha)_{kj}}{\prt z_p} dz_p
\]
We notice that
$\prt_{X/U}((\bar{H}_\alpha)_{kj})|_{(\set{u} \times X) \cap W_\alpha} =
\prt((\bar{H}_\alpha)_{kj}|_{(\set{u} \times X) \cap W_\alpha})$, where
$\prt$ is the Dolbeault operator of $X$. Thus, by construction,
$A^\alpha_u = [A^\alpha_{ij}|_{(\set{u} \times X) \cap W_\alpha}]$ is the connection matrix of the
Chern connection $\nabla_{H_u}$ for each $u \in \pi_U(W_\alpha)$. In particular, the
$\set{A^\alpha_u}_{\alpha \in I}$ satisfy the gluing condition for connections:
\[
A^\beta_u = g^{\alpha\beta}_u d(g^{\beta\alpha}_u)
          + g^{\alpha\beta}_u A^\alpha_u g^{\beta\alpha}_u
\]
where $g^{\alpha\beta}_u = g^{\alpha\beta}|_{(\set{u} \times X) \cap W_\alpha \cap W_\beta}$
and $d(g^{\alpha\beta}_u)$ is the matrix of $(1, 0)$--forms on $X$
obtained by applying the exterior differential of $X$ to $g^{\alpha\beta}_u$ entrywise.
However, by the local formula for the partial exterior differential $d_{X/U}$, we can see that
$d(g^{\beta\alpha}_u)(x) = d_{X/U}(g^{\beta\alpha})(u, x)$ since $d_{X/U}$ has no
derivatives in the $U$ direction. Thus, we have:
\begin{align*}
A^\beta(u, x)
=& A^\beta_u(x) \\
=& g^{\alpha\beta}_u(x) d(g^{\beta\alpha}_u)(x)
   + g^{\alpha\beta}_u(x) A^\alpha_u(x) g^{\beta\alpha}_u(x) \\
=& g^{\alpha\beta}(u, x) d_{X/U}(g^{\beta\alpha})(u, x)
   + g^{\alpha\beta}(u, x) A^\alpha(u, x) g^{\beta\alpha}(u, x) \\
=& \br{g^{\alpha\beta} d_{X/U}(g^{\beta\alpha})
   + g^{\alpha\beta} A^\alpha g^{\beta\alpha}}(u, x)
\end{align*}
That is, the $\set{A^\alpha}_{\alpha \in I}$ satisfy the gluing condition for
$1$--$\prt_{X/U}$--connections and yield such a connection $\nabla_{H/U}$ by
\cref{prop:conn-form-coll}.
If any other $1$--$d_{X/U}$--connections $\nabla$ satisfies that the pulled back connection
$\nabla_u$ is the Chern connection of $H_u$, then the connection matrices of $\nabla$ restricted to
$\set{u} \times X$ must equal $A^\alpha_u$ by the uniqueness of the Chern connection. Since this
holds for all $u \in U$, the connection matrices of $\nabla$ must equal $A^\alpha$. By
\cref{prop:conn-form-coll} again, $\nabla = \nabla_{H/U}$.

Next, observe that:
\[
d_{X/U}(A^\alpha_{ij})(u, x) + \sum_{k} (A^\alpha_{ik} \wedge A^\alpha_{kj})(u, x)
= d(A^\alpha_u(x)) + \sum_{k} A^\alpha_{u, ik}(x) \wedge A^\alpha_{u, kj}(x) = 0
\]
since $\nabla_{H/U, u}$, being a Chern connection, is flat. The flatness of $\nabla_{H/U}$ now
follows from \cref{rmk:conn-higher-forms-loc} and \cref{prop:curvature-is-linear}.
\end{proof}

\begin{defn}[Partial Chern Connection]\label{defn:prt-Chern-conn}
We will call the $1$--$d_{X/U}$--connection $\nabla_{H/U}$ from \cref{prop:prt-Chern-conn} the
partial Chern connection of $H$.
\end{defn}

\begin{prop}\label{prop:prt-Chern-conn-pullback}
In the context of \cref{prop:prt-Chern-conn}, consider a smooth map $f : V \to U$ and pullback
metric $f^<H$, then $f^<\nabla_{H/U} = \nabla_{f^<H/V}$.
\end{prop}
\begin{proof}
The connection matrix of $f^<\nabla_{H/U}$ has entries:
\begin{align*}
 & (f \times \id_X)^\sharp\br{\sum_k (\bar{H}_\alpha^{-1})_{ik}\prt_{X/U}((\bar{H}_\alpha)_{kj})} \\
=& \sum_k (f \times \id_X)^\sharp(\bar{H}_\alpha^{-1})_{ik}
          \cdot (f \times \id_X)^\sharp(\prt_{X/U}((\bar{H}_\alpha)_{kj})) \\
=& \sum_k (f \times \id_X)^\sharp(\bar{H}_\alpha^{-1})_{ik}
          \cdot \prt_{X/V}((f \times \id_X)^\sharp(\bar{H}_\alpha)_{kj})
\end{align*}
by the properties of the pullback of differential forms.
Let $W'_\alpha = (f \times \id_X)^{-1}(W_\alpha)$ and $(f^<H)_\alpha$ be the matrix of $f^<H$ in the
$\set{f^\sharp e^\alpha_i}_i$ frame of $E$ over $W'_\alpha$.
Then, observe that $(f \times \id_X)^\sharp(H_{\alpha, ab}) = (f^<H)_{\alpha, ab}$
by the definition of pullback metrics. This implies
$(f \times \id_X)^\sharp((\bar{H}_\alpha^{-1})_{ab}) = (\ol{f^<H}^{-1})_{ab}$.
Thus, $f^<\nabla_{H/U}$ has the same connection matrices as $\nabla_{f^<H / V}$, and we are done
by \cref{prop:conn-form-coll}.
\end{proof}

\begin{defn}[Smooth Family of Harmonic Bundles]\label{defn:Harmonic-bun-family}
Let $X$ be a complex manifold and $U$, a smooth manifold. Consider a smooth complex vector bundle
$p : E \to U \times X$ equipped with the following data:
\begin{enumerate}
\item a flat $1$--$\oprt_{X/U}^\bullet$--connection, $D''$
\item a flat $1$--$d_{X/U}$--connection, $D$
\end{enumerate}
By the discussion in \cref{exm:Higgs-bun-family}, we also get:
\begin{enumerate}
\item a smooth family of holomorphic structures: that is, a
    $1$--$\oprt_{X/U}^{0, \bullet}$--connection $\oprt_{E/U}$
\item a smooth family of Higgs fields: that is, a
    $0$--$\prt_{X/U}^{\bullet, 0}$--connection $\theta$
\end{enumerate}
Suppose also that there exists a smooth Hermitian metric $H$ on $E$ with partial Chern connection
$\nabla_{H/U}$. Let $\bar{\theta}_{H} : E \to E \otimes \sA_{X/U}^{0, 1}$ denote the map
corresponding to the adjoint of $\theta$ with respect to $H$.
Then, we say that $D$ and $D''$ are $H$--related if the following identity holds:
\begin{equation}\label{eqn:Harmonicity}
D = \nabla_{H/U} + \iota^{1, 0} \circ \theta + \iota^{0, 1} \circ \bar\theta_H
\end{equation}
where $\iota^{j, 1 - j} : E \otimes \sA_{X/U}^{j, 1 - j} \to E \otimes \sA_{X/U}^{1}, j = 0, 1$ are
the inclusions. In this case, we call the triple $(E, D'', D)$, a smooth family of harmonic bundles
on $X$ parametrized by $U$, and a metric such as $H$, a smooth family of harmonic metrics.
\end{defn}

\begin{rmk}\label{rmk:harmonic-bun-family-pt}
When $U = \pt$, \cref{defn:Harmonic-bun-family} reduces to the usual notion of a harmonic bundle
on $X$ as defined in \cite[18]{HiggsLocSys} or \cite[\S 4.2]{GR15}.
\end{rmk}

\begin{prop}\label{prop:harmonic-bun-family-pullback}
In the context of \cref{defn:Harmonic-bun-family}, let $f : V \to U$ be a smooth map.
Then, the triple $(f^<E, f^<D'', f^<D)$ is a smooth family of harmonic bundles on $X$ parametrized
by $U$. More specifically, for any harmonic metric relating $D''$ and $D$, the pulled back metric
$f^<H$ relates $f^<D''$ and $f^<D$.
\end{prop}
\begin{proof}
Choose any harmonic metric $H$ relating $D''$ and $D$.
We saw in \cref{prop:prt-Chern-conn-pullback} that $\nabla_{f^<H/V} = f^<\nabla_{H/U}$.
By considering connection matrices and the definition of pullback connections
(\cref{defn:pullback-conn}), it is straightforward to show that
$pr^{1, 0} \circ f^<D'' = f^<(pr^{1, 0} \circ D'') = f^<\theta$ for $p = 0, 1$ --- the
connection matrices of both sides are entrywise pullbacks of the same submatrix of the connection
matrix of $D''$.
A similar argument shows that $\iota^{1, 0} \circ f^<\theta = f^<(\iota^{1, 0} \circ \theta)$,
and $\iota^{0, 1} \circ f^<\bar{\theta}_H = f^<(\iota^{0, 1} \circ \bar{\theta}_H)$.
We can further show that $f^<\bar{\theta}_H = \ol{(f^<\theta)}_{f^<H}$ as follows.
In a local frame, $\bar{\theta}_{H}$ has connection matrix entries:
\[
\sum_{k, l} (H^{-1})_{ik} \bar{\theta}^T_{kl} H_{lj}
= \sum_{k, l} (H^{-1})_{ik} \bar{\theta}_{lk} H_{lj}
\]
In the pulled back frame, $f^<\bar{\theta}_H$ has connection matrix entries:
\begin{equation}\label{eqn:harmonic-bun-pb-conj-Higgs-field}
(f \times \id_X)^\sharp\br{\sum_{k, l} H^{-1}_{ik} \bar{\theta}_{lk} H_{lj}}
= \sum_{k, l} (f \times \id_X)^\sharp(H^{-1}_{ik})
              (f \times \id_X)^\sharp(\bar{\theta}_{lk})
              (f \times \id_X)^\sharp(H_{lj})
\end{equation}
Recall that $(f \times \id_X)^\sharp(H^{-1}_{ik})(v, x) = H^{-1}_{ik}(f(v), x)$ as a smooth complex
valued function, which is the $(i, k)$--entry of the matrix
$H(f(v), x)^{-1} = (H \circ (f \times \id_X))(v, x)^{-1} = (f^<H)^{-1}(v, x)$. That is,
$(f \times \id_X)^\sharp(H^{-1}_{ik}) = (f^<H)^{-1}_{ik}$.
Next, we recall that pullback of complex differential forms commutes with conjugation, so that
$(f \times \id_X)^\sharp(\bar{\theta}_{lk}) = \ol{(f \times \id_X)^\sharp(\theta_{lk})}
= \ol{(f^<\theta)}_{lk}$. Of course, $(f \times \id_X)^\sharp(H_{lj}) = (f^<H)_{lj}$.
Putting these together, \ref{eqn:harmonic-bun-pb-conj-Higgs-field} becomes:
\[
(f \times \id_X)^\sharp\br{\sum_{k, l} H^{-1}_{ik} \bar{\theta}_{lk} H_{lj}}
= \sum_{k, l} (f^<H)^{-1}_{ik}
              \ol{(f^<\theta)}_{lk}
              (f^<H)_{lj}
\]
which is precisely the corresponding connection matrix entry of $\ol{(f^<H)}_{f^<H}$.

Therefore, we have:
\begin{align*}
 & \nabla_{f^<H/V} + \iota^{1, 0} \circ (pr^{1, 0} \circ f^<D'')
                   + \iota^{0, 1} \circ \ol{(pr^{0, 1} \circ f^<D'')}_{f^<H} \\
=& f^<\nabla_{H/U} + \iota^{1, 0} \circ f^<\theta
                   + \iota^{0, 1} \circ \ol{(f^<\theta)}_{f^<H} \\
=& f^<\nabla_{H/U} + \iota^{1, 0} \circ f^<\theta
                   + \iota^{0, 1} \circ f^<\bar{\theta}_H \\
=& f^<\nabla_{H/U} + f^<(\iota^{1, 0} \circ \theta) + f^<(\iota^{0, 1} \circ \bar{\theta}_H) \\
=& f^<(\nabla_{H/U} + \iota^{1, 0} \circ \theta + \iota^{0, 1} \circ \bar{\theta}_H) \\
=& f^<D
\end{align*}
where the second last step follows by a straightforward comparison of connection matrices again.
Now, this is exactly the desired conclusion.
\end{proof}

\begin{rmk}\label{rmk:harmonic-bun-family-slice}
In light of \cref{prop:harmonic-bun-family-pullback} and \cref{rmk:harmonic-bun-family-pt}, for each
$u \in U$, the triple $(E_u, D''_u, D_u)$ obtined by pullback along the inclusion
$\iota_u : X \cong \set{u} \times X \to U \times X$ is a harmonic bundle on $X$ in the usual sense.
\end{rmk}

\begin{defn}\label{defn:harmonic-op}
Let $(E, D'', D)$ be a smooth family of harmonic bundles on $X$ parametrized by $U$. Then,
suppose $H$ is a harmonic metric on $E$ that relates $D''$ and $D$. We define the
following operators:
\begin{enumerate}
\item $\prt_H := pr^{1, 0} \circ \nabla_{H/U}$
\item $D'_{H} := \iota^{1, 0} \circ \prt_H + \iota^{0, 1} \circ \bar{\theta}_H$
\end{enumerate}
\end{defn}

\begin{prop}\label{prop:harmonic-op-conn}
In the context of \cref{defn:harmonic-op}, the following hold:
\begin{enumerate}
\item $D'_H = D - D''$
\item $D'_H$ is a flat $1$--$\prt_{X/U}^\bullet$--connection
\item $(D'_H)^1 \circ (D'')^0 + (D'')^1 \circ (D'_H)^0 = 0$
\end{enumerate}
\end{prop}
\begin{proof}
(i)
Observe that:
\begin{align*}
D - D''
=& \nabla_{H/U} + \iota^{1, 0} \circ \theta + \iota^{0, 1} \circ \bar{\theta}_H
   - \iota^{0, 1} \circ \oprt_{E/U} - \iota^{1, 0} \circ \theta \\
=& \iota^{0, 1} \circ \oprt_{E/U} + \iota^{1, 0} \circ \prt_H + \iota^{0, 1} \circ \bar{\theta}_H
   - \iota^{0, 1} \circ \oprt_{E/U} \\
=& \iota^{1, 0} \circ \prt_H + \iota^{0, 1} \circ \bar{\theta}_H \\
=& D'_H
\end{align*}

(ii)
Apply \cref{prop:conn-sum}(i) along with the fact that
$d_{X/U} = \prt_{X/U}^\bullet + \oprt_{X/U}^\bullet$ to deduce that $D'_H$ is a
$1$--$\prt_{X/U}^\bullet$--connection.
For flatness, first apply \cref{prop:rest-conn-ext-full}(vi)
along with the flatness of $\nabla_{H/U}$ to obtain the flatness of $\circ \oprt_{X/U}$
and $\circ \prt_{H}$. Then, we can show that $\bar{\theta}_H$ is flat from the flatness of $\theta$.
If $T = [\theta_{ij}]$ is the connection matrix of $\theta$ in a local frame, then the connection
matrix of $\bar{\theta}_H$ is $H^{-1}\bar{T}H$, where we are using $H$ to also denote its matrix in
the chosen frame and $\bar{T}$ is the matrix obtained by conjugating the differential forms giving
the entries of $T$. Then, the matrix of $\bar{\theta}_H^1 \circ \bar{\theta}_H^0$ is:
\[
H^{-1}\bar{T}H \wedge H^{-1}\bar{T}H = H^{-1}\bar{T} \wedge \bar{T}H
\]
where $\wedge$ means matrix multiplication with the product of entries given by wedge product of
differential forms. It is easy to see that the wedge product of forms commutes with conjugation.
Thus, $\bar{T} \wedge \bar{T} = \bar{T \wedge T} = \bar{0} = 0$, by the flatness of $\theta$.
This shows that $\bar{\theta}_H$ is flat. The flatness of $D'$ now follows from
\cref{prop:rest-conn-ext-full}(vi).

(iii)
Combine (i) above with the flatness of $D$ and $D''$, and \cref{prop:conn-sum}(iv).
\end{proof}

\begin{prop}\label{prop:harmonic-op-metric-inv}
Let $(E, D'', D)$ be a smooth family of harmonic bundles on $X$ parametrized by $U$. Then,
suppose $H_i, i = 1, 2$ are harmonic metrics on $E$ that relate $D''$ and $D$. Then,
\[
(D'_{H_1})^k = (D'_{H_2})^k
\]
for all $k = 0, 1, 2, \dots$.
\end{prop}
\begin{proof}
We can see that $D_{H_1} = D - D'' = D_{H_2}$.
It follows from the definition of extensions of connections to higher forms that
$(D'_{H_i})^k = (D'_{H_i})^k$ for all $k = 0, 1, 2, \dots$.
\end{proof}

\begin{defn}[Harmonic Operator]
In light of \cref{prop:harmonic-op-metric-inv}, we will write $D'$ to denote $D'_H$ from this point
onwards, and call it the harmonic operator of the harmonic bundle $(E, D'', D)$.
\end{defn}

\begin{prop}\label{prop:harmonic-op-Higgs-op}
Let $(E, D'', D)$ be a smooth family of harmonic bundles on $X$ parametrized by $U$. Then,
suppose $H$ is a harmonic metric on $E$ that relates $D''$ and $D$. Then,
$(D'')^\bullet$ restricts to a complex of sheaves of $\bF$--modules:
\[
0 \to \ker{(D')^0} \to[(D'')^0] \ker{(D')^{1}} \to[(D'')^1] \ker{(D')^2} \to[(D'')^2] \cdots
\]
Furthermore, $D^k|_{\ker{(D')^k}} = (D'')^k|_{\ker{(D')^k}}$ making the above complex equal to:
\[
0 \to \ker{(D')^0} \to[D^0] \ker{(D')^{1}} \to[D^1] \ker{(D')^2} \to[D^2] \cdots
\]
This gives inclusions of the above complex into the complexes
$(E \otimes \sA_{X/U}^\bullet, (D'')^\bullet)$ and $(E \otimes \sA_{X/U}^\bullet, D^\bullet)$.
\end{prop}
\begin{proof}
By the flatness of $D$, $D''$ and $D'$ (\cref{prop:harmonic-op-conn}(ii)), the relation
$D = D'' + D'$ and \cref{prop:conn-sum}, we have:
\[
(D'')^{k + 1} \circ (D')^k + (D')^{k + 1} \circ (D'')^k = 0
\]
Let $s \in \ker{(D')^k}(W) \subset (E \otimes \sA_{X/U}^k)(W)$ for some open subset
$W \subset U \times X$. Then, the above identity yields:
\[
(D')^{k + 1}((D'')^k(s)) = -(D'')^{k + 1}((D')^k(s)) = 0
\]
which implies $(D'')^k(s) \in \ker{(D')^{k + 1}}$. This yields the complex
$(\ker{(D'_E)^\bullet}, (D''_E)^\bullet)$ and the inclusion into
$(E \otimes \sA_{X/U}^\bullet, (D'')^\bullet)$.
At the same time, $D^k(s) = (D')^k(s) + (D'')^k(s) = (D'')^k(s) \in \ker{(D')^{k + 1}}$. In
particular, $D^k|_{\ker{(D')^k}} = D''|_{\ker{(D')^k}}$, giving the
inclusion into $(E \otimes \sA_{X/U}^\bullet, D^\bullet)$.
\end{proof}

\begin{notn}
For brevity, we will denote a smooth family of harmonic bundles $(E, D'', D)$ as $\sH_E$.
Given such a family, we will write $D''_E, D'_E$ and $D_E$ for $D'', D'$ and $D$.
\end{notn}

\subsection{Constructions Involving Harmonic Bundle Families}
\label{subsec:cons-harmonic}

\begin{prop}\label{prop:harmonic-bun-tensor}
Let $\sH_E, \sH_F$ be two smooth families of harmonic bundles on $X$ parametrized by
$U$. Then, $\sH_E \otimes \sH_F = (E \otimes F, D''_E \otimes D''_F, D_E \otimes D_F)$ is a smooth
family of harmonic bundles with harmonic operator $D'_E \otimes D'_F$.
\end{prop}
\begin{proof}
Let $H_E, H_F$ be harmonic metrics on $\sH_E, \sH_F$ respectively and consider the tensor product
metric $H_E \otimes H_F$.
We will first compute the connection matrix of $\nabla_{H_E \otimes H_F/U}$.
Let $\set{e_i}_i, \set{f_j}_j$ be local frames for $E$ and $F$, with dual frames
$\set{\epsilon_i}_i, \set{\phi_j}_j$ respectively.
We will denote the matrices of $\theta_E, H_E$ and $\theta_F, H_F$ in the $\set{e_i}, \set{f_j}$
frames respectively by the same symbols.
Then, the connection matrix of $\nabla_{H_E \otimes H_F/U}$ has the following entry
corresponding to the basis element $\epsilon_i \otimes e_p \otimes \phi_j \otimes f_q$:
\begin{align*}
 & \sum_{k, l} \br{\ol{H_E \otimes H_K}^{-1}}_{ijkl}
               \prt_{X/U}\br{\br{\ol{H_E \otimes H_F}}_{klpq}} \\
=& \sum_{k, l} \br{\bar{H}_E^{-1} \otimes \bar{H}_K^{-1}}_{ijkl}
          \prt_{X/U}\br{\br{\bar{H}_E \otimes \bar{H}_F}_{klpq}} \\
=& \sum_{k, l} \bar{H}_{E, ik}^{-1}\bar{H}_{F, jl}^{-1}
          \prt_{X/U}\br{\bar{H}_{E, kp}\bar{H}_{F, lq}} \\
=& \sum_{k, l} \bar{H}_{E, ik}^{-1}\bar{H}_{F, jl}^{-1}
          \br{\prt_{X/U}\br{\bar{H}_{E, kp}}\bar{H}_{F, lq} +
              \bar{H}_{E, kp}\prt_{X/U}\br{\bar{H}_{F, lq}}} \\
=& \br{\sum_{l} \bar{H}_{F, jl}^{-1}\bar{H}_{F, lq}}
   \br{\sum_{k} \bar{H}_{E, ik}^{-1}\prt_{X/U}\br{\bar{H}_{E, kp}}}
   + \br{\sum_k \bar{H}_{E, ik}^{-1}\bar{H}_{E, kp}}
     \br{\sum_l \bar{H}_{F, jl}^{-1}\prt_{X/U}\br{\bar{H}_{F, lq}}} \\
=& \delta_{jq}\br{\sum_{k} \bar{H}_{E, ik}^{-1}\prt_{X/U}\br{\bar{H}_{E, kp}}}
   + \delta_{ip} \br{\sum_l \bar{H}_{F, jl}^{-1}\prt_{X/U}\br{\bar{H}_{F, lq}}}
\end{align*}
which is precisely the corresponding connection matrix entry of
$\nabla_{H_E/U} \otimes \nabla_{H_F/U}$ by
\cref{rmk:tensor-conn-matrix}.

By the same remark, the connection matrix entry of
$\theta_E \otimes \theta_F$ corresponding to $\epsilon_i \otimes e_p \otimes \phi_j \otimes f_q$ is:
\[
\delta_{jq}\theta_{E, ip} + \delta_{ip}\theta_{F, jq}
\]
Similarly, the connection matrix entry of $\ol{\theta_E \otimes \theta_F}_{H_E \otimes H_F}$ is:
\begin{align*}
 & \sum_{k, l, m, n} (H_E \otimes H_F)^{-1}_{ijkl}
               \ol{(\theta_E \otimes \theta_F)}_{klmn}
               (H_E \otimes H_F)_{mnpq} \\
=& \sum_{k, l, m, n} (H_{E, ik}^{-1} H_{F, jl}^{-1})
               (\delta_{ln}\bar{\theta}_{E, km} + \delta_{km}\bar{\theta}_{F, ln})
               (H_{E, mp}H_{F, nq}) \\
=& \br{\sum_{l, n} \delta_{ln}H_{F, jl}^{-1}H_{F, nq}
       \sum_{k, m} H_{E, ik}^{-1}\bar{\theta}_{E, km}H_{E, mp}}
     + \br{\sum_{k, m} \delta_{km}H_{E, ik}^{-1}H_{E, mp}
           \sum_{l, n} H_{F, jl}^{-1}\bar{\theta}_{F, ln}H_{F, nq}} \\
=& \br{\sum_{l} H_{F, jl}^{-1}H_{F, lq}
       \sum_{k, m} H_{E, ik}^{-1}\bar{\theta}_{E, km}H_{E, mp}}
     + \br{\sum_{k}H_{E, ik}^{-1}H_{E, kp}
           \sum_{l, n} H_{F, jl}^{-1}\bar{\theta}_{F, ln}H_{F, nq}} \\
=& \delta_{jq} \sum_{k, m} H_{E, ik}^{-1}\bar{\theta}_{E, km}H_{E, mp}
     + \delta_{ip} \sum_{l, n} H_{F, jl}^{-1}\bar{\theta}_{F, ln}H_{F, nq}
\end{align*}
which is precisely the corresponding connection matrix entry of
$\bar{\theta}_{E, H_E} \otimes \bar{\theta}_{F, H_F}$.
By \cref{prop:tensor-conn-sum} and \cref{prop:tensor-conn-ext}, we have:
\begin{align*}
 & D''_E \otimes D''_F \\
=& (\iota^{0, 1} \circ \oprt_{E/U} + \iota^{1, 0} \circ \theta_E)
   \otimes (\iota^{0, 1} \circ \oprt_{F/U} + \iota^{1, 0} \circ \theta_F) \\
=& (\iota^{0, 1} \circ \oprt_{E/U}) \otimes (\iota^{0, 1} \circ \oprt_{F/U})
   + (\iota^{1, 0} \circ \theta_E) \otimes (\iota^{1, 0} \circ \theta_F) \\
=& \iota^{0, 1} \circ (\oprt_{E/U} \otimes \oprt_{F/U})
   + \iota^{1, 0} \circ (\theta_E \otimes \theta_F)
\end{align*}
Similarly:
\[
D'_E \otimes D'_F
= \iota^{1, 0} \circ (\prt_{H_E} \otimes \prt_{H_F})
  + \iota^{0, 1} \circ (\bar{\theta}_{E, H_E} \otimes \bar{\theta}_{F, H_F})
= \iota^{1, 0} \circ (\prt_{H_E} \otimes \prt_{H_F})
  + \iota^{0, 1} \circ (\ol{\theta_{E} \otimes \theta_{F}}_{H_E \otimes H_F})
\]

Combining these observations with the
harmonicity of $\sH_E, \sH_F$, \cref{prop:tensor-conn-sum} and \cref{prop:tensor-conn-ext}, we see
that:
\begin{align*}
 & \nabla_{H_E \otimes H_F/U} + \iota^{1, 0} \circ (\theta_E \otimes \theta_F)
   + \iota^{0, 1} \circ (\ol{\theta_E \otimes \theta_F}_{H_E \otimes H_F}) \\
=& \nabla_{H_E/U} \otimes \nabla_{H_F/U}
   + (\iota^{1, 0} \circ \theta_E) \otimes (\iota^{1, 0} \circ \theta_F)
   + (\iota^{0, 1} \circ \bar{\theta}_{E, H_E})
     \otimes (\iota^{0, 1} \circ \bar{\theta}_{E, H_E}) \\
=& (\nabla_{H_E/U} + \iota^{1, 0} \circ \theta_E + \iota^{0, 1} \circ \bar{\theta}_{E, H_E}) \otimes
   (\nabla_{H_F/U} + \iota^{1, 0} \circ \theta_F + \iota^{0, 1} \circ \bar{\theta}_{F, H_F}) \\
=& D_E \otimes D_F
\end{align*}
This shows the harmonicity of $\sH_E \otimes \sH_F$.
Finally, by \cref{prop:harmonic-op-conn}, we have $D_E = D'_E + D''_F, D_F = D'_F + D''_F$, and by
\cref{prop:tensor-conn-sum}, we have:
\[
D_E \otimes D_F = D'_E \otimes D'_F + D''_E \otimes D''_F
\implies D'_E \otimes D'_F = D_E \otimes D_F - D''_E \otimes D''_F
\]
showing that $D'_E \otimes D'_F$ is the harmonic operator by \cref{prop:harmonic-op-conn} again.
\end{proof}

\begin{prop}\label{prop:harmonic-bun-dual}
Let $\sH_E$ be a smooth family of harmonic bundles on $X$ parametrized by $U$. Then,
$(E^\vee, (D'')^\vee, D^\vee)$ is also such a family with harmonic operator $(D'_E)^\vee$.
\end{prop}
\begin{proof}
We will show that the dual metric $H^\vee$ on $E^\vee$ is a harmonic metric relating
$(D'')^\vee$ and $D^\vee$.
Let $H$ be a harmonic metric on $E$ relating $D''$ and $D$.
Consider the matrix of $H$ in a local frame of $E$, and denote this matrix also by $H$.
Then, in the dual frame, the matrix of $H^\vee$ is $(H^{-1})^T$ and the connection matrix entries of
$\nabla_{H^\vee/U}$ are:
\[
\sum_k \br{\ol{(H^{-1})^T}}^{-1}_{ik}\prt_{X/U}\br{\br{\ol{(H^{-1})^T}}_{kj}} \\
= \sum_k H_{ik}\prt_{X/U}\br{H^{-1}_{kj}}
\]
Here, we are using the facts that taking conjugates and transposes commute with taking inverses, and
that $H$ is a Hermitian matrix over each point.
Now, notice that:
\begin{align*}
        & \prt_{X/U}\br{\sum_k H_{ik}H^{-1}_{kj}} = \prt_{X/U}(\delta_{ij}) \\
\implies& \sum_{k} \br{\prt_{X/U}(H_{ik})H^{-1}_{kj} + H_{ik}\prt_{X/U}(H^{-1}_{kj})} = 0 \\
\implies& \sum_{k} H_{ik}\prt_{X/U}(H^{-1}_{kj}) = -\sum_k \prt_{X/U}(H_{ik})H^{-1}_{kj}
\end{align*}
By \cref{rmk:dual-conn-matrix}, the connection matrix entries of $\nabla_{H/X}^\vee$ are:
\[
-\sum_k \bar{H}_{jk}^{-1} \prt_{X/U}(\bar{H}_{ki})
= -\sum_k H_{kj}^{-1} \prt_{X/U}(H_{ik})
= -\sum_k \prt_{X/U}(H_{ik})H_{kj}^{-1}
\]
This shows that $\nabla_{H^\vee/U} = \nabla_{H/U}^\vee$ by \cref{prop:conn-form-coll}.
It is straightforward to see that if $[\theta_{ij}]$ is the connection matrix of the Higgs field
$\theta$ of $D''$, then that of the Higgs field $\theta^\vee$ of $(D'')^\vee$ is $[-\theta_{ji}]$.
Then, the connection matrix entries of $\ol{\theta^\vee}_{H^\vee}$ are:
\begin{align*}
 & \sum_{k, l} (H^\vee)^{-1}_{ik} \ol{(-\theta_{lk})} H^\vee_{lj} \\
=& \sum_{k, l} (H^{-1, T})^{-1}_{ik} \ol{(-\theta_{lk})} (H^{-1, T})_{lj} \\
=& \sum_{k, l} H_{ki} \ol{(-\theta_{lk})} H^{-1}_{jl} \\
=& -\sum_{k, l}  H^{-1}_{jl} \ol{\theta_{lk}} H_{ki}
\end{align*}
These are precisely the entries of the connection matrix of $(\bar{\theta}_H)^\vee$.
Thus, $(\bar{\theta}_H)^\vee = \ol{\theta^\vee}_{H^\vee}$ by \cref{prop:conn-form-coll}.
Combining these observations with the harmonicity of $D$, \cref{prop:dual-conn-sum} and
\cref{prop:dual-conn-ext}, we get:
\begin{align*}
 & \nabla_{H^\vee/U}
   + \iota^{1, 0} \circ \theta^\vee
   + \iota^{0, 1} \circ \ol{\theta^\vee}_{H^\vee} \\
=& \nabla_{H/U}^\vee
   + (\iota^{1, 0} \circ \theta)^\vee
   + (\iota^{0, 1} \circ \ol{\theta}_{H})^\vee \\
=& \br{\nabla_{H/U}
   + \iota^{1, 0} \circ \theta
   + \iota^{0, 1} \circ \ol{\theta}_{H}}^\vee \\
=& D_E^\vee
\end{align*}
Similarly,
\[
(D''_E)^\vee = (\iota^{0, 1} \circ \oprt_{E/U} + \iota^{1, 0} \circ \theta)^\vee
= \iota^{0, 1} \circ \oprt_{E/U}^\vee + \iota^{1, 0} \circ \theta^\vee
\]
Finally, by \cref{prop:harmonic-op-conn}, we have $D_E = D'_E + D''_E$, and by
\cref{prop:tensor-conn-sum}, we have:
\[
D_E^\vee = (D'_E + D''_E)^\vee = (D'_E)^\vee + (D''_E)^\vee
\implies (D'_E)^\vee = D^\vee - (D''_E)^\vee
\]
so that $(D'_E)^\vee$ is the harmonic operator by \cref{prop:harmonic-op-conn} again.
\end{proof}

\subsection{$\dg$--Prestack of Harmonic Bundles}
\label{subsec:dg-prest-harmonic}

\begin{defn}[Hom Harmonic Bundle Family]\label{defn:harmonic-bun-family-hom}
Let $\sH_E$ and $\sH_F$ be two smooth families of harmonic bundles on a complex
manifold $X$ parametrized by a smooth manifold $U$.
By \cref{prop:harmonic-bun-tensor} and \cref{prop:harmonic-bun-dual}, the following triple:
\[
(E^\vee \otimes F, [D''_E, D''_F], [D_E, D_F])
= (E^\vee \otimes F, (D''_E)^\vee \otimes D''_F, D_E^\vee \otimes D_F)
\]
is a smooth family of harmonic bundles on $X$ parametrized by $U$.
We will denote this as $[\sH_E, \sH_F]$ and call it the Hom harmonic bundle from $\sH_E$ to
$\sH_F$. We will denote by $D'_{[E, F]}$ the harmonic operator of this family of harmonic bundles.
We will denote the following complex of sheaves of $\bF$--modules provided by
\cref{prop:harmonic-op-Higgs-op}:
\[\begin{tikzcd}[column sep = large]
0 \ar[r] &
\ker{(D'_{[E, F]})^0} \ar[r, "{[D''_E, D''_F]^0}"] &
\ker{(D'_{[E, F]})^1} \ar[r, "{[D''_E, D''_F]^1}"] &
\ker{(D'_{[E, F]})^2} \ar[r, "{[D''_E, D''_F]^2}"] &
\cdots
\end{tikzcd}\]
by $\dbr{\sH_E, \sH_F}$.
\end{defn}

\begin{prop}\label{prop:hom-harmonic-op}
In the context of \cref{defn:harmonic-bun-family-hom},
$D'_{[E, F]} = [D'_E, D'_F] = (D'_E)^\vee \otimes D'_F$ as a $1$--$\prt_{X/U}^\bullet$--connection.
\end{prop}
\begin{proof}
Apply \cref{prop:harmonic-bun-tensor} and \cref{prop:harmonic-bun-dual}.
\end{proof}

\begin{thm}\label{thm:harmonic-bun-family-cat-enriched-in-sheaves}
Taking $\dbr{\sH_E, \sH_F}$ as defined in \cref{defn:harmonic-bun-family-hom} as the $\Hom$ objects
makes the collection of smooth families of harmonic bundles on $X$ parametrized by $U$ a category
enriched in the category of complexes of sheaves of $\bC$--modules over $U \times X$ with
composition and identity morphisms obtained by restricting the ones from
\cref{thm:conn-cat-enriched-in-sheaves} applied to flat $1$--$\oprt_{X/U}^\bullet$--connections.
Taking global sections of the sheaf Hom complexes yields a differential graded category structure on
the collection of smooth families of harmonic bundles on $X$ parametrized by $U$.
\end{thm}
\begin{proof}
It suffices to show that the composition and identity maps of
\cref{thm:conn-cat-enriched-in-sheaves} considered for $\Conn_{1, \oprt_{X/U}}(U \times X)$
restrict to the $\ker{(D'_{[E, F]})^\bullet} \subset E^\vee \otimes F \otimes K^{\wedge \bullet}$.
In the proof of \cref{thm:conn-cat-enriched-in-sheaves},
we showed that for any three connections $\nabla_E, \nabla_F, \nabla_G$ with respect to any
differential, the following identity holds for elementary tensors:
\[
c\br{[\nabla_E, \nabla_F](s) \otimes_{\ul\bF} t
+ (-1)^{\deg s} s \otimes_{\ul\bF} [\nabla_E, \nabla_F](t)}
= [\nabla_E, \nabla_G](c(s \otimes_{\ul\bF} t))
\]
Taking $\nabla_P := D'_P, P = E, F, G$ implies that if $s \in \ker{[D'_F, D'_G]}$
and $t \in \ker{[D'_E, D'_F]}$, then $c(s \otimes_{\ul\bF} t) \in \ker{[D'_E, D'_G]}$.
Since every tensor is locally a sum of elementary tensors, we see that the composition morphisms
of $\Conn_{1, \oprt_{X/U}}(U \times X)$ restrict to composition morphisms for the sheaf Hom
complexes of harmonic bundles.

Next, we must show that the identity maps $e : \ul\bC[0] \to \dbr{D''_E, D''_E}$ factor through
$\ker{(D'_{[E, E]})^\bullet}$. Again, in the proof of \cref{thm:conn-cat-enriched-in-sheaves},
we showed that for any connections $\nabla_E$ with respect to any differential,
$[\nabla_E, \nabla_E](\id_E) = 0$. Thus, taking $\nabla_E = D'_E$ yields the first statement.

Since the global sections functor, which is the same as the derived global sections functor by
\cref{rmk:RConn-is-Conn-smooth-man}, is lax monoidal, the differential graded category structure
follows from change of enrichment just as in \cref{thm:dg-cat-conn}(i).
\end{proof}

\begin{defn}[$\dg$--Category of Smooth Families of Harmonic Bundle]
We will denote the differential graded category of
\cref{thm:harmonic-bun-family-cat-enriched-in-sheaves} as ${\sH}(U \times X)$.
\end{defn}

\begin{prop}\label{prop:harmonic-bun-family-pullback-harmonic-op}
In the context of \cref{prop:harmonic-bun-family-pullback}, the harmonic operator of
$(f^<E, f^<D'', f^<D)$ is $f^<D'$.
\end{prop}
\begin{proof}
We observe that $f^<D = f^<(D'' + D') = f^<D'' + f^<D'$, where the last step follows from comparing
connection matrices. By \cref{prop:harmonic-op-conn}, $f^<D'$ is the harmonic operator.
\end{proof}

\begin{rmk}\label{rmk:harmonic-bun-hom-pb-nat}
For two smooth families of harmonic bundles $\sH_E, \sH_F$ on $X$ parametrized by $U$, and a smooth
map $f : V \to U$, consider the following diagram:
\[\begin{small}\begin{tikzcd}[column sep=tiny]
{f^<(E^\vee \otimes F) \otimes \sA_{X/V}^k} &&
{f^<(E^\vee \otimes F) \otimes \sA_{X/V}^{k + 1}} & \\ &
{f^<(E^\vee \otimes F) \otimes \sA_{X/V}^{k + 1}} &&
{f^<(E^\vee \otimes F) \otimes \sA_{X/V}^{k + 2}} \\
{f^<E^\vee \otimes f^<F \otimes \sA_{X/V}^{k}} &&
{f^<E^\vee \otimes f^<F \otimes \sA_{X/V}^{k + 1}} \\ &
{f^<E^\vee \otimes f^<F \otimes \sA_{X/V}^{k + 1}} &&
{f^<E^\vee \otimes f^<F \otimes \sA_{X/V}^{k + 2}}
    \arrow["{f^<D'}", from=1-1, to=1-3]
    \arrow["{f^<D''}"', from=1-1, to=2-2]
    \arrow["\cong"', from=1-1, to=3-1]
    \arrow["{f^<D''}", from=1-3, to=2-4]
    \arrow["\cong"{pos=0.7}, from=1-3, to=3-3]
    \arrow["{-f^<D'}"{pos=0.3}, from=2-2, to=2-4, crossing over]
    \arrow["\cong", from=2-4, to=4-4]
    \arrow["{D'}"{pos=0.8}, from=3-1, to=3-3]
    \arrow["{D''}"', from=3-1, to=4-2]
    \arrow["{D''}", from=3-3, to=4-4]
    \arrow["\cong"'{pos=0.7}, from=2-2, to=4-2, crossing over]
    \arrow["{-D'}"', from=4-2, to=4-4]
\end{tikzcd}\end{small}\]
where the vertical isomorphisms are the natural ones from \cref{prop:pullback-conn-dual} and
\cref{prop:pullback-conn-tensor} making the vertical faces commute. The top and bottom faces
commute by a combination \cref{prop:harmonic-op-conn}(iii) and \cref{prop:conn-scale}.
The back face implies that the vertical isomorphisms restrict to the kernels of the harmonic
operators, while the left and right faces imply that isomorphisms commute with the differentials.
That is, the isomorphisms of the form in the above diagram assemble to an isomorphism of
complexes of sheaves of $\bC$--modules:
\begin{align*}
          & \br{\ker{f^<((D'_E)^\vee \otimes D'_F)^\bullet},
                f^<((D''_E)^\vee \otimes D''_F)^\bullet} \\
\to[\cong]& \br{\ker{(f^<(D'_E)^\vee \otimes f^<D'_F)^\bullet},
                (f^<(D''_E)^\vee \otimes f^<D''_F)^\bullet}
\end{align*}
natural in $E, F$. The latter complex is simply $\dbr{f^<\sH_E, f^<\sH_F}$ by
\cref{defn:harmonic-bun-family-hom}. We will not distinguish between these complexes in
the remainder of the paper.
\end{rmk}

\begin{thm}\label{thm:harmonic-bun-pullback-func}
Given a complex manifold $X$, a smooth map $f : V \to U$, and
a smooth family of harmonic bundles $\sH_E$ on $X$ parametrized by $U$, the following hold:
\begin{enumerate}
\item The morphism of sheaves of $\bC$--vector spaces $\eta^\bullet$ from
\cref{prop:pullback-conn-higher-forms} restricts to a morphism
$\ker{(D'_E)^k} \to f_>\ker{(f^<D'_E)^k}$ making the following diagram commute:
\[\begin{tikzcd}
\ker{(D'_E)^k} \ar[r] \ar[d, hook] & f_>\ker{(f^<D'_E)^k} \ar[d, hook] \\
E \otimes \sA_{X/U}^{k} \ar[r, "\eta^k" below] & f_>(f^<E \otimes \sA_{X/V}^{k})
\end{tikzcd}\]
Denote this map also by $\eta^k$, dropping the superscript when there is no confusion.

\item In particular, the morphism of complexes of $\bC$--vector spaces provided by
\cref{prop:pullback-conn-higher-forms}(i) restricts to a morphism of $\bC$--vector spaces:
\[
\Gamma(U \times X, \ker{(D'_E)^\bullet}, (D''_E)^\bullet)
\to \Gamma(V \times X, \ker{(f^<D'_E)^\bullet}, (f^<D''_E)^\bullet)
\]

\item Taking \cref{prop:harmonic-bun-family-pullback-harmonic-op} into account, the mapping
$\sH_E \mapsto f^<\sH_E := (f^<E, f^<D''_E, f^<D_E)$ and the above morphisms assemble to a
$\bC\dg$--functor:
\[
f^< : \sH(U \times X) \to \sH(V \times X)
\]
\end{enumerate}
\end{thm}
\begin{proof}
(i)
By \cref{prop:pullback-conn-higher-forms}, we have a morphism of complexes of sheaves of
$\bC$--modules:
\[\begin{tikzcd}[column sep = large]
0 \ar[r] \ar[d, equal] &
E \ar[r, "(D'_E)^0"] \ar[d, "\eta^0"] &
E \otimes \sA_{X/U}^1 \ar[r, "(D'_E)^1"] \ar[d, "\eta^1"] &
E \otimes \sA_{X/U}^2 \ar[r, "(D'_E)^2"] \ar[d, "\eta^2"] & \cdots \\
0 \ar[r] & f_>f^<E \ar[r, "f_>(f^<D'_E)^0" below] &
f_>(f^<E \otimes \sA_{X/V}^1) \ar[r, "f_>(f^<D'_E)^1" below] &
f_>(f^<E \otimes \sA_{X/V}^{2}) \ar[r, "f_>(f^<D'_E)^2" below] & \cdots
\end{tikzcd}\]
and the first claim follows.

(ii)
Consider the following diagram:
\[\begin{tikzcd}[column sep = small]
\ker{(D'_E)^k} \ar[rr, "D''_E"] \ar[rd, hook] \ar[dd, "\eta" left] & &
\ker{(D'_E)^{k + 1}} \ar[rd, hook] \ar[dd, "\eta" near start] & \\ &
E \otimes \sA_{X/U}^k \ar[rr, crossing over, "D''_E" near start] & &
E \otimes \sA_{X/U}^{k + 1} \ar[dd, "\eta"] \\
f_>\ker{(f^<D'_E)^k} \ar[rr, "f_>f^<D''_E" near start] \ar[rd, hook] & &
f_>\ker{(f^<D'_E)^{k + 1}} \ar[rd, hook] & \\ &
f_>(f^<E \otimes \sA_{X/U}^k)
    \ar[from=uu, crossing over, "\eta" near end] \ar[rr, "f_>f^<D''_E" below] & &
f_>(f^<E \otimes \sA_{X/U}^{k + 1}) \\
\end{tikzcd}\]
The top and bottom faces of this diagram commute by \cref{prop:harmonic-op-Higgs-op}.
The left and right faces commute by (i). The front face commutes by
\cref{prop:pullback-conn-higher-forms}. The back face now commutes because the bottom right
arrow is a monomorphism.

(iii)
We first recall from \cref{defn:harmonic-bun-family-hom} that:
\[
\ker{(D'_{[E, F]})^\bullet} = \ker{((D'_E)^\vee \otimes D'_F)^\bullet} = \dbr{\sH_E, \sH_F}
\]
We then proceed as in the proof of \cref{prop:pullback-func}(i).
First observe that the following diagram commutes by the naturality of the lax monoidal structure
maps of $\Gamma$:
\begin{equation}\label{eqn:harmonic-bun-pb-fun-comp-1}
\begin{small}\begin{tikzcd}
\Gamma\dbr{\sH_{F}, \sH_{G}} \otimes \Gamma\dbr{\sH_{E}, \sH_{F}}
    \ar[r] \ar[d, "\Gamma(\eta) \otimes_{\ul\bF} \Gamma(\eta)"] &
\Gamma(\dbr{\sH_{F}, \sH_{G}} \otimes \dbr{\sH_{E}, \sH_{F}})
    \ar[d, "\Gamma(\eta \otimes_{\ul\bF} \eta)"] \\
\Gamma f_>\dbr{f^<\sH_{F}, f^<\sH_{G}} \otimes_{\ul\bF} \Gamma f_> \dbr{f^<\sH_{E}, f^<\sH_{F}}
    \ar[r] &
\Gamma f_> \br{\dbr{f^<\sH_{F}, f^<\sH_{G}} \otimes_{\ul\bF} \dbr{f^<\sH_{E}, f^<\sH_{F}}}
\end{tikzcd}\end{small}\end{equation}
Consider the following diagram:
\begin{equation}\label{eqn:harmonic-bun-pb-fun-comp-2}\begin{small}
\begin{tikzcd}[column sep=-3.5em]
\dbr{\sH_{F}, \sH_{G}} \otimes_{\ul\bF} \dbr{\sH_{E}, \sH_{F}}
    \ar[rr, "c"] \ar[rd, hook] \ar[dd, "\eta \otimes_{\ul\bF} \eta" left] & &
\dbr{\sH_{E}, \sH_{G}}
    \ar[rd, hook] \ar[dddd, "\eta"] & \\ &
\dbr{D''_F, D''_G} \otimes_{\ul\bF} \dbr{D''_E, D''_F}
    \ar[rr, crossing over, "c" near start] \ar[dd, "\eta \otimes_{\ul\bF} \eta" left] & &
\dbr{D''_E, D''_G}
    \ar[dddd, "\eta"] \\
f_>\dbr{f^<\sH_{F}, f^<\sH_{G}} \otimes_{\ul\bF} f_>\dbr{f^<\sH_{E}, f^<\sH_{F}}
    \ar[rd, hook] \ar[dd] & & & \\ &
f_>\dbr{f^<D''_F, f^<D''_G} \otimes_{\ul\bF} f_>\dbr{f^<D''_E, D''_F} & & \\
f_>\br{\dbr{f^<\sH_{F}, f^<\sH_{G}} \otimes_{\ul\bF} \dbr{f^<\sH_{E}, f^<\sH_{F}}}
    \ar[rr, "f_>c" near end] \ar[rd, hook] & &
f_>\dbr{f^<\sH_{E}, f^<\sH_{G}}
    \ar[rd, hook] & \\
& f_>(\dbr{f^<D''_F, f^<D''_G} \otimes_{\ul\bF} \dbr{f^<D''_E, D''_F})
    \ar[from=uu, crossing over] \ar[rr, "f_>c"] & &
f_>\dbr{f^<D''_E, f^<D''_G}
\end{tikzcd}\end{small}\end{equation}
where the $c$ are the composition maps.
The top and bottom faces of this diagram commute by
\cref{thm:harmonic-bun-family-cat-enriched-in-sheaves}.
The top left face and the right face commute by (i).
The bottom left face commutes by the naturality of the lax
monoidal structure maps of $f_> = (f \times \id_X)_*$. We showed the front face commutes in the
proof of \cref{prop:pullback-func}(i). We can now conclude that the back face commutes because the
bottom right arrow is a monomorphism. Now, applying $\Gamma$ to the back face of diagram
\ref{eqn:harmonic-bun-pb-fun-comp-2} and pasting with diagram \ref{eqn:harmonic-bun-pb-fun-comp-1}
along the $\Gamma(\eta \otimes_{\ul\bF} \eta)$ edge shows that the following diagram commutes:
\[\begin{tikzcd}
\Gamma\dbr{\sH_{F}, \sH_{G}} \otimes_{\ul\bF} \Gamma\dbr{\sH_{E}, \sH_{F}}
    \ar[r, "c"] \ar[d, "\eta \otimes_{\ul\bF} \eta" left] &
\Gamma\dbr{\sH_{E}, \sH_{G}}
    \ar[d, "\eta"] \\
\Gamma\dbr{f^<\sH_{F}, f^<\sH_{G}} \otimes_{\ul\bF} \Gamma\dbr{f^<\sH_{E}, f^<\sH_{F}}
    \ar[r, "c" below] &
\Gamma\dbr{f^<\sH_{E}, f^<\sH_{G}}
\end{tikzcd}\]

Now, consider the following diagram:
\[\begin{tikzcd} &
{\dbr{H_E, H_E}} & \\
{\ul\bC[0]} &&
{\dbr{D''_E, D''_E}} \\ &
{f_>\dbr{f^<H_E, f^<H_E}} \\
{f_>f^<\ul\bC[0]} &&
{f_>\dbr{f^<D''_E, f^<D''_E}}
    \arrow[hook, from=1-2, to=2-3]
    \arrow["\eta"'{pos=0.7}, from=1-2, to=3-2]
    \arrow["e", from=2-1, to=1-2]
    \arrow["e"{pos=0.7}, from=2-1, to=2-3, crossing over]
    \arrow[from=2-1, to=4-1]
    \arrow[from=2-3, to=4-3, "\eta"]
    \arrow[hook, from=3-2, to=4-3]
    \arrow["{f_>e}", from=4-1, to=3-2]
    \arrow["{f_>e}"', from=4-1, to=4-3]
\end{tikzcd}\]
The front face commutes just as in the proof of \cref{prop:pullback-func}, while the
back right face commutes by (i). The back left face then commutes because the bottom arrow
of the back right face is a monomorphism. Taking global sections then shows the commutativity of the
following diagram:
\[\begin{tikzcd} &
{\Gamma(\ul\bC[0])} &
{\Gamma\dbr{H_E, H_F}} \\
{\bC[0]} \\ &
{\Gamma(f^<\ul\bC[0])} & {\Gamma\dbr{f^<H_E, f^<H_F}}
    \arrow["e", from=1-2, to=1-3]
    \arrow[from=1-2, to=3-2]
    \arrow["\eta", from=1-3, to=3-3]
    \arrow[from=2-1, to=1-2]
    \arrow[from=2-1, to=3-2]
    \arrow["e"', from=3-2, to=3-3]
\end{tikzcd}\]
and we are done.
\end{proof}

\begin{thm}\label{thm:harmonic-bun-prest}
Given a complex manifold $X$, the mapping:
\[\begin{array}{ccccc}
\mcM_\sH(X) &:& \Mfd^\op &\to& \bC\dg-\Cat \\
           &:& U &\mapsto& \sH(U \times X) \\
           &:& (f : V \to U) &\mapsto& (f^< : \sH(U \times X) \to \sH(V \times X))
\end{array}\]
is a pseudofunctor.
\end{thm}
\begin{proof}
Let $W \to[f] V \to[g] U$ be a sequence of smooth maps.
Consider the natural isomorphism of finite rank locally free sheaves
$\alpha_E : (gf)^<E \to f^<g^<E$ for each $\sH_E \in \sH_{U \times X}$.
By \cref{prop:pullback-comp-nat}, this map is simultaneously a morphism of connections
$(gf)^<D_E \to f^<g^<D_E$, $(gf)^<D'_E \to f^<g^<D'_E$, $(gf)^<D''_E \to f^<g^<D''_E$.
Thus, it is in the kernels of $[(gf)^<D_E, f^<g^<D_E]$, $[(gf)^<D'_E, f^<g^<D'_E]$ and
$[(gf)^<D''_E, f^<g^<D''_E]$ simultaneously by \cref{rmk:hom-conn-desc}.
This implies that the map $\hat{\alpha}_E : \ul\bC \to ((gf)^<E)^\vee \otimes f^<g^<E$ corresponding
to $\alpha_E$ factors through the inclusion
$\ker{[(gf)^<D'_E, f^<g^<D'_E]} \hto ((gf)^<E)^\vee \otimes f^<g^<E$ and and is zero under
$[(gf)^<D''_E, f^<g^<D''_E]$. Thus, $\hat{\alpha}_E$ restricts to a morphism of complexes:
$\tilde{\alpha}_E : \ul\bC[0] \to \dbr{\sH_E, \sH_F}$. Then, consider the following diagram:
\begin{equation}\label{eqn:harmonic-bun-prest-1}\begin{small}
\begin{tikzcd}[column sep=-1em]
\ul\bC[0] \otimes_{\ul\bC} \dbr{\sH_{E}, \sH_{F}}
    \ar[rr, "\tilde{\alpha}_F \otimes_{\ul\bC} \eta"] \ar[rd, hook] & &
\begin{matrix}
(gf)_>\dbr{(gf)^<\sH_{F}, f^<g^<\sH_{F}} \\
\otimes_{\ul\bC} (gf)_>\dbr{(gf)^<\sH_{E}, (gf)^<\sH_{F}}
\end{matrix}
    \ar[rd, hook] \ar[dd, "c'" near end] & \\ &
\ul\bC[0] \otimes_{\ul\bC} \dbr{D''_{E}, D''_{F}}
    \ar[rr, crossing over, "\tilde{\alpha}_F \otimes_{\ul\bC} \eta" near start] & &
\begin{matrix}
(gf)_>\dbr{(gf)^<D''_{F}, f^<g^<D''_{F}} \\
\otimes_{\ul\bC} (gf)_>\dbr{(gf)^<D''_{E}, (gf)^<D''_{F}}
\end{matrix}
    \ar[dd, "c'"] \\
\dbr{\sH_E, \sH_F}
    \ar[uu, "\cong" left] \ar[rd, hook] \ar[dd, "\cong" left] & &
(gf)_>\dbr{(gf)^<\sH_E, f^<g^<\sH_F}
    \ar[rd, hook] & \\ &
\dbr{D''_E, D''_F}
    \ar[uu, "\cong" left] & &
(gf)_>\dbr{(gf)^<D''_E, f^<g^<D''_F} \\
\dbr{\sH_E, \sH_F} \otimes_{\ul\bC} \ul\bC[0]
    \ar[rr, "\eta \otimes_{\ul\bC} \tilde{\alpha}_E" near end, below] \ar[rd, hook] & &
\begin{matrix}
(gf)_>\dbr{f^<g^<\sH_E, f^<g^<\sH_E} \\
\otimes_{\ul\bC} (gf)_>\dbr{(gf)^<\sH_E, f^<g^<\sH_E}
\end{matrix}
    \ar[rd, hook] \ar[uu, "c'" right] & \\ &
\dbr{D''_E, D''_F} \otimes_{\ul\bC} \ul\bC[0]
    \ar[rr, "\eta \otimes_{\ul\bC} \tilde{\alpha}_E" below]
    \ar[from=uu, crossing over, "\cong" left, near start] & &
\begin{matrix}
(gf)_>\dbr{f^<g^<D''_E, f^<g^<D''_E} \\
\otimes_{\ul\bC} (gf)_>\dbr{(gf)^<D''_E, f^<g^<D''_E}
\end{matrix}
    \ar[uu, "c'" right]
\end{tikzcd}\end{small}\end{equation}
where the $\tilde{\alpha}_E, \tilde{\alpha}_F, c'$ are as in \cref{prop:pullback-comp-nat}, so that
the front face commutes. The top and bottom faces commute by the previous paragraph. The two left
faces commute by the naturality of the vertical isomorphisms --- recall that these are unitor
isomorphisms of the monoidal category of complexes of sheaves of $\bC$--modules. The two right faces
commute by the naturality of the lax monoidal structure maps of $(gf)_>$ along with the fact
that the composition maps of harmonic bundle families are obtained by restriction of the composition
maps of $\Conn_{1, D''}(U \times X)$ --- see \cref{thm:harmonic-bun-family-cat-enriched-in-sheaves}.
Now, the fact that the middle map on the right faces is a monomorphism implies that the back face
commutes. Applying the global sections functor $\Gamma$ along with its lax monoidal structure shows
that the $\tilde{\alpha}_E$ assemble to a $\bC\dg$--natural transformation:
\[\begin{tikzcd}[column sep = 0pt] &
\sH(V \times X)
    \ar[from=rd, "g^<" above right] & \\
\sH(W \times X) \ar[from=ru, "f^<" above left]
    \ar[from=rr, "(g \circ f)^<" below, ""{name=A, above}] & &
\sH(U \times X)
\ar[Rightarrow, from=A, to=1-2]
\end{tikzcd}\]

Next, by \cref{prop:pullback-unit-nat}, the natural isomorphisms $\beta_E : E \to \id_U^<E$ are also
morphisms of connections $D_E \to \id_U^<D_E$, $D'_E \to \id_U^<D'_E$ and $D''_E \to \id_U^<D''_E$.
Thus, $\beta_E$ is simultaneously in the kernels of $[D'_E, \id_U^<D'_E]$ and
$[D''_E, \id_U^<D''_E]$ by \cref{rmk:hom-conn-desc}. This implies that the map
$\ul\bC \to E^\vee \otimes \id_U^<D''_E$ corresponding to $\beta_E$ factors through
$\ker{[D'_E, \id_U^<D'_E]}$ and is also zero under $[D''_E, \id_U^<D''_E]$. This gives a morphism of
complexes of sheaves of $\ul\bC$--modules
$\tilde{\beta}_E : \ul\bC[0] \to \dbr{\sH_E, \id_U^<\sH_E}$.
Now, using \cref{prop:pullback-unit-nat}, arguments similar to the previous paragraph show that the
$\tilde{\beta}_E$ assemble to a $\bC\dg$--natural transformation:
\[\begin{tikzcd}[column sep=huge]
\sH(U \times X)
    \ar[from=r, bend right, "\id_U^<" above, ""{name=A, below}]
    \ar[from=r, bend left, "\id" below, ""{name=B, above}] &
\sH(U \times X)
\ar[Rightarrow, from=B, to=A]
\end{tikzcd}\]

It remains to verify the pseudofunctor coherence conditions for $\tilde{\alpha}$ and
$\tilde{\beta}$, but these are exactly the same as in \cref{thm:Conn-functor}.
\end{proof}

\begin{defn}[$\dg$--Prestack of Harmonic Bundles]
We will call the pseudofunctor $\mcM_\sH(X)$ of \cref{thm:harmonic-bun-prest}, the $\dg$--prestack of
harmonic bundles on $X$.
\end{defn}

\begin{rmk}\label{rmk:hbun-prest-Higgs-op-flat-conn-equiv}
By \cref{prop:harmonic-op-Higgs-op}, the results of this section hold with $D''$ replaced by
$D$.
\end{rmk}

\section{The Non-Abelian Hodge Correspondence}
\label{sec:NAH}

We now have all the foundational pieces necessary to produce the desired diffeological
correspondence between Higgs bundles and flat bundles. This will be achieved as follows.
We first establish forgetful morphisms of $\dg$--prestacks (pseudonatural tranformations):
\[
\mcM_{Dol}(X) \ot[\iota_{Dol}] \mcM_\sH(X) \to[\iota_{dR}] \mcM_{dR}(X)
\]
This is a generalization of the forgetful $\dg$--functors from harmonic bundles to Higgs bundles and
flat bundles in the original theory.
Interstingly, in contrast to the original setup, these forgetful maps, while quasi-faithful, are not
quasi-full --- see \cref{rmk:harmonic-bun-incl-not-quasi-full}. This phenomenon indicates the
presence of some obstruction to mediating between families of Higgs bundles and of flat connections
via families of harmonic bundles.
Nevertheless, the maps are surjective on the zeroth cohomologies of Hom complexes, and
this, in addition to injectivity on first cohomologies, is enough to ensure that the maps are fully
faithful after taking extension completions of $\dg$--categories described in
\cite[\S 3]{HiggsLocSys} and homotopy categories, which will be enough for our purposes.

Recall that $\dg$--prestacks are pseudofunctors $\Mfd^\op \to \bC\dg-\Cat$. Then, one can show
that the extension completion construction provides a pseudofunctor
$(-)^{ext} : \bC\dg-\Cat \to \bC\dg-\Cat$. Next, we have the pseudofunctor
$h : \bC\dg-\Cat \to \Cat$ given by sending a $\bC\dg$--category to its homotopy category. Finally,
we have the pseudofunctor $(-)^\simeq : \Cat \to \Grpd$ sending a category to its maximal
subgroupoid. In total, we have a composite:
$(-)^\simeq \circ h \circ (-)^{ext} : \bC\dg-\Cat \to \Cat$. We can whisker this composite with the
pseudonatural transformations $\iota_{Dol}, \iota_{dR}$, and take the Grothendieck construction
of the resulting diagram of pseudofunctors $\Mfd^\op \to \Cat$ to obtain a diagram of prestacks:
\[
\sM_{Dol}(X) \to \int h\mcM_{\sH}^{ext, \simeq}(X) \to \sM_{dR}(X)
\]

Using the last statement of the first paragraph, we then show that these morphisms of prestacks
are fully faithful so that they provide an equivalence of the essential images.
We provide the details of this argument in
\cref{subsec:harmonic-Dol-dR}.
In turn, we have an equivalence of the substacks $\sM^\sH_{Dol}(X)$ and $\sM^\sH_{dR}(X)$ generated
by the essential images. Furthermore, these generated substacks are diffeological substacks and
their fibres over the point contain contain semistable Higgs bundles with the usual condition on
Chern classes on the Dolbeault side and all flat bundles on the de Rham side. A point to note:
it is not clear exactly which families of semistable Higgs bundles or of flat bundles are captured
by the above stacks.
The stackification above means the families we obtain on either side are the ones that,
locally, are finite iterated extensions of families of harmonic bundles, but it might be that not all
families of semistable Higgs bundles or of flat bundles arise this way.
Importantly, however, there are some families
of semistable bundles that are not just families of stable or polystable bundles that are captured
by the above stacks --- see \cref{exm:sst-family}, which is the counterexample of
\cite[39]{ModRepFunGrpII}. In particular, the correspondence we have
produced applies for a family to a point of which the usual correspondence of coarse moduli spaces
does not extend continuously.
We describe the correspondence and what we know it captures in \cref{subsec:correspondence}.
We conclude with some questions for further study in \cref{subsec:questions}, including
the prospect of studying moduli stacks of $d$ times differentiable families of Higgs bundles and
flat connections for any $d \in \bN$.

\subsection{Harmonic Bundles to Higgs Bundles and Flat Connections}
\label{subsec:harmonic-Dol-dR}

\begin{prop}\label{prop:harmonic-bun-family-cohomology}
Let $(E, D'', D)$ be a smooth family of harmonic bundles on a complex manifold $X$,
parametrized by a smooth manifold $U$.
Then, the morphisms of complexes of sheaves of $\bC$--modules supplied by
\cref{prop:harmonic-op-Higgs-op}:
\begin{align*}
(\ker{(D')^\bullet}, (D'')^\bullet) &\to (E \otimes \sA_{X/U}^\bullet, (D'')^\bullet) \\
(\ker{(D')^\bullet}, (D'')^\bullet) &\to (E \otimes \sA_{X/U}^\bullet, D^\bullet)
\end{align*}
induce monomorphisms in cohomology. When $X$ is K\"haler as well, they induce isomorphisms in
the zeroth cohomology.
\end{prop}
\begin{proof}
Let $a \in \Gamma(\ker{(D')^{k + 1}} \cap \ker{(D'')^{k + 1}})$ such that $a = D''(a_0)$ for some
$a_0 \in \Gamma(\ker{(D')^k})$. Then, $D(a_0) = D'(a_0) + D''(a_0) = D''(a_0) = a$. This shows that
the maps are injective on cohomology groups of the complexes of global sections.

For surjectivity in degree $0$ when $X$ is K\"ahler,
first suppose $a'' \in \Gamma(\ker{(D'')^{0}})$.
We need to show that $a''$ is also in $\Gamma(\ker{(D')^0})$. By \cref{prop:conn-slicewise},
for each $u \in U$, $D''_u(a''_u) = D''(a'')_u = 0$. By \cref{rmk:harmonic-bun-family-slice},
$(E_u, D''_u, D_u)$ is a harmonic bundle on the compact K\"ahler manifold
$X \cong \set{u} \times X$  for each $u \in U$, and hence \cite[Lemma 2.2]{HiggsLocSys} applies,
showing that $D'(a'')_u = D'_u(a''_u) = 0$, by \cref{prop:conn-slicewise} again.
Since this holds for each $u \in U$, we have $D'(a) = 0$.
The argument for the second morphism of complexes is similar.
\end{proof}

\begin{rmk}\label{rmk:harmonic-bun-fmaily-cohomology-not-surjective}
Consider the elliptic curve $X = \bC/(\bZ + i\bZ)$, along with $U = \bR$ and
$E = U \times X \times \bC \to U \times X$,
\[
D'' = \oprt_{X/U} + u dz, D' = \prt_{X/U} + u d\bar{z}, D = D'' + D'
\]
We can immediately see that the constant metric is a harmonic metric relating $D''$ and $D$.
Observe that $dz \in \ker{(D'')^1}$ since $\oprt_{X/U}(dz) + u dz \wedge dz = 0$.
In order for the map of complexes
$\Gamma(\ker{D'}, D'') \to \Gamma(E \otimes \sA_{X/U}^\bullet, D'')
= \Gamma(\sA_{X/U}^\bullet, D'')$
to be surjective in first cohomology, we require a smooth function $s : U \times X \to \bC$ such
that $a := dz - D''(s) = -\oprt_{X/U}(s) + (1 - us) dz$ is in
$\Gamma(\ker{(D')^1}) \cap \Gamma(\ker{(D'')^1})$.
However, since $a$ is in $\Gamma(\sA_{X/U}^1)$, we must have $a = f dz + g d\bar{z}$ for some smooth
functions $f, g : U \times X \to \bC$.
Then, we have:
\begin{align*}
D''(a)
=& D''(dz) \\
=& \oprt_{X/U}(dz) + u dz \wedge dz \\
=& 0
\end{align*}
and
\begin{align*}
D'(a)
=& \prt_{X/U}(a) + u d\bar{z} \wedge a \\
=& -\prt_{X/U}\oprt_{X/U}(s) + \prt_{X/U}((1 - us)dz) - u d\bar{z} \wedge \oprt_{X/U}(s)
   + u d\bar{z} \wedge (1 - us)dz \\
=& -\prt_{X/U}\oprt_{X/U}(s) + u(1 - us) d\bar{z} \wedge dz \\
=& \br{\frac{\prt}{\prt z} \frac{\prt}{\prt \bar{z}}s + u(us - 1)} dz \wedge d\bar{z}
\end{align*}
where $x, y$ are the real coordinates satisfying $z = x + iy, \bar{z} = x - iy$.
It can be shown that $D'(a)$ being zero implies $s = 1/u$ for
$u \in \bR \setminus \set{0}$, but this admits no smooth extension to all of $U = \bR$.
Thus, the map cannot be surjective on first cohomology.
\end{rmk}

\begin{prop}\label{prop:harmonic-Dol-dR-functor}
Let $X$ be a complex manifold and $U$, a smooth manifold.
Then, there exist quasi-faithful $\bC\dg$--functors:
\[
\Conn_{0, \oprt_{X/U}^\bullet}(U \times X) \ot \sH(U \times X)
\to \Conn_{1, d_{X/U}}(U \times X)
\]
with the left map defined by $\sH_E \mapsto (E, D''_E)$ and the right map defined by
$\sH_E \mapsto (E, D_E)$. If $X$ is K\"ahler, the functors induce isomorphisms of the zeroth
cohomologies of morphism complexes and monomorphisms in higher degrees. Furthermore, for each
smooth map $f : V \to U$, the following diagram of $\bC\dg$--categories strictly commutes:
\[\begin{tikzcd}
\Conn_{0, \oprt_{X/U}^\bullet}(U \times X)
    \ar[d, "f^<" left] &
\sH(U \times X)
    \ar[l] \ar[r] \ar[d, "f^<"] &
\Conn_{1, d_{X/U}}(U \times X)
    \ar[d, "f^<"] \\
\Conn_{0, \oprt_{X/V}^\bullet}(V \times X) &
\sH(V \times X)
    \ar[r] \ar[l] &
\Conn_{1, d_{X/V}}(V \times X)
\end{tikzcd}\]
and the maps assemble to pseudonatural transformations:
\[
\mcM_{Dol}(X) \ot \mcM_{\sH}(X) \to \mcM_{dR}(X)
\]
\end{prop}
\begin{proof}
The mapping of object sets is given in the statement. We need morphisms of complexes
$\dbr{D''_E, D''_F} \ot \dbr{\sH_E, \sH_F} \to \dbr{D_E, D_F}$ and these are provided
by \cref{prop:harmonic-op-Higgs-op}. The fact that these morphisms assemble to a $\bC\dg$--functor
$\Conn_{0, \oprt_{X/U}^\bullet}(U \times X) \ot \sH(U \times X)$ is the content of
\cref{thm:harmonic-bun-family-cat-enriched-in-sheaves} combined with
\cref{rmk:hbun-prest-Higgs-op-flat-conn-equiv}.
The fact that we also get a $\bC\dg$--functor $\sH(U \times X) \to \Conn_{1, d_{X/U}}(U \times X)$
follows from the observation that $(D'')^k$ and $D^k$ are equal when restricted to
$\ker{(D')^k}$. The conclusions about quasi-faithfulness as well as the
zeroth cohomologies, when $X$ is K\"ahler, follow from \cref{prop:harmonic-bun-family-cohomology}.

The commutativity of the naturality diagram at the level of objects is immediate. At the level of
$\Hom$ complexes, the commutativity of the left square is a consequence of
\cref{thm:harmonic-bun-pullback-func}(ii) combined with \cref{prop:hom-harmonic-op}.
The commutativity of the right square at the level of $\Hom$ complexes is a consquence of the fact
that $D''$ restricted to $\ker{D'}$ is equal to $D$. We must now verify the coherence conditions
needed for a collection of morphisms to form a pseudonatural transformation.
Write $\mcM_{Dol}(X), \mcM_{dR}(X), \mcM_\sH(X)$ as simply $\mcM_{Dol}, \mcM_{dR}, \mcM_\sH$ for brevity.
Since the target
$2$--category $\bC\dg-\Cat$ has strict composition and identities, these conditions reduce to the
following equalities of horizontal composites in $\bC\dg-\Cat$:
\[\begin{tikzcd}
\mcM_\sH(V) \ar[r] &
\mcM_{Dol}(V)
    \ar[r, bend left, "\id", ""{below, name=A}]
    \ar[r, bend right, "\id_V^<" below, ""{above, name=B}] &
\mcM_{Dol}(V)
    \ar[Rightarrow, to=A, from=B, "\tilde\beta"]
\end{tikzcd} = \begin{tikzcd}
\mcM_\sH(V)
    \ar[r, bend left, "\id", ""{below, name=A}]
    \ar[r, bend right, "\id_V^<" below, ""{above, name=B}] &
\mcM_{\sH}(V) \ar[r] &
\mcM_{Dol}(V)
    \ar[Rightarrow, to=A, from=B, "\tilde\beta"]
\end{tikzcd}\]
\[\begin{tikzcd}[column sep=tiny] & &
\mcM_{Dol}(V) \ar[rd, "f^<" above right] & \\
\mcM_\sH(W) \ar[r] &
\mcM_{Dol}(W)
    \ar[rr, ""{above, name=B}, "(g \circ f)^<" below]
    \ar[ru, "g^<" above left] & &
\mcM_{Dol}(U)
    \ar[Rightarrow, to=1-3, from=B, "\tilde\alpha"]
\end{tikzcd} = \begin{tikzcd}[column sep=tiny] &
\mcM_{\sH}(V)
    \ar[rd, "g^<" above right] & & \\
\mcM_\sH(W)
    \ar[rr, ""{above, name=B}, "(g \circ f)^<" below]
    \ar[ru, "f^<" above left] & &
\mcM_{\sH}(U) \ar[r] &
\mcM_{Dol}(U)
    \ar[Rightarrow, to=1-2, from=B, "\tilde\alpha"]
\end{tikzcd}\]
and analogous equalities with $\mcM_{Dol}$ replaced by $\mcM_{dR}$. These equalities follow from the
observation that the maps $\tilde\beta_E : \ul\bC[0] \to \dbr{D''_E, D''_E}$ and
$\tilde\alpha_E : \ul\bC[0] \to (gf)_>\dbr{(gf)^<D''_E, f^<g^<D''_E}$ factor through
$\dbr{\sH_E, \sH_E}$ and $(gf)_>\dbr{(gf)^<\sH_E, f^<g^<\sH_E}$ respectively, as we saw in the proof
of \cref{thm:harmonic-bun-prest}. The case of $\mcM_{dR}$ is similar as $D$ restricted to $\ker{D'}$
is equal to $D''$.
\end{proof}

\begin{rmk}\label{rmk:harmonic-bun-incl-not-quasi-full}
Identifying $E$ with its endomorphism bundle in
\cref{rmk:harmonic-bun-fmaily-cohomology-not-surjective} shows that the $\bC\dg$--functors
of \cref{prop:harmonic-Dol-dR-functor} are not quasi-full. This indicates that there is some
potential obstruction in interpolating between smooth families of Higgs bundles and those of flat
bundles via smooth families of harmonic bundles.
\end{rmk}

\begin{prop}\label{prop:ext-compl-cohomology}
Let $F : \sA \to \sB$ be a $\bC\dg$--functor such that $\sA, \sB$ have morphism complexes
concentrated in non-negative degrees and the morphisms of cohomologies of the morphisms
complexes are bijective in degree $0$ and injective in degree $1$. Then, the morphisms of
cohomologies of morphism complexes induced by $F^{ext}$ are bijective in degree $0$.
\end{prop}
\begin{proof}
Let $A, B$ be two objects of $\sA^{ext}$. By definition of $\sA^{ext}$, $B$ is obtained
by finitely many successive extensions of objects of $\sA$ in $\sA^{ext}$. We proceed
by induction on the number of extensions needed to obtain $B$. If this number is zero, then
the hypothesis guarantees the conclusion. Suppose this number is greater than zero so that we have
an extension $U \to[a] B \to[b] V$ in $\sA^{ext}$. This gives a diagram of long exact sequences
of cohomologies of $\Hom$ complexes:
\[\begin{tikzcd}[column sep=small]
0 \ar[r] \ar[d, equal] & \Ext^0(A, U) \ar[r] \ar[d, "F^0_{A, U}"] &
    \Ext^0(A, B) \ar[r] \ar[d, "F^0_{A, B}"] &
    \Ext^0(A, V) \ar[r] \ar[d, "F^0_{A, V}"] & \Ext^1(A, U) \ar[r] \ar[d, "F^1_{A, U}"]
    & \cdots \\
0 \ar[r] & \Ext^0(FA, FU) \ar[r] & \Ext^0(FA, FB) \ar[r] & \Ext^0(FA, FV) \ar[r] &
    \Ext^1(FA, FU) \ar[r]
    & \cdots
\end{tikzcd}\]
To see this, we use the $\bF\dg$--functoriality of $F$ along with the fact that extensions are taken
to short exact sequences by $\Hom_\sA(A, -), \Hom_\sB(FA, -)$. The latter of these two facts can be
seen, in turn, by using the morphisms $a$, $b$, a choice of splitting, and the relations defining
an extension \cite[\S 3]{HiggsLocSys}.
Next, the inductive hypothesis guarantees that the
morphisms $F^0_{A, U}, F^0_{A, V}$ are bijections and the morphism $F^1_{A, U}$ is an injection.
By the five lemma, the morphism $F^0_{A, B}$ is a bijection, as required.
\end{proof}

\begin{thm}\label{thm:harmonic-Dol-dR-prest}
In the context of \cref{prop:harmonic-Dol-dR-functor}, consider the pseudonatural transformations
obtained by whiskering (horizontal composition in the $3$--category of $2$--categories) as follows:
\[\begin{tikzcd}[column sep=huge]
\Mfd^\op
    \ar[r, bend left=5.5em, shift left, "\mcM_{Dol}(X)" above, ""{below, name=Dol}]
    \ar[r, "\mcM_{\sH}(X)"{description, name=H}]
    \ar[r, bend right=5.5em, shift right, "\mcM_{dR}(X)" below, ""{above, name=dR}]
    \ar[Rightarrow, from=H, to=Dol]
    \ar[Rightarrow, from=H, to=dR] &
\bC\dg-\Cat \ar[r, "h((-)^{ext})^\simeq"] &
\Grpd
\end{tikzcd}\]
Writing $h\mcM^{ext, \simeq}_{P}(X) = (-)^{\simeq} \circ h \circ (-)^{ext} \circ \mcM_P(X)$ for
$P = \sH, Dol, dR$, we obtain a diagram of pseudofunctors:
\begin{equation}\label{eqn:harmonic-Dol-dR-psfunc}
h\mcM^{\simeq}_{Dol}(X) \ot[\iota_{Dol}]
h\mcM^{ext, \simeq}_\sH(X) \to[\iota_{dR}]
h\mcM^{\simeq}_{dR}(X)
\end{equation}
where, for object $U \in \Mfd$, the two component functors $\iota_{Dol, U}, \iota_{dR, U}$ are fully
faithful. In particular, we obtain fully faithful morphisms of prestacks:
\begin{equation}\label{eqn:harmonic-Dol-dR-prest}
\sM_{Dol}(X) \ot \int h\mcM^{ext, \simeq}_{\sH}(X) \to \sM_{dR}(X)
\end{equation}
\end{thm}
\begin{proof}
First recall from \cref{rmk:ext-compl-equiv-prest} that
$h\mcM^{ext, \simeq}_{Dol} \simeq h\mcM^{\simeq}_{Dol}(X),
h\mcM^{ext, \simeq}_{dR} = h\mcM^{\simeq}_{dR}(X)$ so that the whiskerings indeed provide
pseudonatural transformations of the form shown in diagram \ref{eqn:harmonic-Dol-dR-psfunc}.
Next, apply \cref{prop:harmonic-Dol-dR-functor} and \cref{prop:ext-compl-cohomology} to
deduce that the component functors of diagram \ref{eqn:harmonic-Dol-dR-psfunc} are fully faithful.
By the Grothendieck construction, we obtain the morphisms of prestacks shown in diagram
\ref{eqn:harmonic-Dol-dR-prest} and these are fibrewise fully faithful by the previous statement,
and hence fully faithful functors of the underlying categories by
\cite[\href{https://stacks.math.columbia.edu/tag/003Z}{Tag 003Z}]{stacks-project}.
\end{proof}

\subsection{The Correspondence}
\label{subsec:correspondence}

\begin{defn}[Extended Harmonic Bundle Families]\label{defn:ext-harmonic-family}
In the context of \cref{thm:harmonic-Dol-dR-prest}, denote by $\sM^{ext}_\sH(X)$, the prestack
forming the middle vertex of diagram \ref{eqn:harmonic-Dol-dR-prest}.
We will call the objects of $\sM^{ext}_\sH(X)$ extended smooth families of harmonic bundles.
\end{defn}

\begin{defn}[Substack generated by a subcategory {\cite[Remark 3.24]{FormalAlgSt}}]
\label{defn:generated-st}
Let $\mcX \to \sC$ be a stack over a site $\sC$ with a subcategory $\mcP \hto \mcX$. Then, we define
$\ol\mcP$ to be the intersection of all substacks of $\mcX$ that contain $\mcP$.
\end{defn}

\begin{prop}\label{prop:generated-st-stackification}
Let $\mcX \to \sC$ be a stack over a site $\sC$, with a full, replete sub-prestack
$\mcP \hto \mcX \to \sC$.
Then, substack $\ol\mcP \hto \mcX \to \sC$ generated by $\mcP$ is the full subcategory
of $\mcX$ consisting of those objects $x$ that admit a cover $\set{\beta_i : x_i \to x}_{i \in I}$
in the topology on $\mcX$ induced by that of $\sC$
\cite[\href{https://stacks.math.columbia.edu/tag/06NU}{Tag 06NU}]{stacks-project},
with each $x_i$ an object of $\mcP$. In particular, $\mcP \hto \ol\mcP$ realizes $\ol{\mcP}$
as the stackification of $\mcP$.
\end{prop}
\begin{proof}
Let $\tilde\mcP$ be the subcategory described in the statement.
Consider a substack $\mcX'$ of $\mcX$ containing $\mcP$. Let $x \in \Ob{\tilde\mcP}$. Choose a cover
$\set{\beta_i : x_i \to x}_{i \in I}$ as above over a cover $\set{c_i \to c}_{i \in I}$ in $\sC$.
Let $c_{ij} := c_i \times_c c_j$ and $c_{ijk} := c_i \times_c c_j \times_c c_k$.
There are isomorphisms $\alpha_{ij} : x_i|_{c_{ij}} \cong x|_{c_{ij}} \cong x_j|_{c_{ij}}$
satisfying $\alpha_{jk}|_{c_{ijk}} \circ \alpha_{ij}|_{c_{ijk}} = \alpha_{ik}|_{c_{ijk}}$.
Since $\mcX'$ is a stack containing the objects of $\mcP$, the
$x_i$ glue to an object $x'$ of $\mcX'$, such that there are isomorphisms
$\phi_i : x'|_{c_{i}} \to x_i$ satisfying $\phi_j|_{c_{ij}} = \alpha_{ij} \circ \phi_i|_{c_{ij}}$.
By gluing the $\phi_i$, we get an isomorphism $x' \cong x$ in the ambient
stack $\mcX$. Since substacks are replete subcategories by definition, $x \in \mcX'$.
Furthermore, for any morphism $f : x \to x'$ in $\mcP$, $x, x'$ are in $\mcX'$, as we have just
shown, and since $\mcX'$, being a substack, is full, it must contain $f$.
Thus, every substack of $\mcX$ containing $\mcP$ contains $\tilde{\mcP}$: that is, it is the full
subcategory whose object set is the intersection of the object sets of all substacks
containing $\mcP$. This is precisely $\ol\mcP$.

Given the above description of $\ol{P}$, the inclusion map $\mcP \hto \ol\mcP$ satisfies the
conditions of \cite[\href{https://stacks.math.columbia.edu/tag/02ZP}{Tag 02ZP}]{stacks-project}, and
hence, realizes $\ol\mcP$ as the stackification of $\mcP$.
\end{proof}

\begin{defn}[Higgs Bundle Families Locally Extended from Harmonic Bundles]
\label{defn:harmonic-Dol}
In the context of \cref{thm:harmonic-Dol-dR-prest}, denote by $\sM^{\sH, pre}_{Dol}(X)$
the essential images of the left morphism of prestacks of diagram \ref{eqn:harmonic-Dol-dR-prest}.
Let $\sM^\sH_{Dol}(X)$ denote the substack of $\sM_{Dol}(X)$ generated by $\sM^{\sH, pre}_{Dol}(X)$.
By \cref{prop:generated-st-stackification}, the objects of $\sM^\sH_{Dol}(X)$ are smooth families of
Higgs bundles on $X$ parametrized by $U$ for which there is an open cover $U = \bigcup_i U_i$
such that restricted to $U_i$, the family can be obtained by a finite iterated extension of
families of Higgs bundles coming from families of harmonic bundles parametrized by $U_i$.
Hence, we will call the objects of $\sM^\sH_{Dol}(X)$ smooth families of Higgs bundles locally
extended from harmonic bundles.
\end{defn}

\begin{thm}\label{thm:harmonic-Dol}
For a compact K\"ahler manifold $X$,
$\sM^\sH_{Dol}(X)(\pt)$ is the full subcategory of Higgs bundles on $X$ that are semistable
with vanishing projections of the first and second Chern classes along the K\"ahler form of $X$.
\end{thm}
\begin{proof}
$\sM^\sH_{Dol}(X)(\pt)$ is the full subcategory of families of Higgs
bundles on $X$ parametrized by $\pt$ such that $\pt$ has an open cover on which the family is
a finite iterated extension of Higgs bundles associated to harmonic bundles. That is, it is the full
subcategory of Higgs bundles on $X$ that are finite iterated extensions of Higgs bundles associated
to harmonic bundles. Higgs bundles associated to harmonic bundles are precisely the polystable Higgs
bundles with the given condition on Chern classes: \cite[Theorem 1]{ConsVHS} is the stable
case, and the polystable case extends by using direct sum harmonic metrics for the stable pieces.
Extensions of polystable bundles are precisely the semistable Higgs bundles on $X$ with the given
condition on Chern classes.
\end{proof}

\begin{thm}\label{thm:harmonic-Dol-incl-rep}
For a compact K\"ahler manifold $X$, the inclusion $\sM^\sH_{Dol}(X) \to \sM_{Dol}(X)$ is
representable by diffeological spaces. In particular, $\sM^\sH_{Dol}(X)$ is a diffeological stack.
\end{thm}
\begin{proof}
Let $U$ be a smooth manifold, and $U \to \sM_{Dol}(X)$ be a map corresponding to a smooth family
$(E, D''_E)$ of Higgs bundles on $X$ parametrized by $U$. Then, by
\cite[\href{https://stacks.math.columbia.edu/tag/0040}{Tag 0040}]{stacks-project},
the fibre product $P_E := U \times_{\sM_{Dol}(X)} \sM^\sH_{Dol}(X)$ has the following concrete
description:
\begin{itemize}
\item Objects are tuples $(v, F, D''_F, \alpha)$, where:
    \begin{itemize}
    \item $v : V \to U$ is a smooth map from a smooth manifold $V$,
    \item $(F, D''_F)$ is a smooth family of Higgs bundles on $X$ parametrized by $V$,
    that is locally extended from harmonic bundles, and
    \item $\alpha : v^<(E, D''_E) \to (F, D''_F)$ is an isomorphism of smooth families of Higgs
    bundles on $X$ parametrized by $V$.
    \end{itemize}
\item Morphisms $(w, G, D''_G, \beta) \to (v, F, D''_F, \alpha)$ are tuples $(f, a)$, where:
    \begin{itemize}
    \item $f : W \to V$ is a smooth map commuting with the maps to $U$, and
    \item $a : f^<(F, D''_F) \to (G, D''_G)$ is an isomorphism of smooth families of flat bundles
    on $X$ parametrized by $X$, such that the following diagram commutes:
    \[\begin{tikzcd}
    f^<v^<(E, D''_E) \ar[r, "f^<\alpha"] \ar[d, "\cong" left] & f^<(F, D''_F) \ar[d, "a"] \\
    w^<(E, D''_E) \ar[r, "\beta" below] & G
    \end{tikzcd}\]
    where the left vertical arrow is the canonical isomorphism from \cref{prop:pullback-comp-nat}.
    \end{itemize}
\end{itemize}
Similar to the proof of \cref{thm:mod-st-Dol-dR-diff}, we can see that if $(f, a)$ and $(f, b)$
are two morphisms, then $a = b$.

We will show that this is equivalent as a stack to the stack associated to a diffeological space
$Q_E$ defined as follows.
The underlying set of $Q_E$ is $U$, and the diffeology on $Q_E$ is defined by the formula:
\[
\sD_{Q_E}(V) = \set[v \in \Cinf(V, U)]{v^<(E, D''_E) \text{ is locally extended from harmonic
                                                            bundles}}
\]
with the action on morphisms given by composition.
We need to verify that this is indeed a diffeology. First, if $f : W \to V$ is a smooth map over
$U$, then $(v \circ f)^<(E, D''_E) \cong f^<v^<(E, D''_E)$ is locally extended from harmonic
bundles because $\sM^\sH_{Dol}$ is a stack, and in particular, a prestack, so that
the condition of being locally extended from harmonic bundles is preserved under pullback.
Thus, $\sD_{Q_E}$ is a presheaf.
Suppose, we have an open cover $V = \bigcup_i V_i$ and elements $v_i \in \sD_{Q_E}(V_i)$ such that
$v_i|_{V_i \cap V_j} = v_j|_{V_i \cap V_j}$. Then, they glue to a unique smooth map $v : V \to U$.
The fact that $v_i^<(E, D''_E) \cong (v^<(E, D''_E))|_{V_i}$ is locally extended from harmonic
bundles for each $i$ implies that $v^<(E, D''_E)$ is locally extended from harmonic bundles, by
considering the cover of $V$ given by the union of the covers of the $V_i$ by open sets over which
$(E, D''_E)$ is a finite iterated extension of harmonic bundles.
This shows that $v \in \sD_{Q_E}(V)$. Therefore, $\sD_{Q_E}$ is a subsheaf of $\Cinf(-, U)$
which is a subsheaf of $\Set(-, U)$. Thus, $\sD_{Q_E}$ is a diffeology on the underlying set of $U$.

The equivalence $Q_E \to P_E$ is then given by sending $v \in \sD_{Q_E}(V)$ to
$(v, v^<E, v^<D''_E, \id)$, and $f : W \to V$ over $U$ to $(f, a)$, where $a$ is the canonical
isomorphism $(v \circ f)^<(E, D''_E) \to[\cong] f^<v^<(E, D''_E)$. The functoriality of this mapping
is immediate. The fact that it is fully faithful follows from the last statement of the first
paragraph. That fact that it is essentially surjective follows from the observation that every
object $(v, F, D''_F, \alpha)$ in $P_E$ is isomorphic to $(v, v^<E, v^<D''_E, \id)$ via
$(\id_V, \alpha \circ a)$, where $a$ is the canonical isomorphism
$\id_V^<v^<(E, D''_E) \to[\cong] v^<(E, D''_E)$.

We may now apply \cref{prop:diffeological-rep-morphism} to conclude that $\sM^\sH_{Dol}(X)$ is
a diffeological stack.
\end{proof}

\begin{rmk}\label{rmk:harmonic-Dol-sst}
In light of \cref{thm:harmonic-Dol} and \cref{thm:harmonic-Dol-incl-rep}, $\sM^\sH_{Dol}(X)$ is a
diffeological moduli stack parametrizing semistable Higgs bundles.
However, it is not clear to us yet whether every smooth family $(E, D'')$ of Higgs
bundles on $X$ whose pullbacks $(E_u, D''_u)$ to $\set{u} \times X \cong X$ is semistable is an
object of $\sM^\sH_{Dol}(X)$. Let us call such families semistable.
As a result, we cannot yet make any conclusions as to the relation
between $\sM^\sH_{Dol}(X)$ and the coarse moduli space of semistable Higgs bundles. We expect there
to be a mapping into the coarse moduli space (viewed as the Grothendieck constructions of
its functor of points on the category of smooth manifolds) from $\sM^\sH_{Dol}(X)$.
Our investigations suggest that in order for $\sM^\sH_{Dol}(X)$ to have semistable families as
objects, the following statement needs to be true: if $(E_u, D''_u)$ above is semistable, there is
a neighbourhood $N$ of $u$ in $U$, such that $(E, D'')|_{N}$ comes from a finite iterated extension
smooth families of harmonic bundles, which may or may not be true.

To see this, consider an alternative diffeological space $Q'_E$ in place of $Q_E$
of \cref{thm:harmonic-Dol-incl-rep} defined as follows. It is the subset of $U$ consisting
of those points $u \in U$ that have a neighbourhood $N_u$ on which $(E, D''_E)|_{N_u}$ is
locally extended from harmonic bundles. This is an open submanifold of $U$. The stack associated to
it admits a fully faithful morphism to $P_E$. However, it is unclear whether this map is essentially
surjective. One way to prove essential surjectivity would be to show that if $v^<(E, D''_E)$ is
locally generated by harmonic bundles, then $v$ factors through $Q'_E$, which would follow from
the above claim. Similar considerations apply for families of stable and polystable Higgs bundles.

Hence, we cannot yet make any solid claims about how $\sM^\sH_{Dol}(X)$ relates to all semistable
families or the coarse moduli space. However, there is a notion of diffeological coarse moduli space
associated to a diffeological space \cite{WW24} and the one associated to $\sM^\sH_{Dol}(X)$ is
likely to be related to the usual coarse moduli space.
\end{rmk}

\begin{defn}[Flat Bundle Families Locally Extended from Harmonic Bundles]
\label{defn:harmonic-dR}
In the context of \cref{thm:harmonic-Dol-dR-prest}, denote by $\sM^{\sH, pre}_{dR}(X)$
the essential images of the left morphism of prestacks of diagram \ref{eqn:harmonic-Dol-dR-prest}.
Let $\sM^\sH_{dR}(X)$ denote the substack of $\sM_{dR}(X)$ generated by $\sM^{\sH, pre}_{dR}(X)$.
By \cref{prop:generated-st-stackification}, objects of $\sM^\sH_{dR}(X)$ are smooth families of
flat bundles on $X$ parametrized by $U$ for which there is an open cover $U = \bigcup_i U_i$, such
that restricted to $U_i$, the family can be obtained by a finite iterated extension of families
of flat bundles coming from families of harmonic bundles parametrized by $U_i$.
Hence, we will call the objects of $\sM^\sH_{dR}(X)$ smooth families of flat bundles locally
extended from harmonic bundles.
\end{defn}

\begin{thm}\label{thm:harmonic-dR}
For a compact K\"ahler manifold $X$,
$\sM^\sH_{dR}(X)(\pt)$ is the category of flat bundles on $X$.
\end{thm}
\begin{proof}
Similar to \cref{thm:harmonic-Dol}, by unwrapping the definitions, we can see that
$\sM^\sH_{dR}(X)(\pt)$ is the full subcategory of the category of flat bundles on $X$
that are extensions of flat bundles associated to harmonic bundles. A flat bundle coming from
a harmonic bundle is semisimple \cite{FlatGBunMetrics}.
An arbitrary flat bundle is a finite iterated extension of semisimple flat bundles.
\end{proof}

\begin{thm}\label{thm:harmonic-dR-incl-rep}
For a compact K\"ahler manifold $X$, the inclusion $\sM^\sH_{dR}(X) \to \sM_{dR}(X)$ is
representable by diffeological spaces. In particular, $\sM^\sH_{dR}(X)$ is a diffeological stack.
\end{thm}
\begin{proof}
Similar to \cref{thm:harmonic-Dol-incl-rep}.
\end{proof}

\begin{rmk}\label{rmk:harmonic-dR}
By \cref{thm:harmonic-dR}, $\sM^\sH_{dR}(X)$ is a moduli stack parametrizing flat bundles. Similarly
to the Dolbeault side, it is not yet clear to us that an arbitrary family of flat bundles is locally
extended from harmonic bundles. If this holds, then we have an equivalence
$\sM^\sH_{dR}(X) \simeq \sM_{dR}(X)$. Again, we cannot yet make any statements about the relation
of $\sM^\sH_{dR}(X)$ with the coarse moduli space of flat bundles, but we do expect
(the Grothendieck construction of the functor of points of) the coarse moduli space to have a
mapping from $\sM^\sH_{dR}(X)$.
\end{rmk}

We can now prove a diffeological version of the non-Abelian Hodge correspondence.

\begin{prop}\label{prop:generated-st-equiv}
Let $\mcX, \mcY$ be stacks over a site $\sC$, and let $\mcP \subset \mcX, \mcQ \subset \mcY$
be full, replete sub-prestacks with an equivalence of prestacks $F : \mcP \to \mcQ$. Let $\ol\mcP$
and $\ol\mcQ$ be the substacks of $\mcX$ and $\mcY$ generated by $\mcP$ and $\mcQ$ respectively.
Then, $F$ extends to an equivalence of stacks $\ol{F} : \ol\mcP \to \ol\mcQ$.
\end{prop}
\begin{proof}
Given \cref{prop:generated-st-stackification}, the functoriality of stackification, a consequence
of the universal property of stackification, yields the morphism of stacks $\ol{F}$ and shows that
it is an equivalence.
\end{proof}

\begin{rmk}
The functor $\ol{F}$ can also be obtained more directly.
Observe that objects and morphisms in $\ol\mcP$ give descent data in
$\mcP$ which are carried to descent data in $\mcQ$ by $F$. These descent data then glue to objects
and morphisms in $\ol\mcQ$. This defines the mapping of objects and morphisms for $\ol{F}$. That
this is a morphism of stacks is a straightforward verification using descent data with refinements
of covers as needed.
Then, choose an inverse $G$ of $F$ and obtain a morphism of stacks $\ol{G} : \ol\mcQ \to \ol\mcP$
similarly. Another verification using descent data shows that $\ol{F}$ and $\ol{G}$ are inverses
as morphisms of prestacks. This kind of description might be helpful in understanding what
$\ol{F}$ does more concretely.
\end{rmk}

\begin{thm}[Non-Abelian Hodge Correspondence]\label{thm:NAH}
For a compact K\"ahler manifold $X$, we have an equivalence of stacks:
\[
\sM^\sH_{Dol}(X) \simeq \sM^\sH_{dR}(X)
\]
\end{thm}
\begin{proof}
We have an equivalence of prestacks $\sM^{\sH, pre}_{Dol}(X) \simeq \sM^{\sH, pre}_{dR}(X)$
by \cref{thm:harmonic-Dol-dR-prest}. This induces an equivalence of substacks generated by these
prestacks by \cref{prop:generated-st-equiv}, which are exactly the stacks in the statement.
\end{proof}

We will show with an example that the mapping of \cref{thm:NAH} extends to at least one family of
semistable objects that the mapping of coarse moduli spaces does not extend to.
Of course, this statement is a rough one, because the $U$--points of the coarse moduli spaces are
not families.

\begin{exm}\label{exm:sst-family}
Consider an elliptic curve $X$. Let $U = (\bC \setminus \set{0}) \times \bR$ and consider the
trivial bundle
$U \times X \times \bC^2 \to U \times X$. Then, we have a smooth family of Higgs bundles
on $X$ parametrized by $U$ whose Higgs field, in a chart, is given by the matrix:
\[
A(a, t) = \begin{bmatrix}
          0 & dz \\
          0 & at dz
          \end{bmatrix}
\]
This is the same family in the counterexample given in \cite[39]{ModRepFunGrpII} to show that
the non-Abelian Hodge correspondence for coarse moduli spaces does not extend continuously to the
semistable locus.
However, this family is globally an extension of the form:
\[
(U \times X \times \bC, \oprt_{X/U}) \to
(U \times X \times \bC^2, (\oprt_{X/U} \oplus \oprt_{X/U}) + A) \to
(U \times X \times \bC, \oprt_{X/U} + [at dz])
\]
To see this, use the description of extensions given in
\cite[Remark preceding Lemma 3.5]{HiggsLocSys}. The two rank $1$ families are families of stable
Higgs bundles, and hence come from harmonic bundles with the constant metric being a harmonic metric
for both. Hence, the extended family is a family of semistable Higgs bundles, and is an object of
$\sM^\sH_{Dol}(X)$.
This is not simply a direct sum of two rank $1$ families: in particular, it is not a direct sum
and hence not semistable at $t = 0$.
By the mapping of \cref{thm:NAH}, this maps to a family of flat bundles.
In fact, the computation in \cite[39]{ModRepFunGrpII}, along with the fact that morphisms of stacks
respect pullbacks, shows that it maps to the family of flat bundles whose underlying vector bundle
is again the trivial bundle $U \times X \times \bC^2$ and whose connection matrix is:
\[
B(a, t) = \begin{bmatrix}
          0 & dz + \frac{\ol{a}}{a} d\bar{z} \\
          0 & at dz + \ol{at} d\bar{z}
          \end{bmatrix}
\]
In the moduli space of semistable Higgs bundles $\lim_{t \to 0} A(a, t) = A(0, 0)$ for all
$a \in \bC$, but in the moduli space of flat bundles, the $\lim_{t \to 0} B(a, t)$ is dependent on
$a$. This prevents the non-Abelian Hodge correspondence between moduli spaces from being continuous
when extended to the semistable locus.
However, it causes no problems for the existence of the mapping of stacks realizing the
equivalence of \cref{thm:NAH}.
\end{exm}

\begin{rmk}
Let $(E, D''_E)$ be a smooth family of Higgs bundles on $X$ parametrized by $U$. We say that the
family is stable, semistable and polystable respectively if for each $u$ in $U$, the pullback of the
family to the slice $\set{u} \times X$ is so.
Let $\sM^{st}_{Dol}(X) \subset \sM^{pst}_{Dol}(X) \subset \sM^{sst}_{Dol} \subset \sM_{Dol}(X)$
denote the full subprestacks consisting of stable, polystable and semistable families.
Define simple and semisimple families of flat bundles similarly, and let
$\sM^{si}_{dR}(X) \subset \sM^{ssi}_{dR}(X) \subset \sM_{dR}(X)$ be the full subprestacks consiting
of simple and semisimple families respectively.
In light of \cref{rmk:harmonic-Dol-sst} and \cref{exm:sst-family}, we may summarize the version
of the non-Abelian Hodge correspondence developed in this paper in the following diagram:
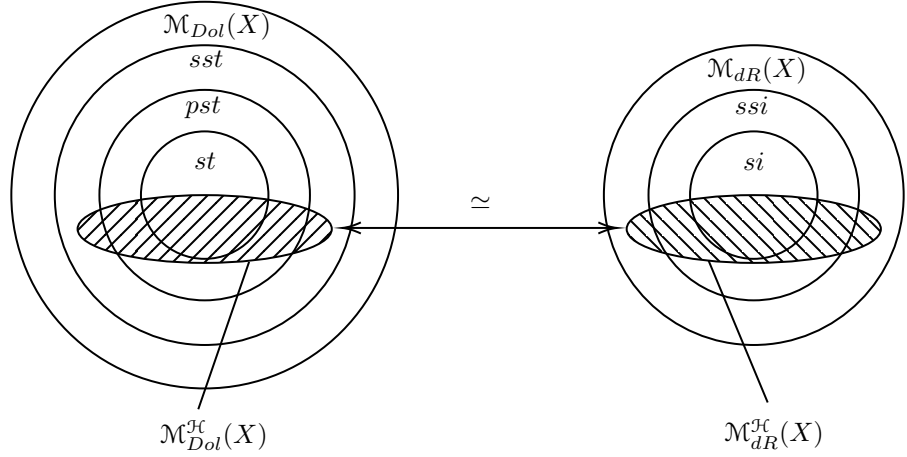
\begin{figure}[H]
 
\tikzset{
pattern size/.store in=\mcSize, 
pattern size = 5pt,
pattern thickness/.store in=\mcThickness, 
pattern thickness = 0.3pt,
pattern radius/.store in=\mcRadius, 
pattern radius = 1pt}
\makeatletter
\pgfutil@ifundefined{pgf@pattern@name@_bley8yjc2}{
\pgfdeclarepatternformonly[\mcThickness,\mcSize]{_bley8yjc2}
{\pgfqpoint{0pt}{0pt}}
{\pgfpoint{\mcSize+\mcThickness}{\mcSize+\mcThickness}}
{\pgfpoint{\mcSize}{\mcSize}}
{
\pgfsetcolor{\tikz@pattern@color}
\pgfsetlinewidth{\mcThickness}
\pgfpathmoveto{\pgfqpoint{0pt}{0pt}}
\pgfpathlineto{\pgfpoint{\mcSize+\mcThickness}{\mcSize+\mcThickness}}
\pgfusepath{stroke}
}}
\makeatother

 
\tikzset{
pattern size/.store in=\mcSize, 
pattern size = 5pt,
pattern thickness/.store in=\mcThickness, 
pattern thickness = 0.3pt,
pattern radius/.store in=\mcRadius, 
pattern radius = 1pt}
\makeatletter
\pgfutil@ifundefined{pgf@pattern@name@_597zx4bjc}{
\pgfdeclarepatternformonly[\mcThickness,\mcSize]{_597zx4bjc}
{\pgfqpoint{0pt}{-\mcThickness}}
{\pgfpoint{\mcSize}{\mcSize}}
{\pgfpoint{\mcSize}{\mcSize}}
{
\pgfsetcolor{\tikz@pattern@color}
\pgfsetlinewidth{\mcThickness}
\pgfpathmoveto{\pgfqpoint{0pt}{\mcSize}}
\pgfpathlineto{\pgfpoint{\mcSize+\mcThickness}{-\mcThickness}}
\pgfusepath{stroke}
}}
\makeatother
\tikzset{every picture/.style={line width=0.75pt}} 

\begin{tikzpicture}[x=0.75pt,y=0.75pt,yscale=-0.85,xscale=0.85, font=\small]

\draw   (6.94,118.38) .. controls (6.94,55.36) and (58.02,4.28) .. (121.03,4.28) .. controls
(184.04,4.28) and (235.13,55.36) .. (235.13,118.38) .. controls (235.13,181.39) and (184.04,232.47)
.. (121.03,232.47) .. controls (58.02,232.47) and (6.94,181.39) .. (6.94,118.38) -- cycle ;
\draw   (32.58,118.38) .. controls (32.58,69.53) and (72.18,29.93) .. (121.03,29.93) .. controls
(169.88,29.93) and (209.48,69.53) .. (209.48,118.38) .. controls (209.48,167.22) and (169.88,206.82)
.. (121.03,206.82) .. controls (72.18,206.82) and (32.58,167.22) .. (32.58,118.38) -- cycle ;
\draw   (58.92,118.38) .. controls (58.92,84.07) and (86.73,56.27) .. (121.03,56.27) .. controls
(155.33,56.27) and (183.14,84.07) .. (183.14,118.38) .. controls (183.14,152.68) and (155.33,180.48)
.. (121.03,180.48) .. controls (86.73,180.48) and (58.92,152.68) .. (58.92,118.38) -- cycle ;
\draw   (83.42,118.38) .. controls (83.42,97.6) and (100.26,80.77) .. (121.03,80.77) .. controls
(141.8,80.77) and (158.64,97.6) .. (158.64,118.38) .. controls (158.64,139.15) and (141.8,155.98) ..
(121.03,155.98) .. controls (100.26,155.98) and (83.42,139.15) .. (83.42,118.38) -- cycle ;
\draw  [pattern=_bley8yjc2,pattern size=6pt,pattern thickness=0.75pt,pattern radius=0pt, pattern
color={rgb, 255:red, 0; green, 0; blue, 0}] (46.03,138.38) .. controls (46.03,127.33) and
(79.61,118.38) .. (121.03,118.38) .. controls (162.45,118.38) and (196.03,127.33) .. (196.03,138.38)
.. controls (196.03,149.42) and (162.45,158.38) .. (121.03,158.38) .. controls (79.61,158.38) and
(46.03,149.42) .. (46.03,138.38) -- cycle ;
\draw   (356.58,118.38) .. controls (356.58,69.53) and (396.18,29.93) .. (445.03,29.93) .. controls
(493.88,29.93) and (533.48,69.53) .. (533.48,118.38) .. controls (533.48,167.22) and (493.88,206.82)
.. (445.03,206.82) .. controls (396.18,206.82) and (356.58,167.22) .. (356.58,118.38) -- cycle ;
\draw   (382.92,118.38) .. controls (382.92,84.07) and (410.73,56.27) .. (445.03,56.27) .. controls
(479.33,56.27) and (507.14,84.07) .. (507.14,118.38) .. controls (507.14,152.68) and (479.33,180.48)
.. (445.03,180.48) .. controls (410.73,180.48) and (382.92,152.68) .. (382.92,118.38) -- cycle ;
\draw   (407.42,118.38) .. controls (407.42,97.6) and (424.26,80.77) .. (445.03,80.77) .. controls
(465.8,80.77) and (482.64,97.6) .. (482.64,118.38) .. controls (482.64,139.15) and (465.8,155.98) ..
(445.03,155.98) .. controls (424.26,155.98) and (407.42,139.15) .. (407.42,118.38) -- cycle ;
\draw  [pattern=_597zx4bjc,pattern size=6pt,pattern thickness=0.75pt,pattern radius=0pt, pattern
color={rgb, 255:red, 0; green, 0; blue, 0}] (370.03,138.38) .. controls (370.03,127.33) and
(403.61,118.38) .. (445.03,118.38) .. controls (486.45,118.38) and (520.03,127.33) ..
(520.03,138.38) .. controls (520.03,149.42) and (486.45,158.38) .. (445.03,158.38) .. controls
(403.61,158.38) and (370.03,149.42) .. (370.03,138.38) -- cycle ;
\draw    (204.25,137.75) -- (362.25,137.75) ;
\draw [shift={(364.25,137.75)}, rotate = 180] [color={rgb, 255:red, 0; green, 0; blue, 0 }  ][line
width=0.75]    (10.93,-3.29) .. controls (6.95,-1.4) and (3.31,-0.3) .. (0,0) .. controls (3.31,0.3)
and (6.95,1.4) .. (10.93,3.29)   ;
\draw [shift={(202.25,137.75)}, rotate = 0] [color={rgb, 255:red, 0; green, 0; blue, 0 }  ][line
width=0.75]    (10.93,-3.29) .. controls (6.95,-1.4) and (3.31,-0.3) .. (0,0) .. controls (3.31,0.3)
and (6.95,1.4) .. (10.93,3.29)   ;
\draw    (147.25,157.75) -- (117.25,244.75) ;
\draw    (418.25,156.75) -- (453.25,240.75) ;

\draw (276,117.4) node [anchor=north west][inner sep=0.75pt]    {$\simeq $};

\draw (93,249.4) node [anchor=north west][inner sep=0.75pt]    {$\sM^\sH_{Dol}(X)$};
\draw (110,31.4) node [anchor=north west][inner sep=0.75pt]    {$sst$};
\draw (108,58.4) node [anchor=north west][inner sep=0.75pt]    {$pst$};
\draw (113,91.4) node [anchor=north west][inner sep=0.75pt]    {$st$};
\draw (95,9.4) node [anchor=north west][inner sep=0.75pt]    {$\sM_{Dol}(X)$};
\draw (432,58.4) node [anchor=north west][inner sep=0.75pt]    {$ssi$};
\draw (437,91.4) node [anchor=north west][inner sep=0.75pt]    {$si$};
\draw (416,35.4) node [anchor=north west][inner sep=0.75pt]    {$\sM_{dR}(X)$};
\draw (426,248.4) node [anchor=north west][inner sep=0.75pt]    {$\sM^{\sH}_{dR}(X)$};

\end{tikzpicture}
\end{figure}
We have drawn $\sM^\sH_{Dol}(X)$ as intersecting the stacks of stable, polystable and semistable
families, because it is not yet clear to us exactly which stable, polystable or semistable families
are locally extended from harmonic bundles (\cref{rmk:harmonic-Dol-sst}).
However, recall also that every individual semistable Higgs bundle --- that is, a family
parametrized by the point --- is locally extended from harmonic bundles and is hence in
$\sM^\sH_{Dol}(X)$. Similarly, every individual flat bundle
is an object of $\sM^\sH_{dR}(X)$ but it is not yet clear exactly which simple or semisimple
families are included. Hence, we conclude that \cref{thm:NAH} is, in fact, a common extension
of the two original versions of the non-Abelian Hodge correspondece: one between coarse moduli
spaces and the other between the categories of semistable Higgs bundles (with the condition on Chern
classes, of course) and flat bundles.
\end{rmk}

\subsection{Questions for Further Study}
\label{subsec:questions}

\begin{qsn}
Suppose we have a smooth family $(E, D''_E)$ of Higgs bundles on $X$ parametrized by $U$ such that
$(E_u, D''_{E, u})$, for each $u \in U$, is polystable: that is, a polystable family.
Of course, we also assume the usual condition on Chern classes is satisfied.
We can choose any Hermitian metric $K$ on $E$ and consider operators $\prt_{K}$ and
$D'_{E, K}$ as in \cref{defn:harmonic-op}. Then, we can ask: is $D_{E, K} = D'_{E, K} + D''_E$
a smooth family of flat connections on $E$? For each $u \in U$, the pair $(E_u, D_{K, E, u})$
is a flat bundle precisely if the restriction $K_u$ of $K$ to $E_u$ is a harmonic metric.
Since we have a polystable family of Higgs bundles, there is such a harmonic metric $H_u$ on $E_u$
for each $u \in U$. However, are these metrics smooth in $u$? If we can find such a harmonic
metric for each polystable family, then we can conclude that that the moduli stack of polystable
Higgs bundles is completely contained in $\sM^\sH_{Dol}(X)$: that is, every polystable family is
accounted for. A similar statement applies for $\sM^\sH_{dR}(X)$ and families of simple flat
bundles.
\end{qsn}

\begin{qsn}
Suppose, now, we have a smooth family $(E, D''_E)$ as before but such that each
$(E_u, D''_{E, u})$ is semistable (polystable).
Each such slice is then a finite iterated extension of harmonic bundles.
However, does this imply, for each $u \in U$, we can find a neighbourhood $N_u$ over each the family
is a finite iterated extension harmonic bundle families?
If this is the case, then $\sM^{\sH}_{Dol}(X)$ contains all semistable (polystable) bundles.
A similar question can be asked about families of flat bundles with semistable or polystable
replaced by no conditions or semisimple respectively. In this case, $\sM^\sH_{dR}(X)$ would contain
all families of flat bundles. If both of these hold, then we have full extension of the usual
non-Abelian Hodge correspondence to the setting of diffeological moduli stacks.
\end{qsn}

\begin{qsn}
Recall the failure of fibrewise quasi-fullness of the forgetful map $\mcM_\sH(X) \to \mcM_{Dol}(X)$
of $\dg$--prestacks from \cref{rmk:harmonic-bun-incl-not-quasi-full}. Does this have anything to do
with which families of stable, polystable or semistable Higgs bundles are locally extended from
harmonic bundles?
\end{qsn}

\begin{qsn}
Instead of smooth families, we can consider $C^d$ families of Higgs bundles and flat connections for
$d = 0, 1, 2, \dots$. To make this precise, we could consider a $C^d$ manifold $U$ and a topological
vector bundle $E \to U \times X$, whose transition functions are smooth in the $X$ direction but
$C^d$ in the $U$ direction.
We then consider sheaves of continuous sections that are smooth in the $X$ direction
but $C^d$ in the $U$--direction. Denote such sheaves as $E^{sec, d}$.
Then, define $\sA^{k}_{X/U, d}$ to be $((\pi_X^*T^*X)^{\wedge k})^{sec, d}$:
that is, the sheaf of continuous differential $k$--forms with
components only in the $X$ direction, smooth in the $X$ direction and $C^d$ in the $U$ direction.
We can also define the sheaves $\sA^{p, q}_{X/U, d}$ of
relative $(p, q)$--forms smooth in the $X$ direction and $C^d$ in the $U$ direction similarly.
Then, $E^{sec, d}, \sA^{k}_{X/U, d}, \sA^{p, q}_{X/U, d}$ are finite rank
locally free modules over $\sA^0_{X/U, d}$. Then, we expect the theory developed in
\cref{sec:diff-conn} and \cref{sec:mod-st-conn} to apply without much modification,
and provide moduli stacks $\sM_{Dol, d}(X)$ and $\sM_{dR, d}(X)$ of $C^d$ families of Higgs bundles
and flat connections. Locally, these families are described by matrices of $1$--forms smooth in the
$X$ direction and $C^d$ in the $U$ direction, with the usual relation of flatness.
Not that since the partial exterior differentials and Dolbeault operators in this setting do not
derive in the $U$ direction, we do not have any issues with running out of derivatives.
The theory in \cref{sec:harmonic} can also be adapted without much difficulty to provide us with a
$\bC\dg$--prestack $\mcM_{\sH, d}(X)$ of $C^d$ families of harmonic bundles: recall that
harmonicity is a slicewise criterion. Similarly, the results of \cref{sec:NAH} also hold without
much modification to give us stacks $\sM^\sH_{Dol, d}(X)$ and $\sM^\sH_{dR, d}(X)$ of $C^d$ families
of Higgs bundles that, locally, are finite iterated extensions of $C^d$ families of harmonic
bundles. The same arguments also provide an equivalence of these stacks. In this context, it is
interesting to wonder: which semistable $C^d$ families of Higgs bundles and which semisimple $C^d$
families of flat bundles are contained in these respective stacks, and thus correspond under the
equivalence? At this stage, it is also interesting to consider families parametrized by
$\Cinf$--schemes.
\end{qsn}

\begin{qsn}
To what extent can the techniques of this paper be adapted to the context of topological stacks?
\end{qsn}

\printbibliography

\end{document}